\UseRawInputEncoding
\documentclass[final,hidelinks,onefignum,onetabnum]{siamart251216}

\usepackage[T1]{fontenc}
\usepackage{amsfonts,amsmath,amssymb}
\usepackage{bm}
\usepackage{mathtools}
\usepackage{graphicx}
\usepackage{booktabs}
\usepackage{longtable}
\usepackage{array}
\usepackage{algorithm}
\usepackage{algorithmic}
\usepackage[acronym,nomain]{glossaries}
\setacronymstyle{long-short}
\glsdisablehyper

\newacronym{ode}{ODE}{ordinary differential equation}
\newacronym{pde}{PDE}{partial differential equation}
\newacronym{roa}{ROA}{region of attraction}
\newacronym{sos}{SOS}{sum-of-squares}
\newacronym{dsos}{DSOS}{diagonally dominant sum-of-squares}
\newacronym{sdsos}{SDSOS}{scaled diagonally dominant sum-of-squares}
\newacronym{sdp}{SDP}{semidefinite program}
\newacronym{lmi}{LMI}{linear matrix inequality}
\newacronym{lp}{LP}{linear program}
\newacronym{socp}{SOCP}{second-order cone program}
\newacronym{sindy}{SINDy}{sparse identification of nonlinear dynamics}

\newsiamthm{assumption}{Assumption}
\newsiamthm{problem}{Problem}
\newsiamremark{remark}{Remark}
\newsiamremark{example}{Example}

\headers{Optimizing Lyapunov Certificates via Dissipative Quadratization}{Y. Cai and G. Zardini}

\title{Optimizing Lyapunov Certificates via Stability-Preserving Quadratization for Polynomial Systems}

\providecommand{\thanksmark}[1]{\footnotemark[#1]}
\author{Yubo Cai\thanks{Laboratory for Information \& Decision Systems,
Massachusetts Institute of Technology, Cambridge, MA
(\email{yubocai@mit.edu}; \email{gzardini@mit.edu}).
\funding{Y. Cai was partially supported by the MathWorks Research Fellowship. Cai and Zardini were sponsored by the Air Force Office of Scientific Research under agreement number FA9550261B038, and by the Defense Advanced Research Projects Agency (DARPA) under Award No. D25AC00373.}}\and Gioele Zardini\thanksmark{1}}

\newcommand{\R}{\mathbb{R}}
\newcommand{\Sset}{\mathbb{S}}
\newcommand{\M}{\mathcal{M}}
\newcommand{\norm}[1]{\lVert #1 \rVert}
\newcommand{\kron}{\otimes}
\newcommand{\Real}{\operatorname{Re}}
\DeclareMathOperator{\diag}{diag}
\newcommand{\eig}{\operatorname{eig}}
\newcommand{\Aof}{A(\Lambda)}
\newcommand{\Hof}{H(\Lambda,\mu)}
\newcommand{\Dof}{D(\Lambda)}

\newcommand{\vecx}{\mathbf{x}}
\newcommand{\vecy}{\mathbf{y}}

\newcommand{\vecw}{\mathbf{w}}

\newcommand{\vecg}{\mathbf{g}}
\newcommand{\vech}{\mathbf{h}}
\newcommand{\vecq}{\mathbf{q}}

\ifpdf
\hypersetup{
  pdftitle={Optimizing Lyapunov Certificates via Stability-Preserving Quadratization for Polynomial Systems},
  pdfauthor={Yubo Cai and Gioele Zardini}
}
\fi

\usepackage[numbers,sort&compress]{natbib}
\usepackage{enumitem}
\usepackage{microtype}
\usepackage{xcolor}
\usepackage{tikz}
\usepackage{titletoc}
\usepackage{needspace}
\usetikzlibrary{arrows.meta,positioning,calc}
\AddToHook{cmd/append/before}{\Needspace{4\baselineskip}}

\titlecontents{appsection}[0pt]
  {\addvspace{.45em}\small\bfseries}
  {}{}
  {\titlerule*[.6pc]{.}\contentspage}
\titlecontents{appsubsection}[3.4em]
  {\small}
  {\contentslabel{2.6em}}{}
  {\titlerule*[.6pc]{.}\contentspage}

\crefname{assumption}{Assumption}{Assumptions}
\crefname{problem}{Problem}{Problems}

\begin{document}
\maketitle

\begin{abstract}
Region-of-attraction certificates for polynomial systems become expensive as the state
dimension and polynomial degree grow, because direct sum-of-squares formulations require a
monomial basis whose size grows combinatorially. Quadratization sidesteps this cost by
representing a polynomial vector field exactly on an invariant manifold of a quadratic system,
so that a quadratic Lyapunov function can certify the region of attraction directly. Even once
the lifting map is fixed, this exactness constrains the lifted dynamics only on the manifold,
leaving freedom in how they extend off it. Stabilizer gains shape this off-manifold extension
and change the transverse dynamics, while representation gauges leave the vector field unchanged
but change its matrix representation. Both choices affect the resulting spectral-norm
certificate, yet prior work has treated them separately, fixing the gain with a feasibility
heuristic and optimizing the gauge only afterward, for that fixed pair.

We formulate optimal dissipative quadratization (\textsc{ODQ}), which designs the gain and the
gauge together for a fixed monomial lift, reference extension, and stabilizer factorization,
with the Lyapunov weight fixed to the identity. \textsc{ODQ} selects stabilizer gains within a
prescribed compact Hurwitz box and, for each candidate gain, optimizes the representation gauge
by an exact semidefinite program that attains the global minimum of the spectral-norm bound at
that gain. Residual-aware bounds turn this spectral-norm bound into a certified closed Lyapunov
sublevel set, accounting for the floating-point residual of the underlying Lyapunov solve. Under
our stated assumptions, every accumulation point of the idealized outer search is box-Clarke
stationary. A finite run instead returns the best candidate it can independently verify, and we
do not claim global optimality for the gain search.

On a planar quintic system, optimizing the gain increases the certified area by a factor of
$2.238$ over a matched zero-gain gauge. Across a benchmark of 16 heterogeneous polynomial
systems with stabilizer freedom, \textsc{ODQ} improves on both fixed-gain lifted baselines. All
36 \textsc{ODQ} runs on the relay benchmark complete, and \textsc{ODQ} is favored in all 27
repeat-level comparisons across the nine fully paired relay cases against a sum-of-squares
baseline with a fixed quadratic Lyapunov function, in both the fixed-direction proxy and
construction time. Broader comparisons with direct sum-of-squares methods remain mixed.
\end{abstract}



\section{Introduction}
\label{sec:intro}

The \gls{roa} of an asymptotically stable equilibrium is the set of initial
states from which trajectories return to it.
Exact quadratization represents a polynomial system on an invariant
lifting manifold of a quadratic system. Even with the lifting map fixed, any vector-valued
polynomial of total degree at most two that vanishes on this manifold may be added to the
lifted right-hand side \cite[Ex.~2]{cai2024dissipative}. These additions preserve the original
trajectories from compatible lifted initial conditions, while they can change transverse
stability and the resulting Lyapunov certificate. We treat this off-manifold extension as a
design variable for \gls{roa} estimation.

\gls{roa} estimates are required for post-fault screening
in power systems \cite{chiang1988stability} and safe exploration in learning-based control
\cite{berkenkamp2017safe}. A closed-form description of the \gls{roa}
is generally unavailable for nonlinear systems, and the standard
substitute is a certified inner approximation: a sublevel set of a Lyapunov function, together
with a proof that this sublevel set is contained in the \gls{roa}. For polynomial vector fields this
construction has been carried out by convex optimization for two decades, through sum-of-squares
(SOS) programming \cite{parrilo2000structured,papachristodoulou2002construction}, the $V$--$s$
iterations and shape-function enlargements built on it
\cite{jarviswloszek2003controls,tan2008stability,topcu2008local}, and cheaper relaxations of the
SOS cone \cite{ahmadi2019dsos}.

These programs work over a monomial basis whose size grows combinatorially in the state dimension
and the polynomial degree, and the effect is visible in the size of the published examples: many
illustrative \gls{roa} studies for polynomial systems use two state variables
\cite{papachristodoulou2005recast,levin1994analytical,pitarch2013closed}, occasionally three
\cite{najafi2016fast,zarei2018arc}. \emph{Quadratization} trades that degree for
dimension. Adjoining auxiliary variables $\vecy=\vecg(\vecx)$ makes the augmented state
$\vecw=(\vecx,\vecy)\in\R^N$ obey an exactly quadratic system $\dot\vecw=A\vecw+H(\vecw\kron\vecw)$. Such a lift always
exists \cite{carleman1932,kerner1981universal}, modern algorithms produce monomial lifts with
provably few auxiliary variables \cite{bychkov2021optimal}, and
for a quadratic system with a Hurwitz linear part, a quadratic
Lyapunov function yields an ellipsoidal certificate from a closed-form radius or a single
semidefinite program \cite{tesi1996stability,kramer2021stability}. The same quadratic form recurs
across otherwise unrelated settings: it is the working class of lift-and-learn
\cite{qian2020lift,kramer2019lifting} and of operator inference
\cite{peherstorfer2016data,kramer2024learning}, it is the entry point of Carleman linearization for
reachability analysis \cite{forets2021reachability}, and bimolecular mass-action
chemical kinetics is quadratic outright \cite{feinberg1987chemical}.

Quadratizations of the same
system differ in structure and in how many variables they introduce
\cite{alauddin2020quadratization}. Every
polynomial system with a dissipative equilibrium admits a \emph{dissipative} quadratization,
obtained by doubling a stabilizer gain until the lifted origin passes a Routh--Hurwitz test
\cite[Thm.~1, Alg.~2, Prop.~3]{cai2024dissipative}. There is also a representation gauge:
adding a matrix $K$ satisfying $K(\mathbf w\kron\mathbf w)=0$ leaves the quadratic vector field
unchanged but can alter the spectral-norm bound.  This gauge has been optimized for a fixed pair
$(A,H)$ \citep{kramer2021stability}.  The two choices are coupled at the certificate level, since
the stabilizer gain changes $A$ and $H$, while $A$ determines the Lyapunov metric in which the
gauge is evaluated.

We characterize the full affine space of admissible quadratic
extensions for a fixed embedding and restrict the optimization to the stabilizer family.
For a fixed monomial inner-quadratic lift, reference extension,
prescribed stabilizer factorizations, and Lyapunov
weight $Q=I_N$, we optimize the diagonal stabilizer gains $\Lambda$ and the Kronecker
representation coefficients $\mu$. The gains determine the lifted dynamics, while the
representation coefficients determine the spectral-norm bound used in the certificate. Over a
prescribed compact box $B$ contained in the Hurwitz set,
we formulate optimal dissipative quadratization
(\textsc{ODQ}) with optimal value
$\phi_B:=\min_{\Lambda\in B}\min_{\mu\in\R^d}\norm{J(\Lambda,\mu)}$,
where $d=N^2(N-1)/2$. Its reciprocal
$\rho_B:=\phi_B^{-1}$ is the optimal Lyapunov-level threshold within this family and box,
with $1/0:=+\infty$. For fixed $\Lambda$, the gauge problem is
an exact convex \gls{sdp}; the outer value function is generally nonsmooth and nonconvex.

\paragraph{Contributions and organization}
Within this fixed family, we first derive a tangent--transverse spectral splitting that reduces
Hurwitz feasibility to the auxiliary block.  We prove that the outer value is locally Lipschitz and
attains its minimum on every admissible box, and give an example for which the unrestricted
threshold is unbounded.  We then derive a residual-aware rule for reporting a strict closed
certificate and show that, when $\phi_B>0$ and $Q=I_N$, threshold maximization orders the leading
coefficients of scaled small-radius pullback volumes.  For computation, we remove gauge directions
invisible to the certificate and obtain an exact fixed-gain \gls{sdp} with $q=N(N^2-1)/3$
effective gauge coordinates.  Analytic sensitivities at regular points are combined with a
box-constrained gradient-sampling L-BFGS method.  Under the stated assumptions, every accumulation
point of the idealized sequence is box-Clarke stationary; finite runs instead return the best
independently screened candidate encountered, and we do not claim
global optimality of the outer gain search.

The numerical study comprises a planar quintic, a heterogeneous suite of 19 polynomial systems,
and a 12-system single-link relay grid.  On the planar example, gain optimization increases the
certified area by a factor of $2.238$ over the matched zero-gain gauge; on all 16 heterogeneous
systems with nontrivial stabilizer freedom, \textsc{ODQ} improves the reported size metric over
both fixed-gain lifted baselines.  Direct \gls{sos} methods with
quadratic or quartic Lyapunov functions give a
structure-dependent size--time tradeoff rather than uniform dominance.  All 36 \textsc{ODQ} relay
runs completed, and each of the nine fully paired \textsc{sos}-2 cells favors \textsc{ODQ} in both
proxy and time.  Within the tested dimensions, the favorable cases are higher-degree,
higher-dimensional systems whose sparse nonlinear structure admits a compact quadratic lift.  Section~\ref{sec:prelim} gives the required background,
Section~\ref{sec:formulation} develops the certificate problem,
Section~\ref{sec:alg} presents the algorithm, and Section~\ref{sec:num} reports the experiments.

\section{Preliminaries}
\label{sec:prelim}

\subsection{Notation}
\label{subsec:notation}
Vectors are set in bold and their scalar entries in plain type: $\mathbf x=(x_1,\dots,x_n)$, and the
origin of $\R^n$ is $\mathbf 0$. Matrices are plain capitals ($A$, $H$, $P$), and scalars are plain
lowercase. We write $\R$ for the real numbers and $\R[\mathbf x]$ for the ring of polynomials in
$\mathbf x$, with $\R[\mathbf x]^n$ its $n$-vectors. For a matrix $M$, $M^{\!\top}$ is the transpose,
$\norm{M}$ the spectral norm (its largest singular value), and $\eig(M)$ is the spectrum; for a vector,
$\norm{\cdot}$ is the Euclidean norm. We index eigenvalues by decreasing real part,
$\Real\lambda_1(M)\ge\cdots\ge\Real\lambda_N(M)$, so $M$ is \emph{Hurwitz} precisely when
$\Real\lambda_1(M)<0$; the quantity $-\Real\lambda_1(M)>0$ is then its \emph{stability margin}.
$\Sset^N$ denotes the symmetric $N\times N$ matrices, $P\succ0$ ($P\succeq0$) positive
(semi)definiteness, and $\preceq$ the Loewner order. The Jacobian of a differentiable map
$\mathbf f$ is $J_x(\mathbf f)=[\partial f_i/\partial x_j]$. We use $\kron$ for the Kronecker product
and write $\mathbf w^{\kron j}=\mathbf w\kron\cdots\kron\mathbf w$ ($j$ factors) for Kronecker
powers; $I_k$ is the $k\times k$ identity and $\mathbf e_k$ the $k$-th standard basis vector. A
homogeneous quadratic vector field on $\R^N$ is written $H(\mathbf w\kron\mathbf w)$ with $H\in\R^{N\times N^2}$,
so that its $r$-th component is the quadratic form $\mathbf w^{\!\top}H_r\,\mathbf w$, where
$H_r\in\R^{N\times N}$ is the $r$-th row of $H$ reshaped to a matrix.

\subsection{Polynomial systems and dissipative equilibria}
\label{subsec:polynomial_systems}

We study polynomial systems of \glspl{ode},
\begin{equation}\label{eq:sys}
\dot{\mathbf x} = \mathbf p(\mathbf x),\qquad \mathbf x(t)\in\R^n,\quad
\mathbf p=(p_1,\dots,p_n)\in\R[\mathbf x]^n .
\end{equation}
Each component $p_i\in\R[\mathbf x]$ is a polynomial in the state $\mathbf x=(x_1,\dots,x_n)$. Our
interest is in the behaviour of \eqref{eq:sys} near an equilibrium at which it is locally stable.

\begin{definition}[equilibrium; dissipativity \citep{cai2024dissipative}]\label{def:dissip}
A point $\mathbf x^*\in\R^n$ is an \emph{equilibrium} of \eqref{eq:sys} if
$\mathbf p(\mathbf x^*)=\mathbf 0$, and \eqref{eq:sys} is \emph{dissipative at $\mathbf x^*$} if every
eigenvalue of the Jacobian $J_x(\mathbf p)|_{\mathbf x=\mathbf x^*}$ has negative real part.
\end{definition}

We analyse one equilibrium at a time and place it at the origin without loss of generality. The
substitution $\mathbf z=\mathbf x-\mathbf x^*$ takes \eqref{eq:sys} to
$\dot{\mathbf z}=\mathbf p(\mathbf z+\mathbf x^*)$, which is again polynomial, has an equilibrium
at $\mathbf 0$, and has the same Jacobian there. We retain the notation
$\mathbf x$ and $\mathbf p$ for the translated coordinates and vector field, so
$\mathbf p(\mathbf 0)=\mathbf 0$.

\begin{definition}[dissipativity margin]\label{def:margin}
The \emph{dissipativity margin} of the origin is
\begin{equation}\label{eq:alphap}
\alpha_p \;:=\; -\Real\lambda_1\big(J_x(\mathbf p)|_{\mathbf 0}\big),
\end{equation}
so that \eqref{eq:sys} is dissipative at $\mathbf 0$ precisely when $\alpha_p>0$.
\end{definition}

Our analysis assumes that the original system \eqref{eq:sys}
is dissipative at the origin, equivalently, $\alpha_p>0$, since exact quadratization
preserves the original dynamics and cannot alter the local stability of this equilibrium.
By Lyapunov's indirect method, for every rate
$a<\alpha_p$ there exist a neighbourhood of the origin and a constant $C\ge1$ such that
$\norm{\mathbf x(t)}\le C\,e^{-at}\norm{\mathbf x(0)}$ for all $t\ge0$. Trajectories therefore relax on
the asymptotic time scale $1/\alpha_p$ and the constant $C$ absorbs nonnormal transient amplification,
which $\alpha_p$ does not bound. Geometrically, $-\alpha_p$ is the spectral abscissa of
$J_x(\mathbf p)|_{\mathbf 0}$, so the spectrum lies at least $\alpha_p$ to the left of the
imaginary axis and $\alpha_p$ is its distance to the stability boundary. Thus, the margin is intrinsic
to \eqref{eq:sys} and fixed before any lift is chosen.

\subsection{Quadratization, stabilizers, and the quadratic form}
\label{subsec:quadratization}

A quadratic right-hand side permits closed-form Lyapunov bounds and convex certificate
optimization. For a fixed Hurwitz linear part and a fixed positive definite Lyapunov weight,
the analytic ellipsoidal \gls{roa} radius scales as $1/\norm{H}$ when $H\ne0$ and
is inexpensive to evaluate once the Lyapunov equation is solved
\citep[Prop.~3.1 and \S4]{kramer2021stability}. The nonlinear term is represented by a fixed
tensor whose projection can be precomputed, allowing exact evaluation of the projected
nonlinearity in reduced coordinates without interpolation
\citep{kramer2025discovering,qian2020lift}.

Reducing the polynomial degree requires auxiliary state variables. Any polynomial system
admits a finite quadratic extension \citep[Prop.~1]{cai2024dissipative}, with freedom in both
the lifting map and the off-manifold dynamics. Algorithms for \emph{optimal monomial
quadratization} minimize the number of auxiliary variables
\citep{bychkov2021optimal,bychkov2024exact}, whereas \emph{dissipative quadratization} also
requires the lifted equilibrium to have a Hurwitz Jacobian \citep{cai2024dissipative}.
We describe these freedoms below and formulate the gain and gauge optimization in
Section~\ref{sec:formulation}.

\begin{definition}[quadratization \citep{cai2024dissipative}]\label{def:quad}
A \emph{quadratization} of \eqref{eq:sys} consists of new variables $y_i=g_i(\mathbf x)$,
$i=1,\dots,m$, together with polynomial vectors $\mathbf q_1(\mathbf x,\mathbf y)\in\R[\mathbf x,\mathbf y]^n$
and $\mathbf q_2(\mathbf x,\mathbf y)\in\R[\mathbf x,\mathbf y]^m$ of total degree at most two such that
\begin{equation}\label{eq:quadsys}
\dot{\mathbf x} = \mathbf q_1(\mathbf x,\mathbf y),\qquad \dot{\mathbf y} = \mathbf q_2(\mathbf x,\mathbf y)
\end{equation}
holds along $\mathbf y=\mathbf g(\mathbf x)$. It is \emph{monomial quadratization} if every $g_i$ is a monomial, and
its \emph{lifting manifold} is $\M=\{(\mathbf x,\mathbf g(\mathbf x)):\mathbf x\in\R^n\}\subset\R^N$
with $N=n+m$.
\end{definition}

We restrict attention to \emph{monomial} quadratizations with nonconstant auxiliary monomials,
so $\mathbf g(\mathbf 0)=\mathbf 0$ and exactness places the lifted equilibrium at the origin.
With $\mathbf w=(\mathbf x,\mathbf y)\in\R^N$, collecting the linear and
quadratic terms in \eqref{eq:quadsys} gives the lifted vector field
\begin{equation}\label{eq:lifted}
\dot{\mathbf w}=\mathbf F(\mathbf w):=A\mathbf w+H(\mathbf w\kron\mathbf w),
\qquad A\in\R^{N\times N},\quad H\in\R^{N\times N^2}.
\end{equation}
Here $A=J_w(\mathbf F)|_{\mathbf 0}$ is the linearization at the lifted equilibrium, and $H$
is a Kronecker representation of the homogeneous quadratic term.

Reading off the coefficients of $\mathbf q_1,\mathbf q_2$ gives one such pair, written $(A_0,H_0)$ and
called the \emph{zero-gain lift}.
The corresponding field
$\mathbf F_0(\mathbf w):=A_0\mathbf w+H_0(\mathbf w\kron\mathbf w)$ serves as the reference
extension below.
The original dynamics are exactly the
restriction of \eqref{eq:lifted} to the invariant manifold $\M$. Off $\M$, equation \eqref{eq:lifted}
is an \emph{extension} of those dynamics, and the quadratization is free to shape
that extension without touching the system it represents.

\begin{definition}[admissible quadratic extensions]\label{def:extension-space}
Fix the polynomial embedding $\Phi:\R^n\to\R^N$,
$\Phi(\mathbf x)=(\mathbf x,\mathbf g(\mathbf x))$, with image $\M$ and $N=n+m$.
Let $\mathcal P_2:=\R[\mathbf w]_{\le2}$ denote the scalar polynomials of total degree at most
two, and define the \emph{vanishing ideal}
\[
I(\M):=\{r\in\R[\mathbf w]:r(\Phi(\mathbf x))=0\ \text{for all }\mathbf x\in\R^n\}.
\]
The \emph{admissible quadratic extension space} of $\mathbf p$ for the fixed embedding is
\begin{equation}\label{eq:extension-space}
\mathcal E_\Phi(\mathbf p):=
\{\mathbf F\in\mathcal P_2^N:
\mathbf F(\Phi(\mathbf x))=J_x(\Phi)(\mathbf x)\,\mathbf p(\mathbf x)
\ \text{for all }\mathbf x\in\R^n\}.
\end{equation}
Given a reference extension $\mathbf F_0\in\mathcal E_\Phi(\mathbf p)$, this is the affine space
\begin{equation}\label{eq:extension-affine}
\mathcal E_\Phi(\mathbf p)=\mathbf F_0+\mathcal V_\Phi,
\qquad
\mathcal V_\Phi:=I(\M)^N\cap\mathcal P_2^N,
\end{equation}
where $I(\M)^N$ denotes the Cartesian product of $N$ copies of $I(\M)$.
\end{definition}

Indeed, for $\mathbf F\in\mathcal P_2^N$ and the fixed reference $\mathbf F_0$, exactness in
\eqref{eq:extension-space} is equivalent to $\mathbf F-\mathbf F_0$ vanishing componentwise on
$\M$. A basis of $I(\M)\cap\mathcal P_2$ can be computed by
substituting $\mathbf w=\Phi(\mathbf x)$ into a general linear combination of
$1,w_i,w_iw_j$ ($1\le i\le j\le N$) and taking the null space of the resulting coefficient
equations. If $r_1,\dots,r_s$ is such a basis, then
$\{\mathbf e_k r_\ell:1\le k\le N,\ 1\le\ell\le s\}$ is a basis of $\mathcal V_\Phi$.

\begin{example}[Admissible extensions of a scalar cubic system]
\label{ex:extension-space-example}
Consider $\dot x=p(x):=-x-x^3$ and fix $\Phi(x)=(x,x^2)$, so
$\M=\{(x,y):y=x^2\}$. The quadratic vector field
\[
\mathbf F_0(x,y)=\begin{pmatrix}-x-xy\\-2y-2y^2\end{pmatrix}
\]
is an admissible extension because
$\mathbf F_0(x,x^2)=(p(x),2xp(x))^{\!\top}=J_x(\Phi)(x)p(x)$.

To find all such extensions, write
$r=c_0+c_1x+c_2y+c_3x^2+c_4xy+c_5y^2\in\mathcal P_2$. Substitution gives
\[
r(x,x^2)=c_0+c_1x+(c_2+c_3)x^2+c_4x^3+c_5x^4.
\]
This vanishes identically if and only if $c_0=c_1=c_4=c_5=0$ and $c_3=-c_2$.
Thus $I(\M)\cap\mathcal P_2=\operatorname{span}\{h\}$, where $h(x,y)=y-x^2$, and
\[
\begin{aligned}
\mathcal V_\Phi
=\operatorname{span}\left\{\begin{pmatrix}h\\0\end{pmatrix},
\begin{pmatrix}0\\h\end{pmatrix}\right\}, \qquad
\mathcal E_\Phi(p)
=\left\{\mathbf F_0+\begin{pmatrix}a h\\b h\end{pmatrix}:a,b\in\R\right\}.
\end{aligned}
\]
The two parameters independently modify the lifted equations away from $\M$ while preserving
the dynamics on $\M$. The choice $a=0$, $b=-\lambda$ gives the family in
Example~\ref{ex:unbounded-ray}.
\end{example}

\begin{definition}[inner-quadratic lift; stabilizers \citep{cai2024dissipative}]\label{def:stab}
A quadratization is \emph{inner-quadratic} if each $g_i$ factors as $g_i=a_ib_i$ with
$a_i,b_i\in\{x_1,\dots,x_n,y_1,\dots,y_{i-1}\}$. Its $i$-th \emph{stabilizer} is the quadratic
polynomial
\begin{equation}\label{eq:stabdef}
h_i(\mathbf x,\mathbf y) \;:=\; y_i-a_ib_i ,
\end{equation}
which vanishes identically on $\M$ and therefore belongs to
$I(\M)\cap\mathcal P_2$ in Definition~\ref{def:extension-space}.
\end{definition}

Every polynomial system admits an inner-quadratic monomial quadratization
\cite[Prop.~1]{cai2024dissipative}, computable by branch and bound \cite{bychkov2021optimal}.
Since $\mathbf e_kh_i\in\mathcal V_\Phi$ for every $k=1,\dots,N$,
adding $-\lambda_i h_i$ to the $k$-th lifted equation, with $\lambda_i\in\R$, gives another
element of $\mathcal E_\Phi(\mathbf p)$: its restriction to $\M$ is unchanged, while its
off-manifold dynamics may change. Those transverse
dynamics determine the stability of the lifted equilibrium, which makes the stabilizers the
natural design variables.
Dissipative quadratization \cite{cai2024dissipative} exploits exactly this: it replaces $\vecq_2$ by
$\vecq_2-\lambda\vech$ with $\vech=(h_1,\dots,h_m)$ and raises $\lambda$ through $1,2,4,\dots$ until the
lifted origin is dissipative, a search that always terminates \cite[Alg.~2, Prop.~3]{cai2024dissipative}.
It returns the first gain that works.

\subsection{Ellipsoidal Lyapunov certificates}
\label{sec:certs}

Suppose that $A$ in \eqref{eq:lifted} is Hurwitz. For a prescribed $Q\succ0$, let $P\succ0$ be
the unique Lyapunov solution and define $V$ by
\begin{align} \label{eq:lyap}
A^{\!\top}P+PA=-Q, \ \
V(\mathbf w)=\mathbf w^{\!\top}P\mathbf w, \ \
\dot V(\mathbf w)=-\mathbf w^{\!\top}Q\mathbf w
 +2\mathbf w^{\!\top}PH(\mathbf w\kron\mathbf w).
\end{align}

Write $\mathbf q(\mathbf w):=H(\mathbf w\kron\mathbf w)$ for the homogeneous quadratic term.
For symmetric positive definite $P,Q$, define its \emph{intrinsic normalized cubic growth} by
\begin{equation}\label{eq:intrinsic-growth}
\beta(P,Q,\mathbf q):=
\sup_{\mathbf w\ne\mathbf 0}
\frac{2|\mathbf w^{\!\top}P\mathbf q(\mathbf w)|}
{\sqrt{\mathbf w^{\!\top}P\mathbf w}\,\mathbf w^{\!\top}Q\mathbf w}.
\end{equation}
This quantity depends on the quadratic map, independently of its matrix representation.
For the exact Lyapunov solution in \eqref{eq:lyap}, set
$c(\mathbf u):=\mathbf u^{\!\top}P\mathbf q(\mathbf u)$ and define
\begin{equation}\label{eq:rhorad}
\rho_{\mathrm{rad}}
:=\inf_{\substack{\mathbf u^{\!\top}P\mathbf u=1\\ c(\mathbf u)>0}}
\frac{\mathbf u^{\!\top}Q\mathbf u}{2\,c(\mathbf u)}, \quad
\inf\varnothing:=+\infty .
\end{equation}

\begin{proposition}[Ellipsoidal Lyapunov certificate]
\label{prop:ellipsoidal-certificate}
For the exact Lyapunov solution in \eqref{eq:lyap},
$\rho_{\mathrm{rad}}=1/\beta(P,Q,\mathbf q)$ is the supremal $r>0$ for which
$\dot V(\mathbf w)<0$ whenever $0<V(\mathbf w)<r^2$, with $1/0:=+\infty$.
For every $0<\rho<\rho_{\mathrm{rad}}$, the closed ellipsoid
$\mathcal D(P,\rho):=\{\mathbf w\in\R^N:V(\mathbf w)\le\rho^2\}$ is positively invariant and
contained in the \gls{roa} of the origin. For the chosen Kronecker representation $H$,
\begin{equation}\label{eq:rhoan}
\rho_{\mathrm{an}}(A,H,Q):=
\frac{\lambda_{\min}(Q)}{2\norm{H}\sqrt{\norm{P}}}
\le \rho_{\mathrm{rad}},
\qquad 1/0:=+\infty .
\end{equation}
Thus $\mathcal D(P,\rho)$ is certified for every $0<\rho<\rho_{\mathrm{an}}(A,H,Q)$. 

More generally, let $\widehat P=\widehat P^{\!\top}\succ0$ be any computed matrix, put
$\widehat V(\mathbf w):=\mathbf w^{\!\top}\widehat P\mathbf w$, and set
\begin{equation}\label{eq:metric-residual}
R:=A^{\!\top}\widehat P+\widehat PA+Q,
\qquad
\delta:=\norm{Q^{-1/2}RQ^{-1/2}}.
\end{equation}
If $\delta<1$ and $\gamma\ge\beta(\widehat P,Q,\mathbf q)$, then
\begin{equation}\label{eq:metric-residual-decay}
\dot{\widehat V}(\mathbf w)
\le-\bigl(1-\delta-\gamma\sqrt{\widehat V(\mathbf w)}\bigr)
\mathbf w^{\!\top}Q\mathbf w.
\end{equation}
For $\gamma>0$ and $\theta\in(0,1)$, the closed ellipsoid
\begin{equation}\label{eq:metric-safe-ellipsoid}
\mathcal D\!\left(\widehat P,\theta\frac{1-\delta}{\gamma}\right)
=\left\{\mathbf w:\widehat V(\mathbf w)
\le\theta^2\left(\frac{1-\delta}{\gamma}\right)^2\right\}
\end{equation}
is positively invariant, has strict Lyapunov decrease away from the origin, and is contained
in the region of attraction of the origin. If $\gamma=0$, $\widehat V$ is a global Lyapunov
function. Verified upper bounds may replace $\delta$ and $\gamma$, provided the residual bound
is below one.
\end{proposition}

\begin{proof}
See the \hyperref[proof:ellipsoidal-certificate]{proof of Proposition~\ref*{prop:ellipsoidal-certificate}}
in Appendix~\ref{app:ellipsoidal-certificate} (p.~\pageref{proof:ellipsoidal-certificate}).
\end{proof}

The bound \eqref{eq:rhoan} is the $E=I_N$ case of \citep[Prop.~3.1]{kramer2021stability}. For fixed
$P$ and $Q$, $\rho_{\mathrm{rad}}$ depends only on the quadratic vector field, whereas
$\rho_{\mathrm{an}}$ depends on its Kronecker representation; the gauge-optimized bound is
introduced after the dynamics-preserving gauge.

\section{Problem formulation}
\label{sec:formulation}

Under the standing assumptions of Section~\ref{sec:prelim}, fix an inner-quadratic monomial
lift $\mathbf g$ with $\deg g_i\ge2$.
The original system $\mathbf p$, embedding $\Phi(\mathbf x)=(\mathbf x,\mathbf g(\mathbf x))$,
and lifting manifold $\M$ remain fixed throughout the optimization.
Every admissible extension has the form \eqref{eq:lifted}, with a homogeneous quadratic term
and no constant term.

For this fixed embedding, we first select an off-manifold extension and then a matrix
representation of its quadratic term. The design variables are the stabilizer gains
$\Lambda\in\R^m$ and the representation coefficients $\mu\in\R^d$, with
$d=N^2(N-1)/2$. Fixing the lift keeps the state dimension and manifold unchanged;
fixing the reference extension and stabilizer factorizations specifies the affine family
below. We also fix the Lyapunov weight $Q=I_N$ to isolate gain and representation design.
For each extension with $A(\Lambda)$ Hurwitz, the certificate matrix
$P(\Lambda)\in\Sset^N$, $P(\Lambda)\succ0$, is determined by the Lyapunov equation,
rather than optimized independently.
Figure~\ref{fig:problem-flow} in Appendix~\ref{app:problem-structure} shows how these fixed
data and the two design variables determine the certificate and the optimization problem.

\subsection{The affine family of dissipative lifts}
\label{sec:family}

Fix the reference extension $\mathbf F_0=(\mathbf q_1,\mathbf q_2)$ represented by
$(A_0,H_0)$ and one factorization $g_i=a_ib_i$ for each stabilizer $h_i=y_i-a_ib_i$.
These data are computed once using the implementation of \textsc{DQbee} \citep{cai2024dissipative}.
With these data fixed, Definition~\ref{def:extension-space} gives the admissible extension family
$\mathcal E_\Phi(\mathbf p)=\mathbf F_0+\mathcal V_\Phi$, where
$\mathcal V_\Phi=I(\M)^N\cap\mathcal P_2^N$.

We restrict the extension freedom to the stabilizer subspace
\begin{equation}\label{eq:stabilizer-subspace}
\mathcal S_h:=\operatorname{span}\{\mathbf e_{n+i}h_i:1\le i\le m\}
\subseteq\mathcal V_\Phi,
\qquad
\mathbf F_\Lambda:=\mathbf F_0-\sum_{i=1}^m\lambda_i\mathbf e_{n+i}h_i.
\end{equation}
Thus $\{\mathbf F_\Lambda:\Lambda\in\R^m\}=\mathbf F_0+\mathcal S_h$ is an
$m$-dimensional affine subspace of $\mathcal E_\Phi(\mathbf p)$.
The restriction gives one gain per auxiliary equation and the transverse matrix
$D(\Lambda)=W-\diag(\Lambda)$ established in \cref{lem:split}. Directions in
$\mathcal V_\Phi\setminus\mathcal S_h$, including those that modify the physical equations
off $\M$, are excluded from the present family.
\hyperref[proof:extension-space]{Appendix~\ref*{app:extension-gauge}}
(p.~\pageref{proof:extension-space}) gives the coefficient-matching construction of a basis
of $\mathcal V_\Phi$ and verifies the inclusion and dimension stated above.

Writing $\Lambda=(\lambda_1,\dots,\lambda_m)$, the selected extension is

\begin{equation}\label{eq:gainsys}
\dot{\mathbf x} = \mathbf q_1(\mathbf x,\mathbf y),\qquad
\dot y_i = q_{2,i}(\mathbf x,\mathbf y)-\lambda_i\,h_i(\mathbf x,\mathbf y),\quad i=1,\dots,m .
\end{equation}
Since $h_i\equiv0$ on $\M$, system \eqref{eq:gainsys} is a valid quadratization of \eqref{eq:sys} for every $\Lambda$. Split the stabilizer \eqref{eq:stabdef} into linear and quadratic parts, $-\lambda_ih_i=-\lambda_iy_i+\lambda_ia_ib_i$. The term $-\lambda_iy_i$ enters the linear part and $+\lambda_ia_ib_i$ the quadratic part, so
\begin{equation}\label{eq:AH}
\Aof \;=\; A_0-\sum_{i=1}^m\lambda_i\,\mathbf e_{n+i}\mathbf e_{n+i}^{\!\top},
\qquad
H(\Lambda) \;=\; H_0+\sum_{i=1}^m\lambda_i\,\mathbf e_{n+i}\,s_i^{\!\top},
\end{equation}
with $s_i\in\R^{N^2}$ the Kronecker encoding of $a_ib_i$. Both maps are affine in $\Lambda$. Each
gain subtracts $\lambda_i$ from one auxiliary diagonal entry of $A$ and modifies one row of $H$.

The chart $\mathbf z=\mathbf y-\mathbf g(\mathbf x)$ straightens $\M$ to
$\{\mathbf z=\mathbf 0\}$. Since every auxiliary monomial has degree at least two,
$J_x(\mathbf g)|_{\mathbf 0}=0$, so this change of coordinates has identity derivative at the
origin. The resulting tangent--transverse decomposition is the vector-gain specialization of
\citep[Eq.~(9)]{cai2024dissipative}.

\begin{lemma}[tangent--transverse spectrum splitting]\label{lem:split}
Set $J_p:=J_x(\mathbf p)|_{\mathbf 0}\in\R^{n\times n}$,
$R:=J_y(\mathbf q_1)|_{(\mathbf 0,\mathbf 0)}\in\R^{n\times m}$,
and $W:=J_y(\mathbf q_2)|_{(\mathbf 0,\mathbf 0)}\in\R^{m\times m}$.
Then, for every $\Lambda\in\R^m$,
\begin{equation}\label{eq:block}
\Aof=
\begin{bmatrix}
J_p & R\\[2pt]
0 & \Dof
\end{bmatrix},
\qquad
\Dof:=W-\diag(\Lambda).
\end{equation}
Consequently,
\begin{equation}\label{eq:split}
\eig(\Aof)=\eig(J_p)\cup\eig(\Dof),
\end{equation}
with algebraic multiplicities. Under the standing assumption $\alpha_p>0$, $\Aof$ is Hurwitz if
and only if $\Dof$ is Hurwitz. Moreover, its spectral margin satisfies
\begin{equation}\label{eq:margin}
\alpha(\Lambda):=-\Real\lambda_1(\Aof)
=\min\big\{\alpha_p,\alpha_{\mathrm{aux}}(\Lambda)\big\},
\qquad
\alpha_{\mathrm{aux}}(\Lambda):=-\Real\lambda_1(\Dof).
\end{equation}
When $m=0$, $A=J_p$ and we set
$\alpha_{\mathrm{aux}}=+\infty$.
Thus the gains leave the original spectrum fixed and cannot raise the lifted margin above
$\alpha_p$.
\end{lemma}

\begin{proof}
See the \hyperref[proof:split]{proof of Lemma~\ref*{lem:split}}
in Appendix~\ref{app:transverse-feasibility} (p.~\pageref{proof:split}).
\end{proof}

\begin{corollary}[closed-form transverse feasibility and saturation]\label{cor:sat}
Assume $m\ge1$ and let $\mathbf 1=(1,\dots,1)\in\R^m$. For every $\Lambda\in\R^m$ and $c\in\R$,
\begin{equation}\label{eq:gain-shift}
D(\Lambda+c\mathbf 1)=D(\Lambda)-cI_m,
\qquad
\alpha_{\mathrm{aux}}(\Lambda+c\mathbf 1)=\alpha_{\mathrm{aux}}(\Lambda)+c.
\end{equation}
In particular, $\alpha_{\mathrm{aux}}(\lambda\mathbf 1)=\lambda-\Real\lambda_1(W)$, so for every
$\varepsilon>0$,
\begin{equation}\label{eq:uniform-feasibility}
\alpha_{\mathrm{aux}}(\lambda\mathbf 1)\ge\varepsilon
\quad\Longleftrightarrow\quad
\lambda\ge\Real\lambda_1(W)+\varepsilon.
\end{equation}
The full lifted margin satisfies $\alpha(\lambda\mathbf 1)\ge\varepsilon$ if and only if
both $\lambda\ge\Real\lambda_1(W)+\varepsilon$ and $\alpha_p\ge\varepsilon$ hold. Under the standing
assumption $\alpha_p>0$, the unrestricted set
$\{\Lambda\in\R^m:\Aof\text{ is Hurwitz}\}$ is therefore nonempty, and
\begin{equation}\label{eq:uniform-saturation}
\alpha(\lambda\mathbf 1)=\alpha_p
\qquad\text{whenever}\qquad
\lambda\ge\Real\lambda_1(W)+\alpha_p.
\end{equation}
Let $\Delta H:=\sum_{i=1}^m\mathbf e_{n+i}s_i^{\!\top}$. Then
$H(\lambda\mathbf 1)=H_0+\lambda\Delta H$ and $\norm{H(\lambda\mathbf 1)}=O(\lambda)$ as
$\lambda\to\infty$; if $\Delta H\ne0$, then
$\norm{H(\lambda\mathbf 1)}/\lambda\to\norm{\Delta H}$.
\end{corollary}

\begin{proof}
See the \hyperref[proof:sat]{proof of Corollary~\ref*{cor:sat}}
in Appendix~\ref{app:transverse-feasibility} (p.~\pageref{proof:sat}).
\end{proof}

The bound \eqref{eq:uniform-feasibility} provides an explicit uniform-gain initializer whenever
the resulting vector lies in the prescribed gain box. Although the unrestricted family is
feasible, its intersection with a prescribed box can be empty.

Thus feasibility alone does not select a useful gain: once the lifted margin reaches $\alpha_p$,
increasing $\lambda$ continues to change the quadratic term without improving the linear margin.

\begin{remark}[Factorization freedom]\label{rem:factorization}
The stabilizers of a fixed monomial lift need not be unique. Different factorizations into
earlier coordinates give stabilizers with the same linear part, all vanishing on $\M$.
For example, for $\Phi(x)=(x,x^2,x^3,x^4)$, both $h_3=y_3-y_1^2$ and
$\widetilde h_3=y_3-xy_2$ are valid, with $h_3-\widetilde h_3=xy_2-y_1^2$.
With the reference extension and $\Lambda$ fixed, changing the factorization can alter
$H(\Lambda)$ while leaving $A(\Lambda)$ unchanged. When $A(\Lambda)$ is Hurwitz, its
Lyapunov solution is therefore unchanged as well. We keep the stabilizer factorizations
fixed throughout the optimization.
\end{remark}

After fixing $\Lambda$, a separate freedom remains in representing the same quadratic map.
Let $\mathcal K\subset\R^{N\times N^2}$ be the space of matrices satisfying
$K(\mathbf w\kron\mathbf w)=0$ for every $\mathbf w\in\R^N$.
Two matrices represent the same quadratic map if and only if their difference belongs to
$\mathcal K$. Thus $H(\Lambda)+\mathcal K$ contains all Kronecker representations of the
quadratic term of $\mathbf F_\Lambda$. The space has dimension $d=\tfrac12N^2(N-1)$.
See the \hyperref[proof:representation-space]{representation-space derivation}
in Appendix~\ref{app:representation-space} (p.~\pageref{proof:representation-space}).
For a fixed basis $\{K_k\}_{k=1}^{d}$, write

\begin{equation}\label{eq:Hmu}
\Hof:=H(\Lambda)+K(\mu),\qquad
K(\mu):=\sum_{k=1}^{d}\mu_kK_k,\qquad \mu\in\R^d.
\end{equation}

Unlike the extension directions in $\mathcal V_\Phi$, which need only vanish on $\M$, the
gauge contributes the zero vector field on all of $\R^N$. Varying $\mu$ therefore leaves
$\mathbf F_\Lambda$, $A(\Lambda)$, and every trajectory unchanged, while it can change the
matrix norm used to bound Lyapunov growth. We optimize over all $\mu\in\R^d$ in this exact
representation class. The certificate construction below specifies which norm is minimized.

\subsection{The certificate and the optimization problem}
\label{sec:certmap}

Fix an extension $\mathbf F_\Lambda$ with $\Aof$ Hurwitz. Its certificate is constructed from
the quadratic Lyapunov function of Section~\ref{sec:certs} and the spectral-norm bound below.
The choice $Q=I_N$ fixes the Lyapunov weight while the gains and representation vary.
The construction is the $Q=I_N$ specialization of
\citep[Thm.~4.1 and Eqs.~(4.8)--(4.12)]{kramer2021stability}, building on
\citep{tesi1996stability}. Let $P=P(\Lambda)\succ0$ solve
\begin{equation}\label{eq:lyapL}
\Aof^{\!\top}P+P\Aof=-I_N,
\end{equation}
and factor $P=P_f^{\!\top}P_f$. Since $A$ is gauge-independent, $P$ depends on $\Lambda$ but not
on $\mu$.

\begin{remark}[Choice of Lyapunov weight]\label{rem:Q}
We fix $Q=I_N$ to isolate the lift and gauge design. A scalar rescaling is redundant: replacing
$Q$ by $sQ$ replaces $P$ by $sP$ and only rescales the level parameter of the same ellipsoids.
The normalized shape of $Q$ is a genuine design variable. For $Q=Q_f^{\!\top}Q_f$, the general
certificate replaces $J$ below by
$J_Q=(P_f^{\!\top}\kron Q_f^{\!\top})^{-1}GQ_f^{-1}$; it could be optimized over, for example,
$Q\succeq\varepsilon I_N$ and $\operatorname{tr}Q=N$. This adds a nonconvex outer variable and is
left open.
\end{remark}

Partition $\Hof$ into blocks $M_i=M_i(\Lambda,\mu)\in\R^{N\times N}$ defined by
$(M_i)_{rj}:=(\Hof)_{r,(i-1)N+j}$, so that
$\Hof(\mathbf w\kron\mathbf w)=\sum_{i=1}^Nw_iM_i\mathbf w$. With $\operatorname{col}$ denoting
vertical stacking, define
\begin{equation}\label{eq:J}
\begin{aligned}
G(\Lambda,\mu)&:=\operatorname{col}_{i=1}^N\!\big(M_i^{\!\top}P+PM_i\big)
\in\R^{N^2\times N}, \\
J(\Lambda,\mu)&:=\big(P_f^{\!\top}\kron I_N\big)^{-1}G(\Lambda,\mu)
\in\R^{N^2\times N}.
\end{aligned}
\end{equation}
The spectral norm of $J$ is independent of the factor: replacing $P_f$ by $UP_f$, with $U$ orthogonal,
replaces $J$ by $(U\kron I_N)J$. For $V(\mathbf w)=\mathbf w^{\!\top}P\mathbf w$,
\begin{align}\label{eq:vdotJ}
\dot V(\mathbf w)
&=\mathbf w^{\!\top}\!\left[-I_N+
(\mathbf w^{\!\top}P_f^{\!\top}\kron I_N)J(\Lambda,\mu)\right]\!\mathbf w \le \norm{\mathbf w}^2
\left[-1+\sqrt{V(\mathbf w)}\,\norm{J(\Lambda,\mu)}\right]\end{align}
Hence $\dot V(\mathbf w)<0$ for every nonzero $\mathbf w$ with $V(\mathbf w)\le\rho^2$ whenever
$0<\rho<1/\norm{J(\Lambda,\mu)}$, with $1/0:=+\infty$;
Proposition~\ref{prop:ellipsoidal-certificate} then certifies the corresponding closed ellipsoid.

\begin{proposition}[Intrinsic growth and spectral-norm relaxation]\label{prop:intrinsic-bound}
Fix symmetric positive definite $P,Q$ and a homogeneous quadratic map
$\mathbf q(\mathbf w)=H(\mathbf w\kron\mathbf w)$.
For each representation $H+K(\mu)$ in \eqref{eq:Hmu}, form $G$ by \eqref{eq:J} and $J_Q$ by
Remark~\ref{rem:Q}. Then
\begin{equation}\label{eq:intrinsic-spectral-bound}
\beta(P,Q,\mathbf q)
\le\gamma_\star:=\min_{\mu\in\R^d}\norm{J_Q(H+K(\mu))}
\le\norm{J_Q(H+K(\mu))}\qquad(\mu\in\R^d).
\end{equation}
The minimum is attained. It is the smallest spectral-norm upper bound in this gauge and
matricization class; gauge optimality does not establish equality with $\beta$.
\end{proposition}
\begin{proof}
See the \hyperref[proof:intrinsic-bound]{proof of Proposition~\ref*{prop:intrinsic-bound}}
in Appendix~\ref{app:intrinsic-bound} (p.~\pageref{proof:intrinsic-bound}).
\end{proof}

For a fixed extension, $\beta$ and $P(\Lambda)$ are fixed, so the
inner optimization improves the spectral-norm upper bound alone. We define its optimized
threshold by
\begin{equation}\label{eq:rhostarL}
\rho^\star(\Lambda):=
\left(\min_{\mu\in\R^d}\norm{J(\Lambda,\mu)}\right)^{-1}.
\end{equation}
For fixed $\Lambda$, the certificate map is affine in $\mu$:
$
J(\Lambda,\mu)
=
J(\Lambda,0)+\sum_{k=1}^d\mu_kJ_k(\Lambda)
$,
which implies $\mu\mapsto\norm{J(\Lambda,\mu)}$ is convex. Its minimum is
attained because $\operatorname{span}\{J_k(\Lambda)\}_{k=1}^d$ is a
closed finite-dimensional subspace. Every $\mu$ already
provides a valid threshold; minimizing the norm only enlarges it. Consequently, if $\Aof$ is
Hurwitz, then for every $0<\rho<\rho^\star(\Lambda)$,
$\mathcal D(P(\Lambda),\rho)$ is contained in the \gls{roa} of the lifted system. The value
$\rho^\star(\Lambda)$ is the best threshold furnished by this norm bound and satisfies
$\rho^\star(\Lambda)\le\rho_{\mathrm{rad}}$; equality need not hold. If the minimum in
\eqref{eq:rhostarL} is zero, we set $\rho^\star(\Lambda)=+\infty$; an optimizing gauge then has
$J=0$, and \eqref{eq:vdotJ} gives
$\dot V=-\norm{\mathbf w}^2$, so every finite $\rho$ is certified. Although $H(\Lambda)$ is affine in
$\Lambda$, the dependence across
gains is generally nonconvex because $P(\Lambda)$ and $P_f(\Lambda)$ are obtained from the
Lyapunov equation \eqref{eq:lyapL}.

The lifted certificate pulls back to the original coordinates through
$\Phi(\mathbf x):=(\mathbf x,\mathbf g(\mathbf x))$. Define
\begin{equation}\label{eq:pullback}
\begin{aligned}
S(\Lambda,\rho):=\big\{\mathbf x\in\R^n:\widetilde V_\Lambda(\mathbf x)\le\rho^2\big\}, 
\quad \widetilde V_\Lambda(\mathbf x):=\Phi(\mathbf x)^{\!\top}P(\Lambda)\Phi(\mathbf x).
\end{aligned}
\end{equation}
For $0<\rho<\rho^\star(\Lambda)$, exactness and uniqueness imply that the lifted trajectory from
$\Phi(\mathbf x_0)$ is $\Phi(\mathbf x(t))$ and remains on $\M$. Since the lifted ellipsoid is
positively invariant and contained in the lifted \gls{roa}, this trajectory is global and converges
to the origin. Hence $S(\Lambda,\rho)$ is positively invariant and contained in the \gls{roa} of
\eqref{eq:sys}.

\begin{corollary}[Path-connected pullback]\label{cor:conn}
If $\Aof$ is Hurwitz and $0<\rho<\rho^\star(\Lambda)$, then $S(\Lambda,\rho)$ is path-connected;
its connected component containing the origin is therefore the entire set.
\end{corollary}

\begin{proof}
See the \hyperref[proof:conn]{proof of Corollary~\ref*{cor:conn}}
in Appendix~\ref{app:pullback-box-proofs} (p.~\pageref{proof:conn}).
\end{proof}

After optimizing the representation at each gain, the remaining choice is the extension
$\mathbf F_\Lambda$. It changes both $P(\Lambda)$ and $\rho^\star(\Lambda)$, and hence the
pullback certificate \eqref{eq:pullback}, while preserving the original dynamics on $\M$.

By Lemma~\ref{lem:split}, feasibility
reduces to the Hurwitz condition on the auxiliary block $\Dof$. An unrestricted maximizing gain need not exist:
Example~\ref{ex:unbounded-ray} below gives a member of the present stabilizer family for which
$\rho^\star(\lambda)\to+\infty$. We therefore include the gain domain, which is also the domain
used by the implementation, in the problem data.

\begin{definition}[admissible gain box]\label{def:gainbox}
An \emph{admissible gain box} is a nonempty box
$B=[\underline\Lambda,\overline\Lambda]\subset\R^m$ contained in the Hurwitz set
$\mathcal F:=\{\Lambda:\Aof\text{ is Hurwitz}\}$. Equivalently, compactness gives a uniform margin
$\gamma_B:=\min_{\Lambda\in B}-\Real\lambda_1(\Aof)>0$. A common Lyapunov inequality checked at all vertices
is a sufficient, implementation-verifiable certificate that $B\subset\mathcal F$ because $A$ is
affine in $\Lambda$.
Since $\mathcal F$ is open and nonempty, every $\Lambda^0\in\mathcal F$ lies in the interior of
some nondegenerate closed box $B\Subset\mathcal F$.
\end{definition}

\begin{problem}[box-constrained certificate optimization]\label{prob:main}
Given the fixed polynomial system $\mathbf p$, embedding $\Phi$,
weight $Q=I_N$, and lift data $(A_0,H_0,\{h_i\}_{i=1}^m)$, the derived quantities
$(W,\{s_i\}_{i=1}^m,\{K_k\}_{k=1}^d)$, and an admissible gain box $B$, solve
\begin{equation}\label{eq:problem}
\boxed{
\phi_B:=\min_{\Lambda\in B}\phi(\Lambda),
\qquad
\phi(\Lambda):=\min_{\mu\in\R^d}\norm{J(\Lambda,\mu)},
\qquad
\rho_B:=\phi_B^{-1},
}
\end{equation}
with the convention $1/0:=+\infty$. The map $J(\Lambda,\mu)$ is assembled from
\eqref{eq:AH}, \eqref{eq:Hmu}, \eqref{eq:lyapL}, and \eqref{eq:J}. Equivalently, using the
definitions above, an expanded formulation of \eqref{eq:problem} is
\begin{equation}\label{eq:problemfull}
\begin{aligned}
\min_{\substack{\Lambda\in B,\ \mu\in\R^d,\\ P\in\Sset^N,\ t\in\R}}
\quad &t\\[2pt]
\mathrm{s.t.}\quad
&\begin{bmatrix}
tI_{N^2}&J\\
J^{\!\top}&tI_N
\end{bmatrix}\succeq0,\\
&J=\big(P_f^{\!\top}\kron I_N\big)^{-1}G(\Lambda,\mu),
\qquad
P=P_f^{\!\top}P_f,\\
&\Aof^{\!\top}P+P\Aof=-I_N,
\qquad P\succ0,\\
&\Aof=A_0-\sum_{i=1}^m\lambda_i\,
\mathbf e_{n+i}\mathbf e_{n+i}^{\!\top},\\
&\Hof=H_0+\sum_{i=1}^m\lambda_i\,
\mathbf e_{n+i}s_i^{\!\top}+\sum_{k=1}^d\mu_kK_k .
\end{aligned}
\end{equation}
Its optimal value is $t_B=\phi_B$.
Proposition~\ref{prop:box-wellposed} below justifies attainment. A closed certified set is reported
with the safe radius of Corollary~\ref{cor:resid}, never with the open threshold $\rho_B$ itself.
\end{problem}

See the \hyperref[proof:problem-derivation]{derivation of the box-constrained formulation}
in Appendix~\ref{app:problem-derivation} (p.~\pageref{proof:problem-derivation}),
including the exact spectral-norm epigraph.

\begin{remark}\label{rem:forms}
The optimization variables are $\Lambda\in B$ and $\mu\in\R^d$.
The embedding, reference extension, stabilizer factorizations, Lyapunov weight, and gain box
are fixed problem data. Thus optimality in Problem~\ref{prob:main} refers to the stated
spectral-norm certificate within this family and box.
The admissible-box condition guarantees a unique symmetric solution $P(\Lambda)\succ0$ of
\eqref{eq:lyapL} at every outer point. Thus $P$ is not a free design variable; it appears in
\eqref{eq:problemfull} only to make the dependence on $\Lambda$ explicit. After eliminating $P$,
the outer dependence of $J(\Lambda,\mu)$ on $\Lambda$ combines the affine map $H(\Lambda)$ with
the nonlinear Lyapunov--Cholesky map $\Lambda\mapsto P(\Lambda)\mapsto P_f(\Lambda)$. For fixed
$\Lambda$, the problem in $(\mu,t)$ is an exact convex \gls{sdp}; the box-constrained outer problem
is generally nonsmooth and nonconvex. The objective optimizes the Lyapunov level threshold. Since
$P(\Lambda)$ also changes with the gain, it does not by itself order the volumes or set inclusions
of pullbacks obtained at different gains.
\end{remark}

The next result removes gauge directions invisible to the certificate before applying
Weierstrass' theorem.

\begin{proposition}[well-posed box value]\label{prop:box-wellposed}
The value function $\phi$ is locally Lipschitz on $\mathcal F$. Consequently, it attains a global
minimum on every admissible gain box $B$. If $\phi_B>0$, then $\rho^\star=1/\phi$ attains its box
maximum at exactly the same gains. If $\phi_B=0$, an optimizing pair satisfies $J=0$ and the
corresponding quadratic Lyapunov function is global.
\end{proposition}

\begin{proof}
See the \hyperref[proof:box-wellposed]{proof of Proposition~\ref*{prop:box-wellposed}}
in Appendix~\ref{app:box-wellposed} (p.~\pageref{proof:box-wellposed}).
\end{proof}

\begin{example}[unbounded spectral-norm threshold]\label{ex:unbounded-ray}
For $\dot x=-x-x^3$ with lift $y=x^2$, consider the stabilizer family
\begin{equation}\label{eq:unbounded-ray-system}
\dot x=-x-xy,
\qquad
\dot y=-(2+\lambda)y+\lambda x^2-2y^2,
\qquad \lambda>-2.
\end{equation}
Put $s:=\lambda+2>0$. Then $P=\diag(1/2,1/(2s))$ and
$\dot V=-(x^2+y^2)(1+2y/s)$. A gauge whose two blocks in $G$ are $0$ and $-2I_2/s$ gives
$\phi(\lambda)\le2\sqrt{2/s}$. Taking $x=0$ and $y<0$, the cubic contribution is $2|y|^3/s$,
while \eqref{eq:vdotJ} bounds it above by $|y|^3\norm{J}/\sqrt{2s}$. Hence every gauge satisfies
$\norm{J}\ge2\sqrt{2/s}$, so $\phi(\lambda)=2\sqrt{2/s}$ and
$\rho^\star(\lambda)=\sqrt{s}/(2\sqrt2)\to+\infty$. Thus the unrestricted problem need not have a
maximizing gain; a boundary box optimum should be checked by widening $B$.
\end{example}

\begin{corollary}[Numerical certificate]\label{cor:resid}
Fix a gain, a gauge satisfying $K(\mathbf w\kron\mathbf w)\equiv0$, and
$\widehat P=\widehat P^{\!\top}\succ0$. Assemble $\widehat J$ from $\widehat P$ by
\eqref{eq:J}. Suppose verified bounds satisfy
\[
\norm{\Aof^{\!\top}\widehat P+\widehat P\Aof+I_N}\le\widehat\delta<1,
\qquad \norm{\widehat J}\le\widehat j.
\]
For $\eta\in(0,1)$ and $\widehat j>0$, the ellipsoid
$\mathcal D(\widehat P,\rho_{\mathrm{safe}})$ with
$\rho_{\mathrm{safe}}:=(1-\eta)(1-\widehat\delta)/\widehat j$ is positively invariant and
contained in the lifted \gls{roa}; its pullback under $\Phi$ is contained in the \gls{roa} of
\eqref{eq:sys}. If $\widehat j=0$ is verified exactly,
$\mathbf w^{\!\top}\widehat P\mathbf w$ is a global Lyapunov function.
\end{corollary}

\begin{proof}
\leavevmode\phantomsection\label{proof:resid}
Proposition~\ref{prop:intrinsic-bound} gives
$\beta(\widehat P,I_N,\mathbf q)\le\norm{\widehat J}\le\widehat j$.
Apply Proposition~\ref{prop:ellipsoidal-certificate} with $Q=I_N$, $\theta=1-\eta$,
and the verified upper bounds $\widehat\delta,\widehat j$.
Exactness of the lift gives the pullback claim.
\end{proof}
For numerical use, symmetrize the returned $P$, verify positive definiteness, and bound the
normalized residual and the norm of the reassembled matrix. The SDP epigraph variable alone
does not bound the returned matrix. Check the gauge identity coefficientwise; any nonzero
defect must be included in the nonlinear growth bound using
\eqref{eq:app-gauge-defect-bound} in Appendix~\ref{app:verified-certificate}.
Verified bounds must account for rounding errors in assembly and norm evaluation.
Certificate validity requires these bounds, independently of solver optimality.

Problem~\ref{prob:main} reduces to the fixed-pair convex construction in
\citep[\S4]{kramer2021stability} when $m=0$ or $\Lambda$ is fixed, since only the gauge $\mu$
remains. By contrast, the doubling rule in \citep[Alg.~2]{cai2024dissipative} finds a feasible
scalar gain but optimizes neither the gauge nor the certificate over $B$.

\begin{remark}[Multiple equilibria]\label{rem:multi}
To preserve dissipativity at several equilibria $x_1^*,\dots,x_\ell^*$ simultaneously, as in
\cite{cai2024dissipative}, impose $D_j(\Lambda)$ Hurwitz for every $j$, where
\[
D_j(\Lambda)=W_j-\diag(\Lambda)L_j,\qquad
W_j=J_y\mathbf q_2(x_j^*,\mathbf g(x_j^*))
-J_x\mathbf g(x_j^*)J_y\mathbf q_1(x_j^*,\mathbf g(x_j^*)),
\]
and $L_j=J_y\mathbf h(x_j^*,\mathbf g(x_j^*))$. The second term in $W_j$ vanishes at the standing
origin because $J_x\mathbf g(\mathbf0)=0$, but it must be retained at a shifted equilibrium. The
certificate itself remains anchored at one equilibrium.
\end{remark}

\subsection{The certified volume at leading order}
\label{sec:volume}

Problem~\ref{prob:main} maximizes the lifted level radius $\rho^\star(\Lambda)$, whereas the
geometric quantity of interest is the volume of its pullback $S(\Lambda,\rho)$. These objectives
need not agree a priori because the metric $P(\Lambda)$ varies with the gain and $\Phi$ is nonlinear.
They do agree to leading order near the origin. Partition $P$ conformally with
$\mathbf w=(\mathbf x,\mathbf y)$ as
\begin{equation}\label{eq:Pblocks}
P(\Lambda)=
\begin{bmatrix}
P_{xx}(\Lambda)&P_{xy}(\Lambda)\\[2pt]
P_{xy}(\Lambda)^{\!\top}&P_{yy}(\Lambda)
\end{bmatrix},
\qquad P_{xx}(\Lambda)\in\Sset^n.
\end{equation}
Since every auxiliary monomial has degree at least two,
$\Phi(\mathbf x)=(\mathbf x,O(\norm{\mathbf x}^2))$, and hence
\begin{equation}\label{eq:vtildegerm}
\widetilde V_\Lambda(\mathbf x)
=\mathbf x^{\!\top}P_{xx}(\Lambda)\mathbf x+O(\norm{\mathbf x}^3).
\end{equation}

\begin{proposition}[Leading-order certified volume]\label{prop:vol}
For every feasible $\Lambda$, as $\rho\downarrow0$,
\begin{equation}\label{eq:volexp}
\operatorname{vol}S(\Lambda,\rho)
=\kappa_n\frac{\rho^n}{\sqrt{\det P_{xx}(\Lambda)}}+o(\rho^n),
\qquad
\kappa_n:=\frac{\pi^{n/2}}{\Gamma(n/2+1)}.
\end{equation}
\end{proposition}

\begin{proof}
A two-sided quadratic sandwich gives the result; see the
\hyperref[proof:vol]{proof of Proposition~\ref*{prop:vol}}
in Appendix~\ref{app:volume} (p.~\pageref{proof:vol}).
\end{proof}

If $\phi_B>0$, all thresholds on $B$ are finite, and the logarithm of the
leading coefficient at the scaled level $\rho=\varepsilon\rho^\star(\Lambda)$ suggests the score
\begin{equation}\label{eq:volproxy}
n\log\rho^\star(\Lambda)-\tfrac12\log\det P_{xx}(\Lambda).
\end{equation}
The apparent correction is constant across gains.

\begin{proposition}[Gain invariance of the quadratic germ]\label{prop:pxx}
Fix $Q\in\Sset^N$, $Q\succ0$, and let $P_Q(\Lambda)\succ0$ solve
$\Aof^{\!\top}P_Q+P_Q\Aof=-Q$. Then, for every feasible $\Lambda$, its physical block satisfies
\[
P_{Q,xx}(\Lambda)=\bar P_Q,
\qquad
J_p^{\!\top}\bar P_Q+\bar P_QJ_p=-Q_{xx},
\]
where $Q_{xx}$ is the physical block of $Q$. Thus $\bar P_Q$ is the unique solution and is
independent of $\Lambda$; in particular, $P_{xx}$ in \eqref{eq:Pblocks} equals the gain-invariant
matrix $\bar P_I$ obtained for $Q=I_N$.
\end{proposition}

\begin{proof}
See the \hyperref[proof:pxx]{proof of Proposition~\ref*{prop:pxx}}
in Appendix~\ref{app:quadratic-germ} (p.~\pageref{proof:pxx}).
\end{proof}

\begin{corollary}[Radius and leading-order volume]\label{cor:volopt}
If $\phi_B>0$, then, for every $\Lambda\in B$ and as $\varepsilon\downarrow0$,
\[
\operatorname{vol}S\big(\Lambda,\varepsilon\rho^\star(\Lambda)\big)
=\frac{\kappa_n\big(\rho^\star(\Lambda)\big)^n}{\sqrt{\det\bar P_I}}\,
\varepsilon^n+o(\varepsilon^n).
\]
Hence Problem~\ref{prob:main} ranks gains exactly by the leading coefficients of these scaled
certified volumes. If $\phi_B=0$, Proposition~\ref{prop:box-wellposed} instead gives an optimizing
gain with a global quadratic Lyapunov certificate.
\end{corollary}

Thus, for the fixed weight $Q=I_N$, the gain affects the leading-order volume only through the achievable radius,
not through the local shape factor. This is a local statement: finite-radius volumes and set
inclusions can still depend on the higher-order lift and on $P_{xy}$ and $P_{yy}$. For $Q=I_N$,
$\bar P_I$ is precisely the Lyapunov matrix of the plant linearization, so all lifts share the plant's
quadratic germ. Within this family, changing that germ requires varying the normalized weight
$Q_{xx}$, as anticipated in Remark~\ref{rem:Q}.

\section{Solving the fixed-family certificate problem}
\label{sec:alg}

Problem~\ref{prob:main} has one convex and one nonconvex level. For fixed $\Lambda$, the Lyapunov
equation fixes $P(\Lambda)$ and the gauge is optimized by an exact \gls{sdp}; the remaining
box-constrained value function $\phi$ is locally Lipschitz but generally nonsmooth and nonconvex.
This section develops the corresponding value-and-derivative oracle and a gradient-sampling
L-BFGS method. The lift, stabilizer factorizations, and admissible box are fixed throughout.

\subsection{The exact reduced inner oracle}
\label{sec:alg-inner}

Fix $\Lambda\in\mathcal F$ and put $J_0:=J(\Lambda,0)$.
The linear certificate gauge map and its image are
$\mathcal T_\Lambda\mu:=J(\Lambda,\mu)-J_0$ and
$\mathcal W_\Lambda:=\operatorname{im}\mathcal T_\Lambda\subset\R^{N^2\times N}$.
Write $\langle U,V\rangle_F:=\operatorname{tr}(U^{\!\top}V)$ and let $\norm{Z}_*$ denote
the nuclear norm, the sum of the singular values of $Z$.

\begin{proposition}[Gauge-image distance and duality]\label{prop:inner-distance-dual}
For every fixed $\Lambda\in\mathcal F$,
\begin{equation}\label{eq:inner-distance}
\phi(\Lambda)
=\min_{W\in\mathcal W_\Lambda}\norm{J_0+W}
=\operatorname{dist}_{\norm{\cdot}}(J_0,\mathcal W_\Lambda),
\end{equation}
where the distance is induced by the spectral norm. Its dual is
\begin{equation}\label{eq:inner-distance-dual}
\phi(\Lambda)=
\max_{\substack{Z\in\R^{N^2\times N}\\
\norm{Z}_*\le1,\ Z\in\mathcal W_\Lambda^\perp}}
\langle Z,J_0\rangle_F.
\end{equation}
Here $Z\in\mathcal W_\Lambda^\perp$ is equivalent to
$\langle Z,\mathcal T_\Lambda\mathbf e_k\rangle_F=0$ for $k=1,\dots,d$.
Both optima are attained. A matrix $J^\star\in J_0+\mathcal W_\Lambda$ is optimal if and
only if there exists $Z^\star\in\mathcal W_\Lambda^\perp$ with
$\norm{Z^\star}_*\le1$ and $\langle Z^\star,J^\star\rangle_F=\norm{J^\star}$;
equivalently,
$\partial\norm{\cdot}(J^\star)\cap\mathcal W_\Lambda^\perp\ne\varnothing$.
For any gauge $\mu$ and any dual-feasible $Z$,
\begin{equation}\label{eq:inner-duality-gap}
0\le\norm{J(\Lambda,\mu)}-\phi(\Lambda)
\le\norm{J(\Lambda,\mu)}-\langle Z,J_0\rangle_F.
\end{equation}
Thus a zero primal--dual gap certifies global inner optimality.
\end{proposition}

\begin{proof}
See the \hyperref[proof:inner-distance-dual]{proof of Proposition~\ref*{prop:inner-distance-dual}}
in Appendix~\ref{app:inner-distance-dual} (p.~\pageref{proof:inner-distance-dual}).
\end{proof}
This formulation depends on the gauge image, not its basis; it still optimizes a
spectral-norm upper bound on $\beta$. The distance dual and gap bound are standard convex
norm duality. The quadratization-specific step is identifying $\mathcal W_\Lambda$ and
removing its redundant gauge coordinates.

The proof of Proposition~\ref{prop:box-wellposed} identifies the $P$-independent space
$\mathcal C$ seen by the certificate. Fix a basis
$(B_1^{(\ell)},\dots,B_N^{(\ell)})$, $\ell=1,\dots,q$, of $\mathcal C$, where
$q=N(N^2-1)/3$, and define
\begin{equation}\label{eq:reduced-oracle}
\begin{aligned}
\mathcal B_\ell&:=\operatorname{col}_{r=1}^N B_r^{(\ell)},
&X(\Lambda)&:=P_f(\Lambda)^{\!\top}\kron I_N,\\
G_{\rm r}(\Lambda,\nu)&:=G(\Lambda,0)+\sum_{\ell=1}^q\nu_\ell\mathcal B_\ell,
&J_{\rm r}(\Lambda,\nu)&:=X(\Lambda)^{-1}G_{\rm r}(\Lambda,\nu).
\end{aligned}
\end{equation}
The basis matrices $\mathcal B_\ell$ are independent of $\Lambda$ and are computed once.
Their images $J_\ell(\Lambda):=X(\Lambda)^{-1}\mathcal B_\ell$
form a basis of $\mathcal W_\Lambda$, so the reduced coordinates describe the same distance
problem.

\begin{proposition}[Exact reduced inner oracle]\label{prop:reduced-inner}
For every $\Lambda\in\mathcal F$,
\begin{equation}\label{eq:inner-generalized-schur}
\boxed{
\begin{aligned}
\phi(\Lambda)^2=\tau^\star(\Lambda)
:={}&\min_{\nu\in\R^q,\,\tau\in\R}\ \tau\\
&\mathrm{s.t.}\quad
F_\Lambda(\nu,\tau):=
\begin{bmatrix}
P(\Lambda)\kron I_N&G_{\rm r}(\Lambda,\nu)\\
G_{\rm r}(\Lambda,\nu)^{\!\top}&\tau I_N
\end{bmatrix}\succeq0 .
\end{aligned}}
\end{equation}
The problem is strictly feasible and attains its optimum. If $\nu^\star$ is returned, a raw gauge
representative is recovered from the minimum-norm linear problem
\begin{equation}\label{eq:raw-gauge-reconstruction}
\mu^\star=\operatorname*{arg\,min}_{\mu\in\R^d}\ \tfrac12\norm{\mu}^2
\quad\mathrm{s.t.}\quad
\operatorname{col}_{r=1}^N\!\left(S_r(\mu)^{\!\top}P+PS_r(\mu)\right)
=\sum_{\ell=1}^q\nu_\ell^\star\mathcal B_\ell,
\end{equation}
where $S_r(\mu)$ are the slices of $K(\mu)$ in \eqref{eq:Hmu}. This reconstruction leaves the
quadratic vector field and the value $\phi(\Lambda)$ unchanged.
\end{proposition}

\begin{proof}
Since $P\kron I_N=XX^{\!\top}$, the Schur complement of
\eqref{eq:inner-generalized-schur} is
$G_{\rm r}^{\!\top}(P\kron I_N)^{-1}G_{\rm r}\preceq\tau I_N$, equivalently
$J_{\rm r}^{\!\top}J_{\rm r}\preceq\tau I_N$. Thus the reduced generalized-Schur formulation is
exactly the fixed-$\Lambda$ spectral-norm problem, without placing $X^{-1}$ in the cone model.
For the full argument, including attainment and reconstruction, see the
\hyperref[proof:reduced-inner]{proof of Proposition~\ref*{prop:reduced-inner}}
in Appendix~\ref{app:reduced-inner} (p.~\pageref{proof:reduced-inner}).
\end{proof}

\subsection{Inner optimality and value-function sensitivities}
\label{sec:alg-sensitivity}

Let $(\nu^\star,\tau^\star)$ solve \eqref{eq:inner-generalized-schur} and put
$J^\star:=J_{\rm r}(\Lambda,\nu^\star)$.
If $J^\star\ne0$ and $U_1,V_1$ contain paired orthonormal leading singular vectors, so that
$J^\star V_1=\phi(\Lambda)U_1$, then \citep{watson1992subdifferential}
\begin{equation}\label{eq:spectral-subdifferential}
\partial\norm{\cdot}(J^\star)
=\left\{U_1\Theta V_1^{\!\top}:\Theta\succeq0,\ \operatorname{tr}\Theta=1\right\};
\end{equation}
at $J^\star=0$ it is the nuclear-norm unit ball. Convex optimality of the inner problem is therefore
equivalent to the existence of $Z^\star\in\partial\norm{\cdot}(J^\star)$ satisfying
\begin{equation}\label{eq:inner-spectral-kkt}
\left\langle Z^\star,J_\ell(\Lambda)\right\rangle_F=0,\qquad \ell=1,\dots,q,
\end{equation}
with the condition vacuous when $q=0$.

The gain derivatives can be evaluated without finite differences. Put
$E_i:=\mathbf e_{n+i}\mathbf e_{n+i}^{\!\top}$ and use a dot for
$\partial_{\lambda_i}$ while holding $\nu=\nu^\star$ fixed. From \eqref{eq:AH} and
\eqref{eq:lyapL},
\begin{equation}\label{eq:lambda-sensitivities}
\begin{aligned}
\dot A_i&=-E_i,
&\dot H_i&=\mathbf e_{n+i}s_i^{\!\top},\\
\Aof^{\!\top}\dot P_i+\dot P_i\Aof&=E_iP+PE_i,
&\dot P_i&=\dot P_{f,i}^{\!\top}P_f+P_f^{\!\top}\dot P_{f,i}.
\end{aligned}
\end{equation}
Let $M_r^0$ be the $r$th slice of $H(\Lambda)$ and $\dot M_{r,i}$ the corresponding slice of
$\dot H_i$. Differentiating \eqref{eq:reduced-oracle} gives
\begin{equation}\label{eq:GJ-sensitivities}
\begin{aligned}
\dot G_i^0
&=\operatorname{col}_{r=1}^N\!\left(
\dot M_{r,i}^{\!\top}P+(M_r^0)^{\!\top}\dot P_i
+\dot P_iM_r^0+P\dot M_{r,i}\right),\\
\dot J_i
&=X^{-1}\!\left[\dot G_i^0-
(\dot P_{f,i}^{\!\top}\kron I_N)J^\star\right].
\end{aligned}
\end{equation}
The superscript in $G_i^0$ emphasizes that the fixed reduced directions $\mathcal B_\ell$ have
zero gain derivative.

\begin{proposition}[Regularity and envelope sensitivities]\label{prop:regular-sensitivity}
Let $\mathcal U\subset\mathcal F$ be open and let $A,H$ be the affine maps in \eqref{eq:AH}.
For any fixed $Q\in\Sset^N$, $Q\succ0$, the solution $P_Q(\Lambda)\succ0$ of
$A^{\!\top}P_Q+P_QA=-Q$ is real analytic on $\mathcal U$, and
\begin{equation}\label{eq:lyapunov-general-sensitivity}
A^{\!\top}P_{Q,i}+P_{Q,i}A
=-\left(A_i^{\!\top}P_Q+P_QA_i\right),\qquad
P_{Q,i}:=\partial_{\lambda_i}P_Q,\quad A_i:=\partial_{\lambda_i}A.
\end{equation}
For the fixed weight $Q=I_N$ in Problem~\ref{prob:main}, $P=P_{I_N}$, its positive-diagonal
Cholesky factor $P_f$, and $J_{\rm r}$ depend real analytically on the gain.
The value $\phi$ is locally Lipschitz, hence continuous, on $\mathcal U$, and its reduced
inner minimizers are locally uniformly bounded.

Fix $\Lambda\in\mathcal U$, let $(\nu^\star,\tau^\star)$ be a primal optimum and
$Y^\star\succeq0$ a dual optimum of \eqref{eq:inner-generalized-schur}, and put
$J^\star=J_{\rm r}(\Lambda,\nu^\star)$.
Partition $Y^\star=(Y_{ab}^\star)_{a,b=1}^2$ conformally with $F_\Lambda$.
The KKT conditions include
\begin{equation}\label{eq:inner-conic-kkt}
\operatorname{tr}Y_{22}^\star=1,\qquad
\left\langle Y_{12}^\star,\mathcal B_\ell\right\rangle_F=0\ (\ell=1,\dots,q),
\qquad
\left\langle Y^\star,F_\Lambda(\nu^\star,\tau^\star)\right\rangle_F=0.
\end{equation}
If $\phi(\Lambda)>0$, then
\begin{equation}\label{eq:normalized-inner-dual}
Z^\star:=-\frac{X^{\!\top}Y_{12}^\star}{\phi(\Lambda)}
\in\partial\norm{\cdot}(J^\star),\qquad \norm{Z^\star}_*=1,
\end{equation}
and $Z^\star$ solves \eqref{eq:inner-distance-dual} and satisfies
\eqref{eq:inner-spectral-kkt}. The vector $g=(g_1,\dots,g_m)$ with
\begin{equation}\label{eq:dual-envelope-subgradient}
g_i=\left\langle Z^\star,\dot J_i\right\rangle_F
=-\frac{
\left\langle Y_{11}^\star,\dot P_i\kron I_N\right\rangle_F
+2\left\langle Y_{12}^\star,\dot G_i^0\right\rangle_F}{2\phi(\Lambda)}
\end{equation}
belongs to $\partial_C\phi(\Lambda)$ and equals $\nabla\phi(\Lambda)$ wherever $\phi$ is
differentiable. If, in addition, the reduced minimizer $\nu^\star$ is unique and the largest
singular value of $J^\star$ is simple, with paired unit singular vectors $u,v$, then $\phi$ is differentiable at $\Lambda$
and
\begin{equation}\label{eq:regular-envelope}
\partial_{\lambda_i}\phi(\Lambda)=g_i=u^{\!\top}\dot J_i v,
\qquad i=1,\dots,m.
\end{equation}
\end{proposition}

\begin{proof}
See the \hyperref[proof:regular-sensitivity]{proof of Proposition~\ref*{prop:regular-sensitivity}}
in Appendix~\ref{app:inner-sensitivity} (p.~\pageref{proof:regular-sensitivity}).
\end{proof}
The differentiability criterion concerns
uniqueness of the reduced minimizer $\nu$; raw gauge representatives can remain nonunique.
The partial derivative in \eqref{eq:regular-envelope} holds $\nu=\nu^\star$ fixed.
The Clarke inclusion uses the fixed
reduced basis $\mathcal B_\ell$, which makes the conic dual feasible set independent of
$\Lambda$; it requires neither a unique inner minimizer nor a simple leading singular value.
At $\phi(\Lambda)=0$, $Z=0$ is a distance-dual optimum, but the normalization
\eqref{eq:normalized-inner-dual} is undefined. Approximate numerical multipliers remain
sensitivity candidates unless the stated optimality conditions are established.

A single Clarke subgradient does not generally certify descent,
so the outer method samples nearby ordinary gradients. Let $\mathcal R\subset\mathcal U$ be the set where
$\phi$ is differentiable. For $B\Subset\mathcal U\Subset\mathcal F$, define
\begin{equation}\label{eq:clarke-epsilon-hull}
\partial_\epsilon\phi(\Lambda):=\operatorname{cl\,co}
\left\{\nabla\phi(z):z\in\mathbb B_\epsilon(\Lambda)\cap\mathcal U\cap\mathcal R\right\},
\qquad
\partial_C\phi(\Lambda)=\bigcap_{\epsilon>0}\partial_\epsilon\phi(\Lambda).
\end{equation}
The equality follows from local Lipschitz continuity and is the target approximated by a finite
bundle of nearby ordinary gradients.

\subsection{Box-constrained GS--L-BFGS-B}
\label{sec:alg-outer}

Fixed gain coordinates can be removed, so we assume here that $B$ is nondegenerate. For
conditioning, the implementation uses
$D_B:=\diag(\overline\Lambda-\underline\Lambda)$,
$\theta:=D_B^{-1}(\Lambda-\underline\Lambda)\in C:=[0,1]^m$, and
$\psi(\theta):=\phi(\underline\Lambda+D_B\theta)$, so
$\nabla\psi(\theta)=D_B^{\!\top}\nabla\phi(\Lambda)$. Let
$\mathcal U_B:=D_B^{-1}(\mathcal U-\underline\Lambda)$. At $\theta^k$, continuously sample
\begin{equation}\label{eq:gradient-bundle}
z_{kj}=\theta^k+\epsilon_k\xi_{kj},\qquad
\xi_{kj}\sim\operatorname{Unif}(\mathbb B_1),\qquad
\mathcal G_k:=\left\{g_{kj}=\nabla\psi(z_{kj})\right\}_{j=1}^{p_k},\qquad p_k\ge m+1,
\end{equation}
retaining only samples in $\mathcal U_B$ for which the derivative oracle is accepted. In the ideal
algorithm below these are ordinary gradients; the finite implementation uses dual-envelope
candidates that pass its numerical consistency checks. The points $z_{kj}$ need not lie in $C$,
and the deterministic center $\theta^k$ is not a bundle atom.

The L-BFGS memory supplies a Hessian metric. For a bundle-derived secant pair $(s,r)$ and
$M\succ0$, Powell damping with $\kappa\in(0,1)$ uses
\begin{equation}\label{eq:damped-bfgs}
\begin{aligned}
\vartheta&=
\begin{cases}
1,&s^{\!\top}r\ge\kappa s^{\!\top}Ms,\\
\displaystyle\frac{(1-\kappa)s^{\!\top}Ms}{s^{\!\top}Ms-s^{\!\top}r},&
s^{\!\top}r<\kappa s^{\!\top}Ms,
\end{cases}\\
\widetilde r&=\vartheta r+(1-\vartheta)Ms,\\
M^+&=M-\frac{Mss^{\!\top}M}{s^{\!\top}Ms}
+\frac{\widetilde r\widetilde r^{\!\top}}{s^{\!\top}\widetilde r}.
\end{aligned}
\end{equation}
Here $s$ is the most recent accepted displacement and $r$ is the change from the preceding
QP-dual aggregate to the next bundle mean. Only the most recent pairs are retained. Spectral
clipping of the rebuilt matrix enforces
$0<\underline bI\preceq M_k\preceq\overline bI$.

For $x\in C$, use the convex-analysis convention
$N_C(x):=\{n:n^{\!\top}(y-x)\le0\ \text{for all }y\in C\}$. Given $\mathcal G_k$ and $M_k$,
the finite-displacement model is
\begin{equation}\label{eq:bundleqp}
\begin{aligned}
\min_{d\in\R^m,\,\zeta\in\R}\quad&\zeta+\tfrac12d^{\!\top}M_kd\\
\mathrm{s.t.}\quad&g_{kj}^{\!\top}d\le\zeta,\quad j=1,\dots,p_k,\\
&-\theta^k\le d\le\mathbf1-\theta^k.
\end{aligned}
\end{equation}

\begin{lemma}[Bundle model and residual]\label{lem:bundle-model}
Problem~\eqref{eq:bundleqp} has a unique displacement $d_k$. There exist
$\alpha_k\in\R_+^{p_k}$, $\mathbf1^{\!\top}\alpha_k=1$,
$\bar g_k:=\sum_j\alpha_{kj}g_{kj}$, and
$n_k\in N_C(\theta^k+d_k)$ such that
\begin{equation}\label{eq:bundle-kkt}
M_kd_k+\bar g_k+n_k=0,\qquad
\zeta_k=\max_jg_{kj}^{\!\top}d_k=\bar g_k^{\!\top}d_k
=-d_k^{\!\top}M_kd_k-n_k^{\!\top}d_k
\le-d_k^{\!\top}M_kd_k.
\end{equation}
Consequently, with
\begin{equation}\label{eq:bundle-residual}
\chi_k:=\norm{d_k}_{M_k}:=(d_k^{\!\top}M_kd_k)^{1/2},
\end{equation}
one has $d_k=0$ if and only if
$0\in\operatorname{co}\mathcal G_k+N_C(\theta^k)$; if $d_k\ne0$, every bundle atom satisfies
$g_{kj}^{\!\top}d_k\le-\chi_k^2<0$.
\end{lemma}

\begin{proof}
Strong convexity gives uniqueness of $d_k$. The KKT multiplier of the epigraph inequalities lies
in the simplex, complementarity gives $\zeta_k=\bar g_k^{\!\top}d_k$, and the box multipliers form
$n_k$. Since the zero displacement is feasible, the normal-cone inequality gives
$n_k^{\!\top}d_k\ge0$, proving
\eqref{eq:bundle-kkt}; setting $d_k=0$ gives the stated equivalence.
\end{proof}

For $\chi_k>\omega_k$, define the first acceptable backtracking index and the next iterate by
\begin{equation}\label{eq:armijo-transition}
\begin{aligned}
\ell_k&:=\min\Bigl\{\ell\in\{0,\dots,L_k-1\}:
\psi(\theta^k+\beta^\ell d_k)
\le\psi(\theta^k)-c\beta^\ell\chi_k^2\Bigr\},\\
\theta^{k+1}&=
\begin{cases}
\theta^k+\beta^{\ell_k}d_k,&\ell_k\text{ exists},\\
\theta^k,&\text{otherwise}.
\end{cases}
\end{aligned}
\end{equation}
The endpoint bounds in \eqref{eq:bundleqp} make every trial feasible. A missing $\ell_k$ is a null
step, not a stationarity test.

\begin{algorithm}[H]
\caption{Box-constrained GS--L-BFGS-B}\label{alg:outer}
\begin{algorithmic}[1]
\REQUIRE $\theta^0\in C$, $c,\beta,\kappa\in(0,1)$, $p_k\ge m+1$, radii $\epsilon_k$,
tolerances $\omega_k$, and caps $L_k$
\REQUIRE $0<\underline b\le\overline b$ and a finite L-BFGS memory cap
\STATE Evaluate the inner oracle at $\theta^0$ and initialize the set of screened records
$\mathcal E$.
\FOR{$k=0,1,\ldots$}
  \IF{$\psi(\theta^k)=0$ is verified exactly}
    \STATE Return the corresponding global Lyapunov certificate.
  \ENDIF
  \STATE Sample \eqref{eq:gradient-bundle} until $p_k\ge m+1$ accepted derivatives are obtained;
  retain every residual-valid in-box value record in $\mathcal E$.
  \STATE Update and safeguard $M_k$ from the most recent serious-step secant pair; solve
  \eqref{eq:bundleqp} and compute $\chi_k$.
  \IF{$\chi_k\le\omega_k$}
    \STATE Set $\theta^{k+1}\gets\theta^k$, reduce $(\epsilon_k,\omega_k)$, and reset the metric.
  \ELSE
    \STATE Apply \eqref{eq:armijo-transition}; after a serious step retain its secant data, and
    after a null step reduce $\epsilon_k$ and reset the metric.
  \ENDIF
  \STATE Add every residual-valid in-box line-search record to $\mathcal E$.
\ENDFOR
\STATE \textbf{return} the element of $\mathcal E$ with smallest recomputed spectral norm.
\end{algorithmic}
\end{algorithm}

To make the anytime return rule precise, let $\mathcal A_K$ contain every completed value-oracle
record through iteration $K$, including sampled points and rejected Armijo trials, and define
\begin{equation}\label{eq:anytime-return}
\mathcal E_K:=\left\{a\in\mathcal A_K:\Lambda(a)\in B,
\ a\text{ passes Corollary~\ref{cor:resid}}\right\},\quad
\widehat a_K\in\operatorname*{arg\,min}_{a\in\mathcal E_K}\widehat j(a),
\end{equation}
where $\widehat j(a)$ is the independently recomputed norm of $J$ stored with the record. Writing
$(\widehat\Lambda,\widehat P):=(\Lambda(\widehat a_K),P(\widehat a_K))$, the reported set is
\begin{equation}\label{eq:anytime-safe-set}
S_{\rm safe}:=\left\{\mathbf x:
\Phi(\mathbf x)^{\!\top}\widehat P\Phi(\mathbf x)
\le\rho_{\rm safe}(\widehat a_K)^2\right\}.
\end{equation}
The record $\widehat a_K$ may differ from the iterate at which $\chi_k$ is reported; a sample with
an unusable derivative can still enter $\mathcal E_K$ when its value and certificate pass their
independent checks. If $\mathcal E_K$ is empty, the finite run reports failure rather than a
certificate.

\subsection{Cost and conditional guarantee}
\label{sec:alg-guarantee}

An uncached generic value-oracle call requires one $N\times N$ Lyapunov solve, one Cholesky
factorization, and the \gls{sdp} \eqref{eq:inner-generalized-schur}, which has $q+1=O(N^3)$ scalar
variables and a positive-semidefinite block of order $N^2+N$. A direct full sensitivity evaluation
adds $m$ Lyapunov solves with the same coefficient operator. The bundle problem has $m+1$ variables
and $p_k+2m$ linear constraints. If an iteration dispatches $n_k^{\rm samp}$ gradient-sample
attempts and performs $n_k^{\rm ls}$ line-search evaluations, it uses
$n_k^{\rm samp}+n_k^{\rm ls}$ inner-oracle calls before caching, with
$p_k\le n_k^{\rm samp}$ because rejected derivatives still consume an evaluation.

\begin{theorem}[Conditional box-Clarke stationarity]\label{thm:outer-clarke}
Consider the ideal sequence generated by Algorithm~\ref{alg:outer}. Assume exact inner values and
ordinary gradients, $\epsilon_k,\omega_k\downarrow0$, $L_k\to\infty$, and
$0<\underline bI\preceq M_k\preceq\overline bI$. Suppose the finite hulls
$\widehat\partial_k:=\operatorname{co}\mathcal G_k$ are gradient-consistent: along every
subsequence $\theta^{k_r}\to\bar\theta$,
\begin{equation}\label{eq:gradient-consistency}
\operatorname{dist}_{\rm H}\!\left(\widehat\partial_{k_r},
\partial_C\psi(\bar\theta)\right)\longrightarrow0.
\end{equation}
Here $\operatorname{dist}_{\rm H}$ is the Hausdorff distance.
Then $\theta^k\in C$, $\psi(\theta^k)$ is nonincreasing, and every accumulation point satisfies
\begin{equation}\label{eq:box-clarke-stationarity}
0\in\partial_C\psi(\bar\theta)+N_C(\bar\theta).
\end{equation}
Equivalently, for $\bar\Lambda=\underline\Lambda+D_B\bar\theta$,
$0\in\partial_C\phi(\bar\Lambda)+N_B(\bar\Lambda)$. If $\phi_B>0$, this is also Clarke
stationarity for maximizing $\rho^\star=1/\phi$ on $B$.
\end{theorem}

\begin{proof}
See the \hyperref[proof:outer-clarke]{proof of Theorem~\ref*{thm:outer-clarke}}
in Appendix~\ref{app:outer-clarke} (p.~\pageref{proof:outer-clarke}).
\end{proof}

Condition
\eqref{eq:gradient-consistency} is an explicit asymptotic assumption; a fixed bundle of $m+1$
samples does not imply it \citep{burke2005gradient}. Hence the theorem gives neither global
optimality nor a convergence rate. With finite bundles, solver tolerances, and line-search caps,
$\chi_k$ is only a sampled-bundle diagnostic; validity of the returned closed set follows instead
from the independent screen in Corollary~\ref{cor:resid}.

\section{Numerical experiments}
\label{sec:num}

\paragraph{Computational environment}
All experiments were run in CPython~3.11.12 under macOS~26.6.1 (arm64) on an Apple M2 Pro with
ten CPU cores and 16\,GB of RAM.  The implementation used \textsc{DQbee}~0.3.0
\citep{cai2024dissipative}, NumPy~1.26.4, SciPy~1.17.1, and SymPy~1.14.0.  The conic programs in our
optimal dissipative quadratization (\textsc{ODQ}) method were modeled in \textsc{CVXPY}~1.7.5 and
solved numerically by \textsc{Mosek}~11.2.1
\citep{cvxpy2016,mosek}.  The \gls{sos} baselines, applied directly in the original coordinates, used
\textsc{SumOfSquares.py}~1.3.1 \citep{yuan2024sumofsquares} through \textsc{PICOS}~2.6.2 and the
same \textsc{Mosek} backend. 

\paragraph{Compared methods and implementation}
For each system, \textsc{DQbee} generated one monomial lift and its inner-quadratic factorizations,
which were fixed for all lifted methods with $Q=I_N$.  \textsc{ODQ} solves
Problem~\ref{prob:main} using Algorithm~\ref{alg:outer}, searching over
the stabilizer gain $\Lambda$ with a convex gauge solve at each gain.  In the benchmark suite, \emph{Base} uses
the \textsc{DQbee} feasibility gain $\Lambda_{\mathrm{feas}}$ with $\mu=0$, whereas \emph{Gauge}
fixes $\Lambda_{\mathrm{feas}}$ and optimizes only $\mu$; the first experiment instead uses a
matched \emph{Gauge} arm at $\Lambda=0$.  Thus all lifted comparisons share the same lift and $Q$,
with no search over either.

The direct \gls{sos} baselines act on the original system.  \textsc{sos}-2 fixes the quadratic
Lyapunov function from the linearization and maximizes only its certified level.
\textsc{sos}-4($r$) searches a quartic Lyapunov function using at most $r$ alternations between
$V$ and the SOS multiplier $s$, followed by a final level search.  Hence $2$ and $4$ denote
$\deg V$, while $r$ is a round budget, not a radius or ratio; a larger $r$ need not give a nested
or larger set.  We compare the audited closed sets in the original coordinates, not the raw values
of $\rho_{\mathrm{safe}}$ or the SOS level $c$, whose scalings are method dependent.

Each \textsc{ODQ} oracle call used the reduced generalized-Schur program
\eqref{eq:inner-generalized-schur}.  Gain-independent sparse operators were cached, and a
parameterized \textsc{CVXPY}--\textsc{Mosek} model was reused and warm-started across gains.  The
common outer settings were $m+1$ full-ball samples, five L-BFGS pairs, metric eigenvalues in
$[10^{-3},10^3]$, and halving updates of the sampling radius $\epsilon$ and stationarity tolerance
$\omega$ from $(\epsilon_0,\omega_0)=(5\times10^{-2},10^{-2})$ to
$(\epsilon_{\min},\omega_{\min})=(10^{-4},10^{-4})$.  Armijo backtracking used
$(c,\beta)=(10^{-4},1/2)$ and at most $20$ trials; the scaled bundle-QP feasibility and KKT
tolerance was $2\times10^{-5}$.  Gain boxes and computational budgets are specified with each
experiment.

\paragraph{Verification and timing}
For $m>0$, the gain box and its initially $5\%$-padded sampling neighborhood passed a
common-Lyapunov vertex test with requested margin $10^{-7}$; the padding was halved if necessary.
Lifted certificates used Corollary~\ref{cor:resid} with $\eta=10^{-6}$, and the nominal open
threshold was diagnostic only.  Each certificate was independently reassembled, with the
floating-point estimates of $\norm{J}$ and the normalized Lyapunov residual moved upward by one
unit in the last place.  These are residual-aware floating-point checks, not interval or rational
proofs.

The direct \gls{sos} programs used $\varepsilon_{\mathrm{SOS}}=10^{-4}$ and the smallest even
multiplier degree at least $\max\{2,\deg f-1\}$.  A reported set $\{V\le\alpha^2c\}$ had to pass an
independent audit of solver status, coefficient residuals, and Gram-matrix semidefiniteness at
tolerance $10^{-7}$; we used the largest passing
$\alpha\in\{1-10^{-k}:k=3,\ldots,6\}\cup\{0.995,0.99,0.98\}$.
Construction time includes method-specific construction and auditing, \textsc{DQbee} for lifted
methods, and the stable-box test for \textsc{ODQ}.  The planar experiment uses one untimed warm-up
per solver path; in the heterogeneous suite, each fresh-process execution group performs the
corresponding untimed warm-up.  Reported construction times exclude process startup, warm-up,
geometry, and serialization, whereas the 600-second group deadline includes them.  The relay
protocol instead reports launch-to-serialization wall time.  Monte Carlo checks, box widening, and
plotting are excluded throughout.

\paragraph{Experimental questions}
The planar-quintic experiment isolates gain selection by comparing \textsc{ODQ} with the matched
$\Lambda=0$ \emph{Gauge} arm, and then compares both with \textsc{sos}-2 and
\textsc{sos}-4($r$) in certified area and construction time.  The second experiment
compares Base, Gauge, and \textsc{ODQ} on a heterogeneous 19-system suite and tests the
resulting structural hypothesis on a separately locked 12-system relay family.  It includes
$m=0$ controls and direct-\gls{sos} references.  The table retains separate cohort blocks because
the timing boundaries differ, while the discussion considers the evidence jointly.

\subsection{The pipeline end to end: a planar quintic}
\label{sec:exp-pipeline}

Consider
\begin{equation}\label{eq:quintic}
\dot x_1=-x_1+x_2,\qquad \dot x_2=-2x_2-x_1^5 .
\end{equation}
Its linearization has eigenvalues $-1$ and $-2$, while
$V_0=x_1^6+3x_2^2$ satisfies $\dot V_0=-6x_1^6-12x_2^2$.  Thus the origin is globally
asymptotically stable, and the finite areas below measure certificate conservatism rather than the
true basin.

\textsc{DQbee} selects $y_1=x_1^2$ and $y_2=x_1^3$, giving, before stabilizer injection,
\begin{equation}\label{eq:quintic-lift}
\dot x_1=-x_1+x_2,\quad
\dot x_2=-2x_2-y_1y_2,\quad
\dot y_1=-2y_1+2x_1x_2,\quad
\dot y_2=-3y_2+3x_2y_1 .
\end{equation}
The stabilizers are $h_1=y_1-x_1^2$ and $h_2=y_2-x_1y_1$; hence $N=4$, $m=2$, and
$q=20$.  We solve Problem~\ref{prob:main} on $B=[0,12]^2$, with caps of $60$ outer iterations and
$300$ oracle calls, and compare with the matched $\Lambda=0$ Gauge arm, \textsc{sos}-2, and
\textsc{sos}-4($r$) for $r\in\{2,4,6,8,10\}$.  The padded box passed the common stability screen.
Areas are those of the full origin-connected pullbacks; an $8001$-slice integration agreed with
$10^6$-sample Monte Carlo estimates to within $0.081\%$.

\begin{figure*}[t]
\centering
\includegraphics[width=\textwidth]{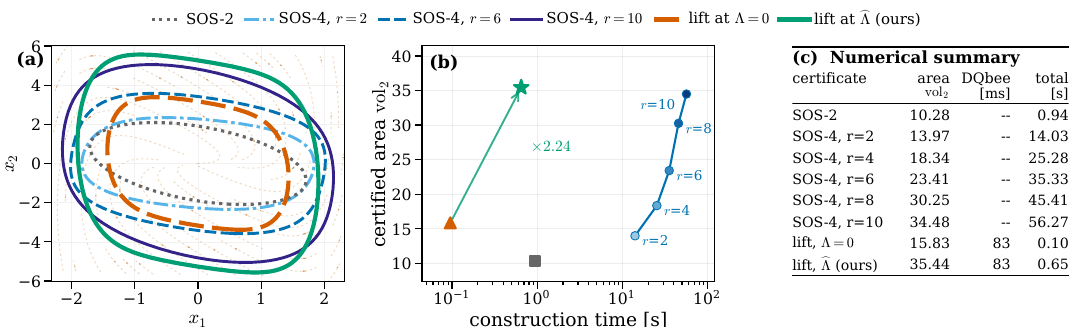}
\caption{Planar-quintic comparison.  (a) Certified boundaries in the original coordinates;
representative \gls{sos}-4 budgets are shown.  (b) Certified area versus construction time for
all eight runs.  (c) Numerical summary.  Panel~(a) uses independent axis scales for readability;
all areas are for the full origin-connected components, and all times follow the common convention
above.}
\label{fig:quintic}
\end{figure*}

Figure~\ref{fig:quintic} summarizes the results.  \textsc{ODQ} terminates after $21$ outer
iterations and $135$ oracle calls at
$\widehat\Lambda=(0.7337,1.5003)$, with $\rho_{\mathrm{safe}}=3.040679$, area $35.44$, and total
time $0.65$\,s.  At $\Lambda=0$, separate gauge optimization gives
$\rho_{\mathrm{safe}}=1.832534$, area $15.83$, and time $0.10$\,s.  Gain optimization therefore
increases the certified area by a factor of $2.238$ with the lift and $Q$ fixed.  The largest direct
certificate tested, \textsc{sos}-4($10$), has area $34.48$ and takes $56.27$\,s.  In this single
post-warm-up execution, \textsc{ODQ} is thus $2.8\%$ larger and about $87$ times faster.

Repeating the search on $[0,24]^2$ and $[0,48]^2$ again returns an interior gain and changes
$\rho_{\mathrm{safe}}$ by at most $6.8\times10^{-8}$.  This indicates insensitivity to the chosen
upper endpoint for this instance, without asserting an unrestricted optimum.

Figure~\ref{fig:quintic-landscape} compares the objective with area at 145 fixed gains:
an $11\times11$ grid on $[0,12]^2$ and a $5\times5$ grid on $[0,3]^2$, sharing the origin.
The rank correlation between $1/\phi$ and area is $0.981$.
The sampled gain $(0.75,1.5)$ has area $35.4478$, slightly above the saved ODQ area $35.4405$,
although its objective is $0.010\%$ larger. This finite-radius difference is consistent with
the leading-order result in Section~\ref{sec:volume}: minimizing $\phi$ need not maximize area
at the reported level.
The grid establishes neither optimum over the box.
\begin{figure}[t]
\centering
\includegraphics[width=\textwidth]{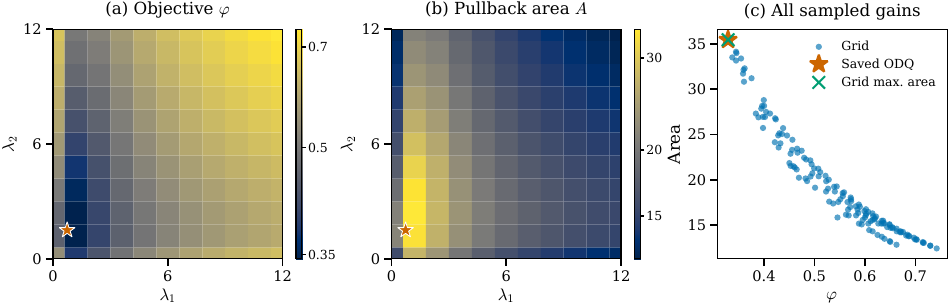}
\caption{Quintic objective and original-coordinate area. Panels (a)--(b) show values on the
$11\times11$ full-box grid; panel (c) includes the 25 local-grid points, with the shared origin
counted once. Stars mark the saved ODQ gain, and the cross marks the largest sampled area.
Each point uses the residual-aware safe radius; colors in (a)--(b) use logarithmic scales.
These are sampled values, not continuous bounds between grid points.}
\label{fig:quintic-landscape}
\end{figure}

\subsection{Systematic comparison across a benchmark suite}
\label{sec:exp-compare}

\paragraph{Benchmark design}
We use two complementary, source-locked cohorts with three repeats each: a heterogeneous suite
tests dependence on system structure, while a controlled relay grid probes dimensional scaling.
The first contains 19 polynomial systems with $1\le n\le10$, $2\le\deg f\le7$, and
$2\le N\le10$: standard and literature-derived models, synthetic degree and coupling probes,
three already-quadratic controls with $m=0$, and one system with known basin area.  Exact vector
fields and provenance appear in Appendix~\ref{app:benchmarks}.

We use an origin-connected physical measure only when every complete arm occupies multiple grid
cells without boundary clipping, giving three length, seven area, and four volume rows.  The
remaining five use the same fixed-direction first-exit proxy $R$ for all methods in the row,
avoiding comparisons of unresolved or truncated measures.  The quantity $R$ is a numerical size
score, not a certified inradius.

To prevent a slow direct-\gls{sos}
solve from suppressing completed methods, each system is split into one lifted group (Base, Gauge,
and \textsc{ODQ}) and three direct groups, one for each \textsc{sos} arm.  Every group runs in a
fresh process with an independent 600-second deadline; at most four groups run concurrently, with
each child restricted to one numerical-solver thread.  This common cap bounds the high-dimensional
attempts while making every missing value an observed computational or audit outcome.
A cell is displayed only when all three repeats return an audited certificate.

For $m>0$, both cohorts stop \textsc{ODQ} when the sampled residual in
\eqref{eq:bundle-residual} satisfies $\chi_k\le\omega_k$ at
$(\epsilon_k,\omega_k)=(\epsilon_{\min},\omega_{\min})$, or at a cap of $40$ outer iterations
or $200$ oracle calls.  The caps bound repeated \gls{sdp} work, whereas the residual test measures
finite-sample stationarity at the prescribed resolution; neither a budget exit nor a failed line
search certifies stationarity.  The best screened candidate is returned as in
\eqref{eq:anytime-return}.  All lifted and \textsc{sos}-2 groups complete.  The two quartic arms
time out in every repeat on \texttt{duffing-chain-3} and \texttt{liu-10d}; \textsc{sos}-4(8)
additionally passes the audit in only $1/3$ repeats on \texttt{cubic-ring-3} and $2/3$ on
\texttt{shear-cubic-a1}.  Appendix~\ref{app:benchmarks} records the execution and completeness
checks.

The relay grid is $(p,n)\in\{3,5,7\}\times\{2,4,6,8\}$, with $m=(p-1)/2$ and $N=n+m$.
Thus increasing $n$ at fixed $p$ leaves the number of stabilizer gains unchanged.  All rows use
$R$ to avoid dimension-dependent volume grids.  Each method runs in an isolated process with a
predeclared 600-second deadline, bounding its cost independently of other methods.  A timeout
or completed budget crossing retires that method only at larger $n$ within the same degree and
repeat; numerical failures alone do not trigger retirement.  These resource-limited outcomes are
distinguished from completed comparisons in Appendix~\ref{app:relay-benchmark}.

\begin{table}[!tp]
\centering
\setlength{\abovecaptionskip}{0pt}
\setlength{\belowcaptionskip}{4pt}
\caption{Relative results (three-repeat medians, two significant digits). For ODQ, the table reports the absolute pair $(M/T)$, where $M\in\{L,A,V,R\}$ denotes length, area, volume, or the fixed-direction proxy and $T$ is the construction time in seconds. Each complete comparator is reported as $(R_M/R_T)$, with $R_M=M_{\rm method}/M_{\rm ODQ}$ and $R_T=T_{\rm method}/T_{\rm ODQ}$; $R_M<1$ and $R_T>1$ favor ODQ. Bold marks strict joint ODQ dominance over every complete direct-\gls{sos} arm. Missing cells: $--^{R}$ retired by the relay protocol, $--^{A}_{k/3}$ attempted but not repeat-complete ($k$ of three repeats), and TO three hard timeouts. Absolute pair results appear in Table~\ref{tab:bench-absolute}.}\label{tab:bench-quality}
\scriptsize
\setlength{\tabcolsep}{1pt}
\renewcommand{\arraystretch}{0.92}
\begin{tabular*}{\linewidth}{@{\extracolsep{\fill}}lccrrrrrr@{}}
\toprule
& & & \multicolumn{1}{c}{Proposed} & \multicolumn{2}{c}{Lifted baselines} & \multicolumn{3}{c}{Direct \gls{sos}} \\
\cmidrule(lr){4-4}\cmidrule(lr){5-6}\cmidrule(l){7-9}
System or $(p,n)$ & $n/N/m$ & $M$ & ODQ & Base & Gauge & SOS--2 & SOS--4(2) & SOS--4(8) \\
\midrule
\addlinespace[2pt]
\multicolumn{9}{@{}l}{\emph{Reference systems and scalar degree probes}} \\[-1pt]
\addlinespace[2pt]
\texttt{vanderpol} & $2/3/1$ & A & $(2.7/0.3)$ & $(0.33/0.38)$ & $(0.73/0.41)$ & $(2.4/1.9)$ & $(2.5/19)$ & $(2.5/51)$ \\
\texttt{planar-quintic} & $2/4/2$ & A & $\boldsymbol{(35/0.72)}$ & $(0.21/0.18)$ & $(0.45/0.19)$ & $(0.29/1.7)$ & $(0.39/21)$ & $(0.85/65)$ \\
\texttt{cubic} & $1/2/1$ & L & $(1.9/0.2)$ & $(0.67/0.59)$ & $(0.67/0.62)$ & $(1/1.2)$ & $(1/8.3)$ & $(1/21)$ \\
\texttt{quintic} & $1/3/2$ & L & $(1.3/0.74)$ & $(0.7/0.17)$ & $(0.72/0.19)$ & $(1.5/0.46)$ & $(1.5/3.8)$ & $(1.5/4)$ \\
\texttt{bistable} & $1/2/1$ & R & $(0.051/0.4)$ & $(0.017/0.27)$ & $(0.052/0.29)$ & $(2/0.67)$ & $(2/3.5)$ & $(2/4.2)$ \\
\texttt{duffing} & $2/3/1$ & A & $(3.2/0.19)$ & $(0.36/0.51)$ & $(0.64/0.55)$ & $(1.1/2.8)$ & $(1.2/31)$ & $(1.2/36)$ \\
\hline
\addlinespace[2pt]
\multicolumn{9}{@{}l}{\emph{Quadratic controls ($m=0$)}} \\[-1pt]
\addlinespace[2pt]
\texttt{quad-cross} & $2/2/0$ & A & $(8.9/0.0098)$ & $(0.67/0.35)$ & $(1/0.95)$ & $(1.2/49)$ & $(2.1/570)$ & $(1.8/1500)$ \\
\texttt{quad-2d} & $2/2/0$ & A & $(3.9/0.0096)$ & $(0.99/0.4)$ & $(1/0.99)$ & $(1.3/49)$ & $(2.4/580)$ & $(2.1/1100)$ \\
\hline
\addlinespace[2pt]
\multicolumn{9}{@{}l}{\emph{Synthetic degree and cascade probes}} \\[-1pt]
\addlinespace[2pt]
\texttt{cubic-3d} & $3/7/4$ & V & $(36/4.3)$ & $(0.24/0.036)$ & $(0.53/0.049)$ & $(0.73/0.21)$ & $(0.92/5.7)$ & $(0.8/18)$ \\
\texttt{cascade-3d} & $3/6/3$ & V & $(32/2.3)$ & $(0.3/0.056)$ & $(0.54/0.071)$ & $(0.69/0.39)$ & $(0.76/10)$ & $(0.94/33)$ \\
\texttt{septic} & $1/4/3$ & L & $(1.1/1.1)$ & $(0.69/0.12)$ & $(0.71/0.13)$ & $(1.8/0.38)$ & $(1.8/3)$ & $(1.8/3.3)$ \\
\texttt{septic-2d} & $2/5/3$ & A & $\boldsymbol{(18/1.6)}$ & $(0.32/0.093)$ & $(0.57/0.11)$ & $(0.42/1.8)$ & $(0.57/18)$ & $(0.9/59)$ \\
\hline
\addlinespace[2pt]
\multicolumn{9}{@{}l}{\emph{Semidiscrete, chain, and ring models}} \\[-1pt]
\addlinespace[2pt]
\texttt{chafee-3} & $3/6/3$ & V & $(6.4/1.8)$ & $(0.44/0.11)$ & $(0.45/0.12)$ & $(1.1/0.67)$ & $(0.93/13)$ & $(0.93/46)$ \\
\texttt{duffing-chain-2} & $4/6/2$ & R & $(0.98/1.1)$ & $(0.59/0.13)$ & $(0.72/0.17)$ & $(1.1/3.4)$ & $(1.2/73)$ & $(1.3/260)$ \\
\texttt{duffing-chain-3} & $6/9/3$ & R & $(0.95/6.4)$ & $(0.58/0.023)$ & $(0.73/0.043)$ & $(1/2.1)$ & TO & TO \\
\texttt{cubic-ring-3} & $3/9/6$ & V & $(33/11)$ & $(0.24/0.021)$ & $(0.58/0.033)$ & $(0.81/0.088)$ & $(1.4/2)$ & $--^{A}_{1/3}$ \\
\hline
\addlinespace[2pt]
\multicolumn{9}{@{}l}{\emph{Literature-derived network models}} \\[-1pt]
\addlinespace[2pt]
\texttt{networked-vdp-2} & $4/6/2$ & R & $(0.48/1.1)$ & $(0.62/0.12)$ & $(0.93/0.15)$ & $(1.8/1.8)$ & $(2/71)$ & $(2/110)$ \\
\texttt{liu-10d} & $10/10/0$ & R & $\boldsymbol{(19/0.29)}$ & $(0.74/0.26)$ & $(1/1.1)$ & $(0.53/160)$ & TO & TO \\
\hline
\addlinespace[2pt]
\multicolumn{9}{@{}l}{\emph{Prescribed-basin geometry test}} \\[-1pt]
\addlinespace[2pt]
\texttt{shear-cubic-a1} & $2/6/4$ & A & $(0.63/3.4)$ & $(0.11/0.12)$ & $(0.25/0.13)$ & $(1.7/0.89)$ & $(3.1/9.2)$ & $--^{A}_{2/3}$ \\
\hline
\addlinespace[2pt]
\multicolumn{9}{@{}l}{\emph{Single-link relay family \eqref{eq:app-relay-family}; instances are indexed by $(p,n)$}} \\[-1pt]
\addlinespace[2pt]
$(3,2)$ & $2/3/1$ & R & $(1.2/1.7)$ & $(0.75/1)$ & $(0.98/0.94)$ & $(0.87/1.1)$ & $--^{A}_{2/3}$ & $(1/10)$ \\
$(3,4)$ & $4/5/1$ & R & $\boldsymbol{(0.89/2.3)}$ & $(0.51/0.72)$ & $(0.98/0.77)$ & $(0.77/2)$ & $--^{A}_{0/3}$ & $--^{A}_{0/3}$ \\
$(3,6)$ & $6/7/1$ & R & $\boldsymbol{(0.7/3.5)}$ & $(0.49/0.43)$ & $(0.99/0.47)$ & $(0.79/3.6)$ & $--^{A}_{0/3}$ & $--^{A}_{0/3}$ \\
$(3,8)$ & $8/9/1$ & R & $\boldsymbol{(0.53/5.7)}$ & $(0.42/0.3)$ & $(1/0.28)$ & $(0.85/5.5)$ & TO & $--^{A}_{0/3}$ \\
$(5,2)$ & $2/4/2$ & R & $(1/1.9)$ & $(0.7/0.92)$ & $(0.91/0.93)$ & $(0.94/1.3)$ & $(1.1/7.4)$ & $--^{A}_{0/3}$ \\
$(5,4)$ & $4/6/2$ & R & $(0.8/2.8)$ & $(0.45/0.63)$ & $(0.96/0.68)$ & $(0.82/6.2)$ & $(1.7/110)$ & TO \\
$(5,6)$ & $6/8/2$ & R & $\boldsymbol{(0.62/2.9)}$ & $(0.45/0.62)$ & $(0.99/0.58)$ & $(0.89/51)$ & TO & $--^{R}$ \\
$(5,8)$ & $8/10/2$ & R & $(0.48/7.7)$ & $(0.37/0.21)$ & $(0.99/0.24)$ & TO & $--^{R}$ & $--^{R}$ \\
$(7,2)$ & $2/5/3$ & R & $(0.94/2.2)$ & $(0.64/0.71)$ & $(0.91/0.64)$ & $(0.99/1.6)$ & $(1.2/11)$ & $(1.3/26)$ \\
$(7,4)$ & $4/7/3$ & R & $\boldsymbol{(0.74/3.8)}$ & $(0.41/0.51)$ & $(0.96/0.42)$ & $(0.85/20)$ & TO & TO \\
$(7,6)$ & $6/9/3$ & R & $(0.57/5.6)$ & $(0.41/0.29)$ & $(0.98/0.31)$ & TO & $--^{R}$ & $--^{R}$ \\
$(7,8)$ & $8/11/3$ & R & $(0.45/13)$ & $(0.33/0.13)$ & $(0.98/0.15)$ & $--^{R}$ & $--^{R}$ & $--^{R}$ \\
\bottomrule
\end{tabular*}
\end{table}

\paragraph{Results and analysis}
Table~\ref{tab:bench-quality} gives the full comparison, while Figure~\ref{fig:bench-tradeoff}
isolates the complete direct-\gls{sos} pairs.  Across the 16 heterogeneous systems with $m>0$,
ODQ improves the reported metric over both Base and Gauge in every row, by median factors of
$2.892$ and $1.652$, respectively; the median ODQ/Gauge time factor is $7.172$.  All returned sets
pass the independent audit.  Six outer searches reach the oracle cap, so those entries are best
audited box candidates rather than stationary points.  Against \textsc{sos}-2, ODQ is larger on
$6/19$ systems, faster on $11/19$, and both on three, with median size ratio
$M_{\rm ODQ}/M_{\textsc{sos}-2}=0.951$.  It is faster on all 17 and 15 repeat-complete quartic
cells for two and eight rounds, respectively; the corresponding median \textsc{sos}/ODQ time
factors are $12.794$ and $45.814$, and ODQ is larger on five cells at each budget.  On the relay
grid, all 36 ODQ runs complete, and the
12 cell-median times remain between $1.69$ and $13.08$ seconds for $N\le11$.  All 27 repeat-level
comparisons in the nine fully paired \textsc{sos}-2 cells favor ODQ in both proxy and time, with
geometric-mean ODQ/\textsc{sos}-2 proxy and \textsc{sos}-2/ODQ time factors of $1.162$ and $4.305$.
In the five repeat-complete relay quartic cells, the proxy is $1.03$--$1.69$ times larger, but
construction is $7.41$--$109$ times slower; the missing high-dimensional endpoints rule out a
family-wide quartic comparison.

The favorable relay scaling is consistent with its one-link structure.  For fixed odd degree $p$,
\textsc{DQbee} introduces only $m=(p-1)/2$ auxiliary powers, so $N=n+m$, whereas the gain dimension
and the $m+1$-sample bundle do not grow with $n$.  Increasing $n$ lengthens only the stable linear
chain without adding nonlinear monomials.  The cached reduced-Schur oracle therefore uses a
sparse quadratic lift and a cone of order $N^2+N$, while the largest direct-\gls{sos} decrease Gram matrices
have orders $\binom{n+(p+1)/2}{(p+1)/2}$ for \textsc{sos}-2 and
$\binom{n+(p+3)/2}{(p+3)/2}$ for \textsc{sos}-4.  At $(p,n)=(7,8)$, these orders are $495$ and
$1287$, compared with $132$ for the corresponding ODQ cone.  The lifted
comparison shows that Gauge
closes most of the gap from Base to ODQ: $M_{\rm Gauge}/M_{\rm ODQ}$ ranges from $0.907$ to $0.995$,
whereas $M_{\rm Base}/M_{\rm ODQ}$ ranges from $0.331$ to $0.751$.  Gain optimization adds a smaller
but consistent improvement.  The relay results suggest that ODQ may be especially useful for
higher-degree, higher-dimensional nonlinear systems whose sparse structure admits a compact
quadratic lift.  This conclusion is empirical rather than asymptotic, and the advantage is not
universal:
\textsc{sos}-2 is both cheaper and larger on five low-dimensional heterogeneous systems.

Numerical diagnostics, failures, and the timing breakdown appear in
Appendix~\ref{app:numerical-diagnostics}.

\begin{figure}[h!]
\centering
\includegraphics[width=\textwidth]{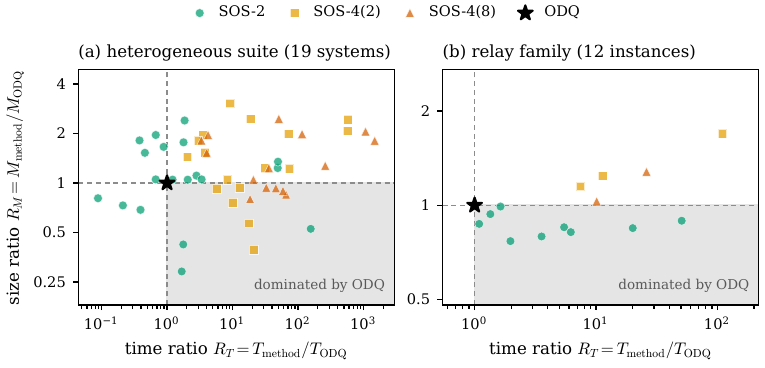}
\caption{Certified size versus construction time for the direct-\gls{sos} comparators, relative
to ODQ.  Each colored marker represents one complete direct-\gls{sos} cell of
Table~\ref{tab:bench-quality}, placed at
$(R_T,R_M)=(T_{\mathrm{method}}/T_{\mathrm{ODQ}},\,M_{\mathrm{method}}/M_{\mathrm{ODQ}})$; the star
marks ODQ at $(1,1)$.  Points in the shaded quadrant are dominated by ODQ; axes are logarithmic
and differ between panels.}
\label{fig:bench-tradeoff}
\end{figure}

\section{Conclusion and Future Work}
\label{sec:conclusion}

This paper separates the choices that enter dissipative quadratization and its \gls{roa}
analysis. These are the lifting map, the off-manifold extension, the representation of its
quadratic term, and the Lyapunov certificate itself. For a fixed embedding, we characterize the
full affine space of admissible quadratic extensions and identify the stabilizer family within
it as a structured affine subspace. Extension choices can change transverse stability and the
resulting Lyapunov metric, while representation gauges preserve the vector field but can change
the spectral-norm bound used in the certificate. Fixing a monomial lift, a reference extension,
stabilizer factorizations, the weight $Q=I_N$, and an admissible gain box, we optimize the
diagonal stabilizer gains together with the Kronecker gauge.

The tangent--transverse splitting reduces feasibility to the auxiliary block, and the
fixed-gain gauge problem is an exact reduced \gls{sdp} whose global optimum minimizes the stated
spectral-norm bound, though equality with the intrinsic cubic growth is not guaranteed. We prove
attainment on the prescribed box, give a residual-aware closed certificate, and show that the
level threshold ranks the leading-order volumes of scaled pullbacks when $\phi_B>0$. An explicit
unbounded-ray example shows why the box must be part of the problem data. The resulting
algorithm combines the reduced oracle with a box-constrained nonsmooth search, and under the
hypotheses of Theorem~\ref{thm:outer-clarke}, every accumulation point of the idealized sequence
is box-Clarke stationary. Its finite implementation instead returns the best independently
screened candidate encountered, and we do not claim global optimality.

The experiments show both the benefit and the limits of this design. On the planar quintic,
gain optimization increases the certified area from $15.83$ to $35.44$ with the lift and $Q$
fixed, and \textsc{ODQ} improves on both fixed-gain lifted baselines for every heterogeneous
system with $m>0$. Direct \gls{sos} nonetheless remains competitive, since \textsc{sos}-2 is
both larger and cheaper on five low-dimensional heterogeneous systems. In all nine fully paired
single-link relay cells, \textsc{ODQ} has a larger fixed-direction proxy and is faster than
\textsc{sos}-2, and the relay's sparse nonlinear structure admits a compact quadratic lift as
the physical dimension grows. We expect the method to extend to similarly structured
higher-degree, higher-dimensional systems, but these finite experiments do not establish an
asymptotic advantage or uniform dominance over direct \gls{sos}.

Three directions would relax the present formulation. Optimizing over the full admissible
extension space, allowing normalized metric design, and bounding the intrinsic cubic growth
more tightly would each remove one of its current restrictions. Global gain optimization and
interval or rational residual bounds remain separate questions, of search quality and of
certificate validity respectively. Beyond \gls{roa} estimation, extension design could also be
used to seek tighter truncation-error bounds and reachable-set enclosures in Carleman
linearization \citep{forets2021reachability}, which would require finite-horizon error estimates
and consistent lifting of initial sets that the present Lyapunov results alone do not provide.
Barrier certificates and feedback synthesis offer further settings in which to study how
extension choices affect safety and controlled invariance. The off-manifold extension of an
exact lift should therefore be chosen together with the certificate used for the original
system.

\section*{Acknowledgments}
We thank Weihan Zhang for contributions during the early stages of this work.
We also thank Dr.~Runyu Zhang, Prof.~Gleb Pogudin, and Prof.~Georgina Hall for helpful discussions.

\section*{Declarations}
The authors acknowledge the use of AI-assisted tools for improving grammar, clarity, and overall writing quality.

\bibliographystyle{siamplain}
\bibliography{references}

\clearpage
\appendix
\startcontents[appendices]
\phantomsection
\pdfbookmark[1]{Appendix contents}{odq-appendix-contents}
\noindent\textbf{Appendix contents}\par\smallskip
\printcontents[appendices]{app}{1}[2]{}
\par\medskip

\section{Proof of the ellipsoidal Lyapunov certificate}
\label{app:ellipsoidal-certificate}

\begin{proof}[Proof of Proposition~\ref{prop:ellipsoidal-certificate}]
\phantomsection\label{proof:ellipsoidal-certificate}
For any symmetric positive definite $P,Q$, the quotient in \eqref{eq:intrinsic-growth} is
unchanged under positive rescaling of $\mathbf w$. It can therefore be restricted to the
compact ellipsoid $\{\mathbf u:\mathbf u^{\!\top}P\mathbf u=1\}$. Its denominator is positive
there, so the supremum is finite and attained. Since $\mathbf q$ is homogeneous quadratic,
$\mathbf q(-\mathbf u)=\mathbf q(\mathbf u)$, whereas
$\mathbf u^{\!\top}P\mathbf q(\mathbf u)$ changes sign. Consequently,
\begin{equation}\label{eq:app-beta-unit}
\beta(P,Q,\mathbf q)
=\max_{\mathbf u^{\!\top}P\mathbf u=1}
\frac{2|\mathbf u^{\!\top}P\mathbf q(\mathbf u)|}{\mathbf u^{\!\top}Q\mathbf u}
=\max_{\mathbf u^{\!\top}P\mathbf u=1}
\frac{2\mathbf u^{\!\top}P\mathbf q(\mathbf u)}{\mathbf u^{\!\top}Q\mathbf u}.
\end{equation}

For the exact Lyapunov solution, write each nonzero $\mathbf w=t\mathbf u$, where
$t=\sqrt{V(\mathbf w)}>0$ and $\mathbf u^{\!\top}P\mathbf u=1$. Homogeneity gives
\begin{equation}\label{eq:app-radial-decay}
\dot V(t\mathbf u)
=t^2\mathbf u^{\!\top}Q\mathbf u
\left[-1+t\frac{2\mathbf u^{\!\top}P\mathbf q(\mathbf u)}
{\mathbf u^{\!\top}Q\mathbf u}\right].
\end{equation}
If $\beta>0$, this is negative for $0<t<1/\beta$. At a maximizing direction in
\eqref{eq:app-beta-unit}, it equals zero at $t=1/\beta$; hence no larger radius has strict
decrease throughout its punctured interior. Taking reciprocals over the positive directions
in \eqref{eq:app-beta-unit} gives \eqref{eq:rhorad} and $\rho_{\mathrm{rad}}=1/\beta$.
If $\beta=0$, the cubic contribution vanishes identically and
$\dot V=-\mathbf w^{\!\top}Q\mathbf w<0$ for every nonzero $\mathbf w$, so
$\rho_{\mathrm{rad}}=+\infty$.

For any representation $\mathbf q(\mathbf w)=H(\mathbf w\kron\mathbf w)$,
Cauchy--Schwarz in the $P$ inner product and
$\norm{\mathbf w\kron\mathbf w}=\norm{\mathbf w}^2$ yield
\[
2|\mathbf w^{\!\top}P\mathbf q(\mathbf w)|
\le2\sqrt{V(\mathbf w)}\sqrt{\norm{P}}\,\norm{H}\,\norm{\mathbf w}^2.
\]
Using $\mathbf w^{\!\top}Q\mathbf w\ge\lambda_{\min}(Q)\norm{\mathbf w}^2$ gives
\begin{equation}\label{eq:app-simple-growth-bound}
\beta(P,Q,\mathbf q)
\le\frac{2\sqrt{\norm{P}}\,\norm{H}}{\lambda_{\min}(Q)}.
\end{equation}
Taking reciprocals proves \eqref{eq:rhoan}, including $H=0$ under the stated convention.

Now let $\widehat P$ be any symmetric positive definite matrix. The symmetric normalized
residual satisfies
$-\delta I_N\preceq Q^{-1/2}RQ^{-1/2}\preceq\delta I_N$.
Thus $\mathbf w^{\!\top}R\mathbf w\le\delta\mathbf w^{\!\top}Q\mathbf w$, and
\begin{align*}
\dot{\widehat V}(\mathbf w)
&=\mathbf w^{\!\top}(-Q+R)\mathbf w
  +2\mathbf w^{\!\top}\widehat P\mathbf q(\mathbf w)\\
&\le-(1-\delta)\mathbf w^{\!\top}Q\mathbf w
  +\gamma\sqrt{\widehat V(\mathbf w)}\,\mathbf w^{\!\top}Q\mathbf w,
\end{align*}
which proves \eqref{eq:metric-residual-decay}. On \eqref{eq:metric-safe-ellipsoid},
$\gamma\sqrt{\widehat V}\le\theta(1-\delta)$. With
$\kappa:=\lambda_{\min}(\widehat P^{-1/2}Q\widehat P^{-1/2})>0$, it follows that
\begin{equation}\label{eq:app-safe-exponential}
\dot{\widehat V}
\le-(1-\theta)(1-\delta)\mathbf w^{\!\top}Q\mathbf w
\le-(1-\theta)(1-\delta)\kappa\widehat V.
\end{equation}
The derivative is strictly negative on the boundary, so a trajectory cannot leave the
ellipsoid. The polynomial vector field is locally Lipschitz, and the invariant ellipsoid is
compact; hence every trajectory starting there exists for all $t\ge0$. Integration of
\eqref{eq:app-safe-exponential} gives
$\widehat V(\mathbf w(t))\le\widehat V(\mathbf w(0))
e^{-(1-\theta)(1-\delta)\kappa t}$ and convergence to the origin.
This also proves the exact-solution closed-ellipsoid claims: when $\beta>0$, take
$\widehat P=P$, $\delta=0$, $\gamma=\beta$, and $\theta=\rho\beta<1$.
If $\gamma=0$, the bound
$\dot{\widehat V}\le-(1-\delta)\kappa\widehat V$ holds globally. Every finite sublevel set
is compact and invariant, so the same argument gives global existence and exponential
convergence. This covers the exact case $\beta=0$ as well. Replacing $\delta$ and $\gamma$
by upper bounds with the residual bound below one preserves all the inequalities.
\end{proof}

\section{Problem formulation: overview and proofs}
\label{app:problem-structure}

\begin{figure}[!htbp]
\centering
\begingroup
\begingroup
\definecolor{odqFlowInk}{HTML}{243247}
\definecolor{odqFlowSlate}{HTML}{617287}
\definecolor{odqFlowLine}{HTML}{758397}
\definecolor{odqFlowBlue}{HTML}{2D6196}
\definecolor{odqFlowTeal}{HTML}{24786F}
\definecolor{odqFlowViolet}{HTML}{68528F}
\definecolor{odqFlowNeutralFill}{HTML}{F2F5F8}
\definecolor{odqFlowBlueFill}{HTML}{E9F1FB}
\definecolor{odqFlowTealFill}{HTML}{E8F4F0}
\definecolor{odqFlowVioletFill}{HTML}{F0ECF8}
\begin{tikzpicture}[
  font=\small,
  stage/.style={
    draw=odqFlowSlate!55, text=odqFlowInk, line width=.65pt,
    rounded corners=3pt, align=center, inner xsep=7pt, inner ysep=7pt,
    fill=odqFlowNeutralFill, outer sep=0pt
  },
  wide/.style={stage,text width=.90\linewidth},
  half/.style={stage,text width=.405\linewidth},
  gain/.style={draw=odqFlowBlue!75,fill=odqFlowBlueFill},
  gauge/.style={draw=odqFlowTeal!75,fill=odqFlowTealFill},
  certificate/.style={draw=odqFlowViolet!75,fill=odqFlowVioletFill},
  choice/.style={line width=.9pt},
  flow/.style={-{Latex[length=2mm,width=1.3mm]},
    draw=odqFlowLine,line width=.7pt,rounded corners=2pt},
  stem/.style={draw=odqFlowLine,line width=.7pt},
  linklabel/.style={font=\footnotesize,text=odqFlowSlate,
    fill=white,inner xsep=4pt,inner ysep=2pt}
]
\node[wide] (data) {
  {\sffamily\bfseries\color{odqFlowSlate}Fixed data}\\[3pt]
  Original system $\mathbf p$; embedding $\Phi$ and manifold $\M$; dimension $N$\\
  Reference $\mathbf F_0$; stabilizers $\{h_i\}$; weight $Q=I_N$; Hurwitz gain box $B$
};

\node[wide,fill=white,below=6mm of data] (space) {
  {\sffamily\bfseries\color{odqFlowSlate}All admissible quadratic extensions}\\[3pt]
  $\mathcal E_\Phi(\mathbf p)=\mathbf F_0+\mathcal V_\Phi,
  \qquad \mathcal V_\Phi=I(\M)^N\cap\mathcal P_2^N$
};

\node[wide,gain,choice,below=9mm of space] (extension) {
  {\sffamily\bfseries\color{odqFlowBlue}Extension variable: $\Lambda\in B$}\\[3pt]
  $\displaystyle
  \mathbf F_\Lambda=\mathbf F_0-\sum_{i=1}^m\lambda_i\mathbf e_{n+i}h_i
  =A(\Lambda)\mathbf w+\mathbf q_\Lambda(\mathbf w)$\\[2pt]
  Original dynamics on $\M$ are preserved
};

\node[half,minimum height=27mm,anchor=north]
  (metric) at ([xshift=-.2475\linewidth,yshift=-7mm]extension.south) {
  {\sffamily\bfseries\color{odqFlowSlate}Metric determined by $A(\Lambda)$}\\[4pt]
  $A^\top P+PA=-I_N,\quad P=P(\Lambda)\succ0$\\[3pt]
  Independent of the gauge $\mu$
};

\node[half,gauge,choice,minimum height=27mm,anchor=north]
  (gauge) at ([xshift=.2475\linewidth,yshift=-7mm]extension.south) {
  {\sffamily\bfseries\color{odqFlowTeal}Representation variable: $\mu\in\R^d$}\\[4pt]
  $H(\Lambda,\mu)=H(\Lambda)+K(\mu)$\\[2pt]
  $K(\mu)(\mathbf w\kron\mathbf w)\equiv0$\\[2pt]
  The quadratic map $\mathbf q_\Lambda$ is unchanged
};

\node[wide,certificate,minimum height=23mm,anchor=north] (bound)
  at ($(metric.south)!.5!(gauge.south)+(0,-.7)$) {
  {\sffamily\bfseries\color{odqFlowViolet}Spectral-norm Lyapunov bound}\\[4pt]
  $\gamma(\Lambda,\mu):=\norm{J(\Lambda,\mu)}
  \ge\beta\bigl(P(\Lambda),I_N,\mathbf q_\Lambda\bigr)$\\[3pt]
  $0<\rho<1/\gamma(\Lambda,\mu)$ certifies
  $\mathcal D(P(\Lambda),\rho)$ and its pullback under $\Phi$
};

\node[half,gauge,minimum height=23mm,anchor=north]
  (inner) at ([xshift=-.2475\linewidth,yshift=-7mm]bound.south) {
  {\sffamily\bfseries\color{odqFlowTeal}Inner problem: fix $\Lambda$}\\[4pt]
  $\displaystyle\phi(\Lambda)=\min_{\mu\in\R^d}\gamma(\Lambda,\mu)$\\[3pt]
  Convex semidefinite program
};

\node[half,gain,minimum height=23mm,anchor=north]
  (outer) at ([xshift=.2475\linewidth,yshift=-7mm]bound.south) {
  {\sffamily\bfseries\color{odqFlowBlue}Outer problem: choose $\Lambda$}\\[4pt]
  $\displaystyle\phi_B=\min_{\Lambda\in B}\phi(\Lambda),
  \qquad \rho_B=\phi_B^{-1}$\\[3pt]
  Generally nonconvex
};

\draw[flow] (data) -- (space);
\draw[flow] (space) -- node[linklabel]
  {restrict to $\mathcal S_h=\operatorname{span}\{\mathbf e_{n+i}h_i\}\subseteq\mathcal V_\Phi$}
  (extension);
\coordinate (fork) at ([yshift=-3.5mm]extension.south);
\draw[stem] (extension.south) -- (fork);
\draw[flow] (fork) -| (metric.north);
\draw[flow] (fork) -| (gauge.north);
\draw[flow] (metric.south) -- ([xshift=-.2475\linewidth]bound.north);
\draw[flow] (gauge.south) -- ([xshift=.2475\linewidth]bound.north);
\draw[flow] ([xshift=-.2475\linewidth]bound.south) -- (inner.north);
\draw[flow] (inner.east) -- (outer.west);
\end{tikzpicture}
\endgroup
\caption{Construction of Problem~\ref{prob:main}.
The gains select the extension; the gauge changes only its matrix representation.
The diagram states the optimization objective; certificates from numerical candidates
require the verification in Corollary~\ref{cor:resid}.}
\label{fig:problem-flow}
\endgroup
\end{figure}
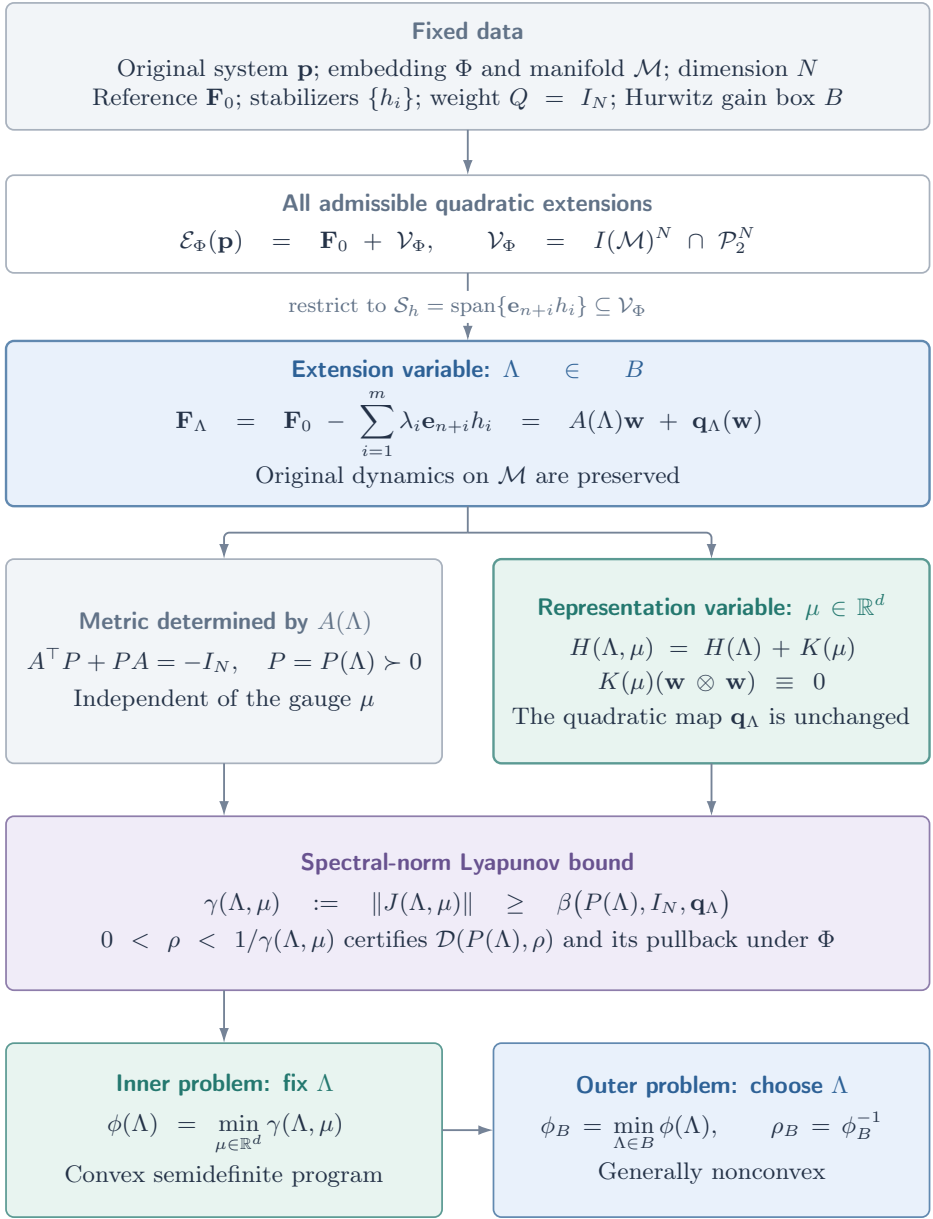

\subsection{Extension space}
\label{app:extension-gauge}

\phantomsection\label{proof:extension-space}
For the fixed polynomial embedding, substitution defines the linear map
\[
T_\Phi:\mathcal P_2\longrightarrow\R[\mathbf x],
\qquad T_\Phi r=r\circ\Phi,
\qquad \ker T_\Phi=I(\M)\cap\mathcal P_2.
\]
Two degree-at-most-two vector fields agree on $\M$ if and only if their difference belongs
componentwise to $\ker T_\Phi$. Since $\mathbf F_0$ is exact, this gives
$\mathcal E_\Phi(\mathbf p)=\mathbf F_0+(\ker T_\Phi)^N$.
To compute the kernel, write
\[
r(\mathbf w)=c_0+\sum_{i=1}^N c_iw_i+
\sum_{1\le i\le j\le N}c_{ij}w_iw_j.
\]
Substitution and coefficient matching in $\mathbf x$ give a finite homogeneous linear system
in the $1+N+N(N+1)/2$ coefficients. Its null space gives a basis
$r_1,\ldots,r_s$ of $\ker T_\Phi$, and
$\{\mathbf e_k r_\ell:1\le k\le N,\ 1\le\ell\le s\}$ is a basis of
$\mathcal V_\Phi$; in particular, $\dim\mathcal V_\Phi=Ns$.

Each $h_i$ belongs to $\ker T_\Phi$, so every direction
$\mathbf e_{n+i}h_i$ in \eqref{eq:stabilizer-subspace} belongs to $\mathcal V_\Phi$.
Its linear part is $\mathbf e_{n+i}y_i$. These linear parts are linearly independent,
which proves $\dim\mathcal S_h=m$. Hence $\Lambda\mapsto\mathbf F_\Lambda$ parametrizes
$\mathbf F_0+\mathcal S_h$ injectively. The restriction can exclude both auxiliary and
physical equation directions. For example, if $m>0$, then
$\mathbf e_1h_1\in\mathcal V_\Phi\setminus\mathcal S_h$ because $h_1$ is nonzero and all
fields in $\mathcal S_h$ have zero physical components. Thus the inclusion is strict whenever
$m>0$. Exactness also gives $\mathbf F(\mathbf 0)=J_x\Phi(\mathbf 0)\mathbf p(\mathbf 0)=0$
for every admissible extension, so its constant term vanishes.

\subsection{Transverse feasibility}
\label{app:transverse-feasibility}

\begin{proof}[Proof of Lemma~\ref{lem:split}]
\phantomsection\label{proof:split}
Let $\Phi(\mathbf x)=(\mathbf x,\mathbf g(\mathbf x))$ and let $\mathbf F_\Lambda$ denote the
right-hand side of \eqref{eq:gainsys}. Exactness of the lift gives
\[
\mathbf F_\Lambda(\Phi(\mathbf x))=J_x(\Phi)(\mathbf x)\,\mathbf p(\mathbf x).
\]
Because $\mathbf p(\mathbf 0)=\mathbf 0$, $\mathbf g(\mathbf 0)=\mathbf 0$, and
$J_x(\mathbf g)|_{\mathbf 0}=0$, differentiation at the origin yields
\[
\Aof
\begin{bmatrix}I_n\\0\end{bmatrix}
=
\begin{bmatrix}I_n\\0\end{bmatrix}J_p;
\]
the quadratic part of $\mathbf F_\Lambda$ has zero derivative there. Hence the first $n$ columns
of $\Aof$ are $[J_p^{\!\top},0]^{\!\top}$. Moreover, both factors $a_i$ and $b_i$ vanish at the
origin, so
\[
J_y(\mathbf h)|_{(\mathbf 0,\mathbf 0)}=I_m,
\qquad \mathbf h:=(h_1,\dots,h_m).
\]
The remaining blocks are therefore the gain-independent matrix $R$ and
$J_y(\mathbf q_2-\diag(\Lambda)\mathbf h)|_{(\mathbf 0,\mathbf 0)}
=W-\diag(\Lambda)$, proving \eqref{eq:block}. Its characteristic polynomial factors over the two
diagonal blocks, which gives \eqref{eq:split}. The Hurwitz equivalence and \eqref{eq:margin} follow
by taking spectral abscissae and using $\alpha_p>0$.
\end{proof}

\begin{proof}[Proof of Corollary~\ref{cor:sat}]
\leavevmode\phantomsection\label{proof:sat}
By \eqref{eq:block}, adding $c\mathbf 1$ to the gains subtracts $cI_m$ from the transverse
matrix and shifts every transverse eigenvalue by $-c$. Taking negative spectral abscissae
gives \eqref{eq:gain-shift}, and setting $\Lambda=\mathbf 0$ gives
\eqref{eq:uniform-feasibility}. The full-margin claims follow from
$\alpha(\Lambda)=\min\{\alpha_p,\alpha_{\mathrm{aux}}(\Lambda)\}$ in Lemma~\ref{lem:split}.
Finally, \eqref{eq:AH} gives the affine formula for $H$, and continuity of the spectral norm
gives its asymptotic behavior.
\end{proof}

\subsection{Representation space}
\label{app:representation-space}

\phantomsection\label{proof:representation-space}
For the representation freedom, reshape the $r$-th row of
$K\in\R^{N\times N^2}$ as $K_r\in\R^{N\times N}$. Its contribution to component $r$ is
$\mathbf w^{\!\top}K_r\mathbf w$, which vanishes identically if and only if $K_r$ is
skew-symmetric: evaluation at $\mathbf e_i$ gives $(K_r)_{ii}=0$, and evaluation at
$\mathbf e_i+\mathbf e_j$ gives $(K_r)_{ij}+(K_r)_{ji}=0$.
There are $N(N-1)/2$ independent entries per component, so
$\dim\mathcal K=N^2(N-1)/2$. Finally, two matrices $H_1,H_2$ represent the same homogeneous
quadratic map if and only if $(H_1-H_2)(\mathbf w\kron\mathbf w)\equiv0$.
This proves that \eqref{eq:Hmu} parametrizes all representations of the fixed quadratic map.

\subsection{Intrinsic growth and spectral-norm relaxation}
\label{app:intrinsic-bound}

\begin{proof}[Proof of Proposition~\ref{prop:intrinsic-bound}]
\phantomsection\label{proof:intrinsic-bound}
Factor $P=P_f^{\!\top}P_f$ and $Q=Q_f^{\!\top}Q_f$. For any exact representation, write
$\mathbf q(\mathbf w)=\sum_i w_iM_i\mathbf w$ and
$G=\operatorname{col}_i(M_i^{\!\top}P+PM_i)$. Then
\begin{align*}
2\mathbf w^{\!\top}P\mathbf q(\mathbf w)
&=(\mathbf w\kron\mathbf w)^{\!\top}G\mathbf w\\
&=(P_f\mathbf w\kron Q_f\mathbf w)^{\!\top}
  J_Q(Q_f\mathbf w),
\qquad
J_Q=(P_f^{\!\top}\kron Q_f^{\!\top})^{-1}GQ_f^{-1}.
\end{align*}
Taking absolute values and applying the spectral-norm inequality gives
\[
2|\mathbf w^{\!\top}P\mathbf q(\mathbf w)|
\le\norm{J_Q}\,\norm{P_f\mathbf w}\,\norm{Q_f\mathbf w}^2
=\norm{J_Q}\sqrt{V(\mathbf w)}\,\mathbf w^{\!\top}Q\mathbf w.
\]
Division by the positive denominator for $\mathbf w\ne0$ proves $\beta\le\norm{J_Q}$.
The gauge identity leaves $\mathbf q$, and hence $\beta$, unchanged. Since $J_Q$ is affine
in $\mu$, its image is a closed affine subspace of a finite-dimensional matrix space.
Intersecting this image with the closed norm ball of radius $\norm{J_Q(H)}$ gives a nonempty
compact set on which the norm attains its minimum. This proves
\eqref{eq:intrinsic-spectral-bound} and attainment.
\end{proof}

\subsection{Pullback connectedness}
\label{app:pullback-box-proofs}

\begin{proof}[Proof of Corollary~\ref{cor:conn}]
\phantomsection\label{proof:conn}
Fix $\mathbf x_0\in S(\Lambda,\rho)$, set $P=P(\Lambda)$ and
$\mathbf w_0=\Phi(\mathbf x_0)$, and let $\mu^\star$ attain the minimum in
\eqref{eq:rhostarL}. Put $j_\star:=\norm{J(\Lambda,\mu^\star)}$ and
$\delta:=1-\rho j_\star>0$. For $\mathbf w\in\mathcal D(P,\rho)$,
\eqref{eq:vdotJ} gives $\dot V(\mathbf w)\le-\delta\norm{\mathbf w}^2\le-cV(\mathbf w)$, where
$c:=\delta/\lambda_{\max}(P)>0$. A first-exit argument makes $\mathcal D(P,\rho)$ positively
invariant; its compactness and the polynomiality of the lifted field imply that the solution
$\mathbf w(t)$ from $\mathbf w_0$ is global. Gronwall's inequality gives
$V(\mathbf w(t))\le e^{-ct}V(\mathbf w_0)$, hence $\mathbf w(t)\to\mathbf 0$.
Exactness, $\mathbf F_\Lambda(\Phi(\mathbf x))=J_x\Phi(\mathbf x)\mathbf p(\mathbf x)$, and
uniqueness yield $\mathbf w(t)=\Phi(\mathbf x(t))$; thus $\mathbf x(t)$ remains in
$S(\Lambda,\rho)$ and converges to $\mathbf 0$. Consequently,
$s\mapsto\mathbf x(s/(1-s))$ for $0\le s<1$, extended by $\mathbf 0$ at $s=1$, is a continuous
path in $S(\Lambda,\rho)$ from $\mathbf x_0$ to the origin. Every point admits such a path, so any
two points can be joined through the origin.
\end{proof}

\subsection{Derivation of the box-constrained formulation}
\label{app:problem-derivation}

\phantomsection\label{proof:problem-derivation}
Equation~\eqref{eq:rhostarL} already establishes the fixed-gain identity
$\rho^\star(\Lambda)=\phi(\Lambda)^{-1}$, and
Proposition~\ref{prop:box-wellposed} gives attainment on $B$. Hence, if $\phi_B>0$, order reversal
under inversion gives
\begin{equation}\label{eq:app-box-reciprocal}
\max_{\Lambda\in B}\rho^\star(\Lambda)
=\left(\min_{\Lambda\in B}\phi(\Lambda)\right)^{-1}
=\rho_B.
\end{equation}
If $\phi_B=0$, the same proposition supplies an optimizing pair with $J=0$, so both sides of
\eqref{eq:app-box-reciprocal} are $+\infty$ under the convention in \eqref{eq:problem}.

It remains to introduce the epigraph variable. For $J\in\R^{N^2\times N}$,
\begin{equation}\label{eq:app-spectral-epigraph}
\norm{J}\le t
\quad\Longleftrightarrow\quad
\begin{bmatrix}
tI_{N^2}&J\\
J^{\!\top}&tI_N
\end{bmatrix}\succeq0.
\end{equation}
Indeed, the block inequality forces $t\ge0$. For $t>0$, the Schur complement gives
\[
tI_N-t^{-1}J^{\!\top}J\succeq0
\quad\Longleftrightarrow\quad
J^{\!\top}J\preceq t^2I_N
\quad\Longleftrightarrow\quad
\norm{J}\le t.
\]
For $t=0$, its $2\times2$ principal minors show that the block inequality holds if and only if
$J=0$, so
\eqref{eq:app-spectral-epigraph} also covers the boundary case. Consequently,
\[
\phi_B
=\min_{\substack{\Lambda\in B,\ \mu\in\R^d,\ t\in\R}}
\left\{t:
\begin{bmatrix}
tI_{N^2}&J(\Lambda,\mu)\\
J(\Lambda,\mu)^{\!\top}&tI_N
\end{bmatrix}\succeq0\right\}.
\]
Substituting \eqref{eq:AH}, \eqref{eq:Hmu}, \eqref{eq:lyapL}, and \eqref{eq:J} gives precisely
\eqref{eq:problemfull}. As noted in Remark~\ref{rem:forms}, the Lyapunov equation uniquely fixes
$P=P(\Lambda)\succ0$; $P_f$ is its upper Cholesky factor with positive diagonal, not an independent
decision variable. Thus the substitution introduces no additional freedom and preserves the
optimal value $t_B=\phi_B$.

\subsection{Well-posedness of the box value}
\label{app:box-wellposed}

\begin{proof}[Proof of Proposition~\ref{prop:box-wellposed}]
\phantomsection\label{proof:box-wellposed}
Write a gauge direction as $S=(S_i)_{i=1}^N$, so $K(\mathbf w\kron\mathbf w)=0$ is equivalent to
$(S_i)_{rj}=-(S_j)_{ri}$. For $P=P^{\!\top}\succ0$, let
$T_P(S)_i:=S_i^{\!\top}P+PS_i$. The bijection $U_i=PS_i$ gives
$T_P(S)_i=U_i^{\!\top}+U_i$. Define
\[
\mathcal C:=\left\{(B_i)_{i=1}^N\in(\Sset^N)^N:
(B_i)_{jk}+(B_j)_{ki}+(B_k)_{ij}=0\ \text{for all }i,j,k\right\},
\]
which is independent of $P$. Gauge antisymmetry gives $\operatorname{im}T_P\subseteq\mathcal C$, while
$\ker T_P$ corresponds under $U_i=PS_i$ to $\wedge^3\R^N$. Thus
$\dim\operatorname{im}T_P=d-\binom N3=:q=N(N^2-1)/3$. The cyclic-sum map has rank
$\binom{N+2}3$ because its restriction to fully symmetric tensors is multiplication by $3$;
hence $\dim\mathcal C=q$ and $\operatorname{im}T_P=\mathcal C$.

Fixing a basis of $\mathcal C$ gives a nonredundant parametrization
$J(\Lambda,\nu)=J_0(\Lambda)+\mathcal L_\Lambda\nu$, $\nu\in\R^q$, with
$\mathcal L_\Lambda$ injective by \eqref{eq:J}. On $\mathcal F$, the Lyapunov solution and its
positive-diagonal Cholesky factor depend analytically on $\Lambda$; hence $J_0$ and
$\mathcal L_\Lambda$ are locally Lipschitz. The case $q=0$ is immediate. Otherwise, fix a closed
ball $K\Subset\mathcal F$ and set
$c_K:=\min_{\Lambda\in K,\,\norm{\nu}=1}\norm{\mathcal L_\Lambda\nu}>0$ and
$M_K:=\max_{\Lambda\in K}\norm{J_0(\Lambda)}$. The bound
$\norm{J(\Lambda,\nu)}\ge c_K\norm{\nu}-M_K$ makes every inner problem coercive, hence attained;
comparison with $\nu=0$ places every minimizer in the common ball
$\norm{\nu}\le R_K:=2M_K/c_K$. Uniform Lipschitz continuity of $J$ on this set now gives
\[
\left|\phi(\Lambda_1)-\phi(\Lambda_2)\right|
\le\sup_{\norm{\nu}\le R_K}
\norm{J(\Lambda_1,\nu)-J(\Lambda_2,\nu)}
\le L_K\norm{\Lambda_1-\Lambda_2}.
\]
Thus $\phi$ is locally Lipschitz and attains its minimum on the compact box $B$. If $\phi_B>0$,
strict monotonicity of $x\mapsto1/x$ gives the same optimizing gains; if $\phi_B=0$, attainment
gives an optimizing pair with $J=0$, and \eqref{eq:vdotJ} makes the corresponding quadratic
Lyapunov function global.
\end{proof}

\subsection{Gauge-defect correction}
\label{app:verified-certificate}

For a returned representation $\widetilde H$ that has a nonzero coefficient defect, choose
$\Delta H$ such that
$\mathbf q(\mathbf w)-\widetilde H(\mathbf w\kron\mathbf w)
=\Delta H(\mathbf w\kron\mathbf w)$ identically. Applying
\eqref{eq:app-simple-growth-bound} to the defect gives the valid bound
\begin{equation}\label{eq:app-gauge-defect-bound}
\beta(\widehat P,Q,\mathbf q)
\le\norm{J_Q(\widetilde H)}
+\frac{2\sqrt{\norm{\widehat P}}\,\norm{\Delta H}}{\lambda_{\min}(Q)},
\end{equation}
where $J_Q$ is assembled using $\widehat P$. Thus a nonzero gauge residual requires a
correction to the growth bound. If the gauge identity holds exactly, one may take
$\Delta H=0$. In numerical use, both terms in \eqref{eq:app-gauge-defect-bound} and the
normalized Lyapunov residual require upper bounds that include rounding errors.

\subsection{Leading-order certified volume}
\label{app:volume}

\begin{proof}[Proof of Proposition~\ref{prop:vol}]
\phantomsection\label{proof:vol}
Fix $\Lambda$ and write $P=P(\Lambda)$. With the blocks in \eqref{eq:Pblocks}, set
$q(\mathbf x):=\mathbf x^{\!\top}P_{xx}\mathbf x$ and
\[
r(\mathbf x):=\widetilde V_\Lambda(\mathbf x)-q(\mathbf x)
=2\mathbf x^{\!\top}P_{xy}\mathbf g(\mathbf x)
+\mathbf g(\mathbf x)^{\!\top}P_{yy}\mathbf g(\mathbf x).
\]
Since $\mathbf g(\mathbf x)=O(\norm{\mathbf x}^2)$, there is $c>0$ such that
$|r(\mathbf x)|\le c\norm{\mathbf x}^3$ near the origin. Let
$\lambda_x:=\lambda_{\min}(P_{xx})>0$. On
$\mathbb B_\delta:=\{\norm{\mathbf x}\le\delta\}$, for all sufficiently small $\delta$,
\begin{equation}\label{eq:app-volume-sandwich}
(1-\varepsilon_\delta)q(\mathbf x)
\le\widetilde V_\Lambda(\mathbf x)
\le(1+\varepsilon_\delta)q(\mathbf x),
\qquad
\varepsilon_\delta:=\frac{c\delta}{\lambda_x}<1.
\end{equation}
Let $E(s):=\{\mathbf x:q(\mathbf x)\le s^2\}$. Since $P\succ0$ and $\Phi$ contains $\mathbf x$,
$\widetilde V_\Lambda(\mathbf x)\ge\lambda_{\min}(P)\norm{\mathbf x}^2$; hence, for fixed
$\delta$ and all sufficiently small $\rho$, both $S(\Lambda,\rho)$ and the inner ellipsoid below
lie in $\mathbb B_\delta$. Equation~\eqref{eq:app-volume-sandwich} then gives
\[
E\!\left(\frac{\rho}{\sqrt{1+\varepsilon_\delta}}\right)
\subseteq S(\Lambda,\rho)
\subseteq
E\!\left(\frac{\rho}{\sqrt{1-\varepsilon_\delta}}\right).
\]
Because $\operatorname{vol}E(s)=\kappa_ns^n/\sqrt{\det P_{xx}}$, division by the proposed leading
term yields bounds $(1+\varepsilon_\delta)^{-n/2}$ and
$(1-\varepsilon_\delta)^{-n/2}$. Letting first $\rho\downarrow0$ and then $\delta\downarrow0$
squeezes the ratio to one, proving \eqref{eq:volexp}.
\end{proof}

\subsection{Gain invariance of the quadratic germ}
\label{app:quadratic-germ}

\begin{proof}[Proof of Proposition~\ref{prop:pxx}]
\phantomsection\label{proof:pxx}
By Lemma~\ref{lem:split},
$\Aof=\left[\begin{smallmatrix}J_p&R\\0&\Dof\end{smallmatrix}\right]$, and only $\Dof$ depends on
$\Lambda$. The $(1,1)$ block of the Lyapunov equation is
$J_p^{\!\top}P_{Q,xx}+P_{Q,xx}J_p=-Q_{xx}$. Since $J_p$ is Hurwitz, this equation has a unique
solution, independent of $\Lambda$.
\end{proof}

\clearpage
\section{Inner-oracle and outer-convergence details}
\label{app:outer-analysis}

\subsection{The gauge-image distance and its dual}
\label{app:inner-distance-dual}

\begin{proof}[Proof of Proposition~\ref{prop:inner-distance-dual}]
\phantomsection\label{proof:inner-distance-dual}
Fix $\Lambda$ and abbreviate $\mathcal T=\mathcal T_\Lambda$ and
$\mathcal W=\mathcal W_\Lambda$. The attainable matrices are exactly
$J_0+\mathcal W$. Since $\mathcal W=-\mathcal W$, minimizing their spectral norm gives
\eqref{eq:inner-distance}. The subspace $\mathcal W$ is closed and finite-dimensional.
The intersection of $J_0+\mathcal W$ with any nonempty bounded norm sublevel set is compact,
so a minimizing matrix $J^\star$ exists; any preimage of $J^\star-J_0$ under $\mathcal T$
is a minimizing gauge.

Introduce a matrix variable $M$ and impose $M=J_0+\mathcal T\mu$. The Lagrangian is
\[
L(M,\mu,Z)=\norm{M}+\langle Z,J_0+\mathcal T\mu-M\rangle_F.
\]
Its infimum over $\mu$ is finite exactly when
$\langle Z,\mathcal T\mu\rangle_F=0$ for all $\mu$, or $Z\in\mathcal W^\perp$.
Spectral--nuclear norm duality gives
$\langle Z,M\rangle_F\le\norm{Z}_*\norm{M}$, with
$\sup_{\norm{M}\le1}\langle Z,M\rangle_F=\norm{Z}_*$: equality follows by taking
$M=UV^{\!\top}$ in a compact singular value decomposition $Z=U\Sigma V^{\!\top}$.
Consequently,
\[
\inf_M\{\norm{M}-\langle Z,M\rangle_F\}
=\begin{cases}0,&\norm{Z}_*\le1,\\-\infty,&\norm{Z}_*>1.\end{cases}
\]
This yields the dual in \eqref{eq:inner-distance-dual}. Its feasible set is nonempty and
compact, so its maximum is attained.

For completeness, strong duality follows directly from convex optimality. The spectral norm
is continuous everywhere, so the convex Fermat rule on $J_0+\mathcal W$ gives
\[
0\in\partial\norm{\cdot}(J^\star)+\mathcal W^\perp.
\]
Hence some $Z^\star\in\partial\norm{\cdot}(J^\star)$ lies in $\mathcal W^\perp$.
The subgradient inequality for a norm is equivalent to
\begin{equation}\label{eq:app-norm-subgradient}
Z\in\partial\norm{\cdot}(J)
\quad\Longleftrightarrow\quad
\norm{Z}_*\le1,\qquad \langle Z,J\rangle_F=\norm{J}.
\end{equation}
Indeed, applying that inequality at $0$ and $2J$ forces equality at $J$, after which the
inequality for arbitrary matrices is precisely the dual-norm bound.
Thus $Z^\star$ is dual feasible and
\[
\langle Z^\star,J_0\rangle_F
=\langle Z^\star,J^\star\rangle_F=\norm{J^\star}=\phi(\Lambda).
\]
Conversely, such a $Z^\star$ bounds the norm of every matrix in $J_0+\mathcal W$ from below
by $\norm{J^\star}$, proving the stated optimality criterion.
Finally, weak duality gives
$\langle Z,J_0\rangle_F\le\phi(\Lambda)\le\norm{J(\Lambda,\mu)}$ for every primal--dual
feasible pair, which proves \eqref{eq:inner-duality-gap}.
\end{proof}

\subsection{The reduced oracle and its dual}
\label{app:reduced-inner}

\begin{proof}[Proof of Proposition~\ref{prop:reduced-inner}]
\phantomsection\label{proof:reduced-inner}
Let $X=P_f^{\!\top}\kron I_N$. Since $P=P_f^{\!\top}P_f$,
$P\kron I_N=XX^{\!\top}$. Congruence of $F_\Lambda(\nu,\tau)$ by
$\diag(X^{-1},I_N)$, or equivalently its Schur complement, gives
\begin{equation}\label{eq:app-inner-congruence}
F_\Lambda(\nu,\tau)\succeq0
\quad\Longleftrightarrow\quad
\begin{bmatrix}I_{N^2}&J_{\rm r}(\Lambda,\nu)\\
J_{\rm r}(\Lambda,\nu)^{\!\top}&\tau I_N\end{bmatrix}\succeq0
\quad\Longleftrightarrow\quad
J_{\rm r}^{\!\top}J_{\rm r}\preceq\tau I_N.
\end{equation}
Minimizing $\tau$ therefore gives
$\tau^\star=\min_\nu\norm{J_{\rm r}(\Lambda,\nu)}^2=\phi(\Lambda)^2$.
Strict feasibility follows by fixing any $\nu$ and taking
$\tau>\norm{J_{\rm r}(\Lambda,\nu)}^2$. For $q>0$, linear independence of the
$\mathcal B_\ell$ and invertibility of $X$ make
$\nu\mapsto X^{-1}\sum_\ell\nu_\ell\mathcal B_\ell$ injective, hence bounded below on the unit
sphere. The objective is therefore coercive in $\nu$; for $q=0$ attainment is immediate.

For the reconstruction, Proposition~\ref{prop:box-wellposed} proves that the image of the raw gauge
under $S\mapsto\operatorname{col}_r(S_r^{\!\top}P+PS_r)$ is exactly $\mathcal C$. Thus the
constraint in \eqref{eq:raw-gauge-reconstruction} is feasible. Its strictly convex quadratic
objective selects a unique representative, and substitution in \eqref{eq:J} adds precisely
$\sum_\ell\nu_\ell^\star\mathcal B_\ell$ to $G(\Lambda,0)$. Hence it realizes the same $J$, value,
and vector field.
\end{proof}

For later use, the dual of \eqref{eq:inner-generalized-schur} is
\begin{equation}\label{eq:inner-dual}
\begin{aligned}
\max_{Y\succeq0}\quad&
-\left\langle Y_{11},P\kron I_N\right\rangle_F
-2\left\langle Y_{12},G(\Lambda,0)\right\rangle_F\\
\mathrm{s.t.}\quad&
\operatorname{tr}Y_{22}=1,\qquad
\left\langle Y_{12},\mathcal B_\ell\right\rangle_F=0,\quad \ell=1,\dots,q,
\end{aligned}
\end{equation}
where $Y$ is partitioned conformally with $F_\Lambda$. This follows from the Lagrangian
$\tau-\langle Y,F_\Lambda\rangle_F$ by stationarity in $(\tau,\nu)$. Strict feasibility gives
strong duality; primal--dual optimality adds the complementarity condition in
\eqref{eq:inner-conic-kkt}.

\subsection{Proof of the sensitivity formula}
\label{app:inner-sensitivity}

The spectral subdifferential in \eqref{eq:spectral-subdifferential} and the affine chain rule give
\begin{equation}\label{eq:app-inner-fermat}
0\in\partial_\nu\norm{J_{\rm r}(\Lambda,\nu^\star)}
\quad\Longleftrightarrow\quad
\exists Z^\star\in\partial\norm{\cdot}(J^\star):
\left\langle Z^\star,J_\ell(\Lambda)\right\rangle_F=0
\quad(\ell=1,\dots,q),
\end{equation}
which proves \eqref{eq:inner-spectral-kkt} by the convex Fermat rule.

We next make the Cholesky derivative in \eqref{eq:lambda-sensitivities} explicit. Set
$C_i:=P_f^{-\top}\dot P_iP_f^{-1}$ and let $\operatorname{up}_{1/2}(C_i)$ retain the strict upper
triangle of $C_i$, take half of its diagonal, and set its strict lower triangle to zero. Then
\begin{equation}\label{eq:app-cholesky-differential}
\dot P_{f,i}=\operatorname{up}_{1/2}(C_i)P_f.
\end{equation}
Indeed, with $R_i:=\dot P_{f,i}P_f^{-1}$, differentiation of
$P=P_f^{\!\top}P_f$ gives $C_i=R_i^{\!\top}+R_i$; the upper-triangularity of $R_i$ determines it
uniquely. Differentiating the block definition of $G(\Lambda,0)$ gives the first line of
\eqref{eq:GJ-sensitivities}, and differentiating
$XJ_{\rm r}=G_{\rm r}$ gives its second line. The sign in the Lyapunov sensitivity follows from
$\dot A_i=-E_i$:
\[
\dot A_i^{\!\top}P+\Aof^{\!\top}\dot P_i+\dot P_i\Aof+P\dot A_i=0
\quad\Longleftrightarrow\quad
\Aof^{\!\top}\dot P_i+\dot P_i\Aof=E_iP+PE_i.
\]

\begin{proof}[Proof of Proposition~\ref{prop:regular-sensitivity}]
\phantomsection\label{proof:regular-sensitivity}
For $\Lambda\in\mathcal U$, define the Lyapunov operator on $\Sset^N$ by
\[
\mathcal A_\Lambda(S):=A(\Lambda)^{\!\top}S+SA(\Lambda).
\]
Since $A(\Lambda)$ is Hurwitz, this operator is invertible, and
\[
P_Q(\Lambda)=-\mathcal A_\Lambda^{-1}(Q)
=\int_0^\infty e^{A(\Lambda)^{\!\top}t}Qe^{A(\Lambda)t}\,dt\succ0.
\]
Indeed, the integral converges and solves the equation by integration of its time
derivative. If $\mathcal A_\Lambda(S)=0$, then
$e^{A^{\!\top}t}Se^{At}$ is constant in $t$ and tends to zero, so $S=0$; in finite
dimensions this proves invertibility. The coefficients of $\mathcal A_\Lambda$ are affine
in $\Lambda$, and matrix inversion is real analytic on the invertible matrices.
Thus $P_Q$ is real analytic on $\mathcal U$. Differentiating its defining equation, with
$Q$ fixed, gives
\[
A_i^{\!\top}P_Q+A^{\!\top}P_{Q,i}+P_{Q,i}A+P_QA_i=0,
\]
which proves \eqref{eq:lyapunov-general-sensitivity}; the derivative is unique by the same
invertibility argument. These are local assertions and require only that $A(\Lambda)$ be
Hurwitz at each point of $\mathcal U$.

The positive-diagonal Cholesky factor is real analytic on the positive definite cone,
as follows from its recursive formulas using positive square roots and division by
positive pivots. Hence $P_f$, $X^{-1}$, and $J_{\rm r}$ are real analytic for $Q=I_N$.
The fixed linearly independent matrices $\mathcal B_\ell$ make the reduced gauge map
injective throughout $\mathcal U$. For every closed ball $K\Subset\mathcal U$, the proof of
Proposition~\ref{prop:box-wellposed} therefore supplies constants $c_K>0$ and $M_K\ge0$ such that
\[
\norm{J_{\rm r}(\Lambda,\nu)}\ge c_K\norm{\nu}-M_K,
\qquad
\norm{\nu^\star(\Lambda)}\le 2M_K/c_K
\quad(\Lambda\in K)
\]
when $q>0$. Minimization is consequently restricted to a common compact ball in $\nu$,
on which $J_{\rm r}$ is uniformly Lipschitz in the gain. Comparing the two objectives at
each other's minimizers proves local Lipschitz continuity of $\phi$, as detailed in that
proof. When $q=0$, there is no inner variable and the conclusion follows directly from
the smoothness of $J_{\rm r}$ and the norm inequality.

The conic KKT conditions follow from \eqref{eq:inner-dual} and strong duality. Write
$\phi=\phi(\Lambda)>0$ and choose paired orthonormal bases $U_1,V_1$ of the leading singular
subspaces of $J^\star$. Solving the block equations for the kernel of
$F_\Lambda(\nu^\star,\phi^2)$ shows that it is spanned by the columns of
\[
T:=\begin{bmatrix}-\phi X^{-\top}U_1\\V_1\end{bmatrix}.
\]
Indeed, the first block equation determines the upper block as $-X^{-\top}J^\star v$,
and the second becomes $((J^\star)^{\!\top}J^\star-\phi^2I_N)v=0$.
For two positive semidefinite matrices, zero Frobenius inner product implies zero product.
Complementarity therefore places the range of $Y^\star$ in this kernel, so
\begin{equation}\label{eq:app-dual-factorization}
Y^\star=T\Theta T^{\!\top},\qquad \Theta\succeq0,\qquad
\operatorname{tr}\Theta=\operatorname{tr}Y_{22}^\star=1.
\end{equation}
It follows that
\[
-\phi^{-1}X^{\!\top}Y_{12}^\star=U_1\Theta V_1^{\!\top}=:Z^\star,
\qquad \norm{Z^\star}_*=\operatorname{tr}\Theta=1.
\]
Equation~\eqref{eq:spectral-subdifferential} gives
$Z^\star\in\partial\norm{\cdot}(J^\star)$, and conic stationarity gives
\[
\langle Z^\star,J_\ell(\Lambda)\rangle_F
=-\phi^{-1}\langle Y_{12}^\star,\mathcal B_\ell\rangle_F=0.
\]
Proposition~\ref{prop:inner-distance-dual} now proves that $Z^\star$ is a distance-dual optimum.

To relate the two sensitivity formulas, put $\dot X_i=\dot P_{f,i}^{\!\top}\kron I_N$.
The blocks in \eqref{eq:app-dual-factorization} are
\[
Y_{11}^\star=\phi^2X^{-\top}U_1\Theta U_1^{\!\top}X^{-1},
\qquad Y_{12}^\star=-\phi X^{-\top}Z^\star.
\]
Using $\dot P_i\kron I_N=\dot X_iX^{\!\top}+X\dot X_i^{\!\top}$ and
$J^\star V_1=\phi U_1$ gives
\[
\langle Y_{11}^\star,\dot P_i\kron I_N\rangle_F
=2\phi\langle Z^\star,X^{-1}\dot X_iJ^\star\rangle_F,
\qquad
\langle Y_{12}^\star,\dot G_i^0\rangle_F
=-\phi\langle Z^\star,X^{-1}\dot G_i^0\rangle_F.
\]
Together with \eqref{eq:GJ-sensitivities}, these identities prove
\begin{equation}\label{eq:app-dual-envelope}
2\phi\langle Z^\star,\dot J_i\rangle_F
=-\langle Y_{11}^\star,\dot P_i\kron I_N\rangle_F
-2\langle Y_{12}^\star,\dot G_i^0\rangle_F.
\end{equation}

It remains to establish the outer subgradient claim. The dual feasible set
\[
\mathcal Y:=\{Y\succeq0:\operatorname{tr}Y_{22}=1,\quad
\langle Y_{12},\mathcal B_\ell\rangle_F=0\ (\ell=1,\dots,q)\}
\]
is independent of the gain. Fix the optimum $Y^\star$ at $\Lambda$ and define
\[
d_{Y^\star}(\Lambda'):=
-\langle Y_{11}^\star,P(\Lambda')\kron I_N\rangle_F
-2\langle Y_{12}^\star,G(\Lambda',0)\rangle_F.
\]
Weak duality gives $d_{Y^\star}(\Lambda')\le\phi(\Lambda')^2$ for every
$\Lambda'\in\mathcal F$, with equality at $\Lambda$. Since $d_{Y^\star}$ is smooth and
$d_{Y^\star}(\Lambda)=\phi^2>0$, it is positive in a neighborhood of $\Lambda$.
There $h:=\sqrt{d_{Y^\star}}$ is a smooth lower bound for $\phi$, with
$h(\Lambda)=\phi(\Lambda)$ and $\partial_{\lambda_i}h(\Lambda)=g_i$ by
\eqref{eq:app-dual-envelope}. Taylor expansion gives
\[
\phi(\Lambda+\Delta)\ge\phi(\Lambda)+g^{\!\top}\Delta+o(\norm{\Delta}).
\]
Thus $g$ is a Fr\'echet subgradient and, by local Lipschitz continuity, belongs to
$\partial_C\phi(\Lambda)$. If $\phi$ is differentiable at $\Lambda$, the same inequality
applied in opposite directions gives $g=\nabla\phi(\Lambda)$.

Finally, suppose that $\nu^\star$ is unique and the leading singular value of $J^\star$ is
positive and simple. Set $f(\Lambda',\nu):=\norm{J_{\rm r}(\Lambda',\nu)}$.
Then $f$ is continuously differentiable near $(\Lambda,\nu^\star)$, and
$a:=\nabla_\Lambda f(\Lambda,\nu^\star)$ has components
$a_i=\langle uv^{\!\top},\dot J_i\rangle_F$.
For each sufficiently small $\Delta$, choose a minimizer $\nu_\Delta$ at $\Lambda+\Delta$.
Local uniform boundedness and continuity imply that every cluster point of $\nu_\Delta$
as $\Delta\to0$ minimizes $f(\Lambda,\cdot)$; uniqueness therefore gives
$\nu_\Delta\to\nu^\star$. Optimality at the two gains yields
\[
\begin{aligned}
f(\Lambda+\Delta,\nu_\Delta)-f(\Lambda,\nu_\Delta)
&\le\phi(\Lambda+\Delta)-\phi(\Lambda)\\
&\le f(\Lambda+\Delta,\nu^\star)-f(\Lambda,\nu^\star).
\end{aligned}
\]
Both bounds equal $a^{\!\top}\Delta+o(\norm{\Delta})$ by continuity of
$\nabla_\Lambda f$ and $\nu_\Delta\to\nu^\star$. Thus $\phi$ is differentiable and
\begin{equation}\label{eq:app-danskin}
\partial_{\lambda_i}\phi(\Lambda)
=\langle uv^{\!\top},\dot J_i\rangle_F=u^{\!\top}\dot J_i v,
\end{equation}
which proves \eqref{eq:regular-envelope}.
\end{proof}

At a nonsmooth point with $\phi(\Lambda)>0$, different exact optimal multipliers can select
different Clarke subgradients. An individual selection need not determine the
nearby-gradient hull in \eqref{eq:clarke-epsilon-hull}. The outer convergence theorem therefore
retains its assumption that the sampled hulls converge to the Clarke subdifferential
\citep{clarke1983optimization}.

Finally, the damping in \eqref{eq:damped-bfgs} gives
$s^{\!\top}\widetilde r\ge\kappa s^{\!\top}Ms>0$. Thus the BFGS update is positive definite whenever
$M\succ0$ and $s\ne0$; rebuilding from finitely many such pairs and clipping its eigenvalues gives
the metric safeguards assumed in Theorem~\ref{thm:outer-clarke}.

\subsection{Proof of conditional Clarke stationarity}
\label{app:outer-clarke}

Recall that
$\psi^\circ(x;d):=\limsup_{z\to x,\,t\downarrow0}
[\psi(z+td)-\psi(z)]/t$ denotes the Clarke directional derivative.

\begin{proof}[Proof of Theorem~\ref{thm:outer-clarke}]
\phantomsection\label{proof:outer-clarke}
The bounds in \eqref{eq:bundleqp} give $\theta^k+d_k\in C$. Convexity of $C$ then makes every
backtracked point $\theta^k+a d_k$, $a\in[0,1]$, feasible. A serious step satisfies
\eqref{eq:armijo-transition}, whereas a null step leaves the iterate unchanged; hence
$\psi(\theta^k)$ is nonincreasing. Compactness of $C$ gives accumulation points.

Fix a subsequence $\theta^{k_r}\to\bar\theta$ and suppose that
$0\notin\partial_C\psi(\bar\theta)+N_C(\bar\theta)$. We first show that the corresponding
residuals are bounded away from zero. Otherwise, after passing to a subsequence,
$\chi_{k_r}\to0$; the metric safeguards imply $d_{k_r}\to0$. By
\eqref{eq:gradient-consistency}, the aggregates $\bar g_{k_r}\in\widehat\partial_{k_r}$ in
Lemma~\ref{lem:bundle-model} are bounded, and a further subsequence converges to some
$\bar g\in\partial_C\psi(\bar\theta)$. Equation~\eqref{eq:bundle-kkt} makes
$n_{k_r}=-M_{k_r}d_{k_r}-\bar g_{k_r}$ bounded. Since the convex normal-cone mapping has a closed
graph and $n_{k_r}\in N_C(\theta^{k_r}+d_{k_r})$, its limit satisfies
$n\in N_C(\bar\theta)$ and $\bar g+n=0$, a contradiction. Consequently, for some $\delta>0$,
\begin{equation}\label{eq:app-residual-gap}
\chi_{k_r}^2=d_{k_r}^{\!\top}M_{k_r}d_{k_r}\ge\delta
\end{equation}
for all sufficiently large $r$.

The displacements lie in the compact set $C-C$, and the safeguarded metrics lie in a compact set.
Along any further subsequence of the chosen indices, take
$d_{k_r}\to\bar d\ne0$ and $M_{k_r}\to\bar M$. Hausdorff convergence in
\eqref{eq:gradient-consistency} implies uniform convergence of the support functions on bounded
sets. Lemma~\ref{lem:bundle-model} therefore yields
\begin{equation}\label{eq:uniform-clarke-descent}
\psi^\circ(\bar\theta;\bar d)
=\max_{g\in\partial_C\psi(\bar\theta)}g^{\!\top}\bar d
=\lim_{r\to\infty}\max_{g\in\widehat\partial_{k_r}}g^{\!\top}d_{k_r}
\le-\bar d^{\!\top}\bar M\bar d<0.
\end{equation}
Because $c<1$, joint upper semicontinuity of the Clarke directional derivative and continuity of
$d^{\!\top}Md$ imply that sufficiently small steps satisfy
\begin{equation}\label{eq:app-uniform-armijo}
\psi(\theta+a d)-\psi(\theta)\le-ca\,d^{\!\top}Md.
\end{equation}
The cluster set of $(d_{k_r},M_{k_r})$ is compact; a finite-cover argument makes the upper step
bound $a_0>0$ uniform for all sufficiently large $r$. Equation~\eqref{eq:app-residual-gap} and
$\omega_{k_r}\downarrow0$ exclude the low-residual branch. Since $L_{k_r}\to\infty$, the grid in
\eqref{eq:armijo-transition} eventually accepts a step satisfying
$\beta^{\ell_{k_r}}\ge\beta\min\{1,a_0\}$. Hence every sufficiently large $r$ decreases $\psi$ by
at least $c\beta\min\{1,a_0\}\delta>0$, contradicting nonnegativity and monotone convergence of
the values.
Thus every accumulation point satisfies \eqref{eq:box-clarke-stationarity}.

The affine change of variables gives
$\partial_C\psi(\theta)=D_B^{\!\top}\partial_C\phi(\Lambda)$ and
$N_C(\theta)=D_B^{\!\top}N_B(\Lambda)$, so stationarity is equivalent in the two coordinates. If
$\phi_B>0$, the gradient-limit characterization applied to the $C^1$ map $s\mapsto-1/s$ gives
$\partial_C(-1/\phi)=\phi^{-2}\partial_C\phi$ on $B$. Since $N_B$ is a cone, this is precisely the
Clarke condition for maximizing $\rho^\star=1/\phi$.
\end{proof}

\section{Benchmark definitions and supplementary results}
\label{app:benchmarks}

\subsection{Heterogeneous suite}

The suite was assembled to vary polynomial degree, state and lift dimension, and nonlinear
coupling, rather than sampled from a statistical population.  It contains standard models and
examples used in the stability literature (reversed Van der Pol, Duffing, a Chafee--Infante
semidiscretization, and systems drawn or adapted from
\citep{cai2024dissipative,liu2025physics,liu2024compositional}); scalar degree probes and synthetic
cascades, chains, and rings; and the already-quadratic controls \texttt{quad-cross},
\texttt{quad-2d}, and \texttt{liu-10d}.  The networked Van der Pol system is the deterministic local
variant motivated by the interconnected model in \citep{liu2024compositional}.
Table~\ref{tab:bench-systems} records the exact executed vector fields, dimensions, and
construction or literature source.

The retained geometry test, \texttt{shear-cubic-a1}, is obtained by pushing forward the base field
\begin{equation}\label{eq:app-shear-base}
 g(z)=\begin{bmatrix}-z_1-z_2+z_1\norm{z}^2\\ z_1-z_2+z_2\norm{z}^2\end{bmatrix}
\end{equation}
under $T_1(z)=(z_1,z_2+z_1^2)$.  Since
$\frac{d}{dt}\norm{z}^2=2\norm{z}^2(\norm{z}^2-1)$, the base field has the open unit disk as its
basin.  The identity
$\det\!\bigl(DT_1(z)\bigr)=1$ shows that $T_1$ preserves area.  Its basin is the sheared unit disk
\begin{equation}\label{eq:app-shear-basin}
T_1\bigl(\{z:\norm{z}<1\}\bigr)
=\bigl\{x:x_1^2+(x_2-x_1^2)^2<1\bigr\},
\end{equation}
which has area $\pi$ and a quartic boundary.  This gives a coordinate-shape benchmark with known
basin area.

\begingroup
\scriptsize
\setlength{\LTpre}{3pt}
\setlength{\LTpost}{3pt}
\setlength{\tabcolsep}{1.2pt}
\renewcommand{\arraystretch}{0.94}
\begin{longtable}{@{}lc>{\raggedright\arraybackslash}p{6.10cm}>{\raggedright\arraybackslash}p{2.65cm}@{}}
\caption{Executed vector fields in the heterogeneous suite.  The dimension tuple is $(n,\deg f,N,m)$.}
\label{tab:bench-systems}\\
\toprule
System & $(n,\deg f,N,m)$ & Dynamics & Provenance \\
\midrule
\endfirsthead
\toprule
System & $(n,\deg f,N,m)$ & Dynamics & Provenance \\
\midrule
\endhead
\bottomrule
\endfoot
\addlinespace[2pt]
\multicolumn{4}{@{}l}{\emph{Reference systems and scalar degree probes}} \\*[-1pt]
\addlinespace[2pt]
\nopagebreak
\texttt{vanderpol} & $(2,3,3,1)$ & \(\dot{x_{1}}=- x_{2}\);\allowbreak\ \(\dot{x_{2}}=x_{1} +\allowbreak x_{2} \left(x_{1}^{2} -\allowbreak 1\right)\) & reversed Van der Pol \citep[Ex.~6]{liu2025physics} \\
\texttt{planar-quintic} & $(2,5,4,2)$ & \(\dot{x_{1}}=- x_{1} +\allowbreak x_{2}\);\allowbreak\ \(\dot{x_{2}}=- x_{1}^{5} -\allowbreak 2 x_{2}\) & Section~\ref{sec:exp-pipeline} \\
\texttt{cubic} & $(1,3,2,1)$ & \(\dot{x_{1}}=x_{1}^{3} -\allowbreak x_{1}\) & \citep[Ex.~2]{cai2024dissipative} \\
\texttt{quintic} & $(1,5,3,2)$ & \(\dot{x_{1}}=x_{1}^{5} -\allowbreak x_{1}\) & scalar degree probe \\
\texttt{bistable} & $(1,3,2,1)$ & \(\dot{x_{1}}=- x_{1}^{3} +\allowbreak 0.4 x_{1}^{2} -\allowbreak 0.03 x_{1}\) & bistable reaction \citep[\S5.2]{cai2024dissipative} \\
\texttt{duffing} & $(2,3,3,1)$ & \(\dot{x_{1}}=x_{2}\);\allowbreak\ \(\dot{x_{2}}=x_{1}^{3} -\allowbreak x_{1} -\allowbreak x_{2}\) & Duffing oscillator \citep[\S5.1]{cai2024dissipative} \\
\hline
\addlinespace[2pt]
\multicolumn{4}{@{}l}{\emph{Quadratic controls ($m=0$)}} \\*[-1pt]
\addlinespace[2pt]
\nopagebreak
\texttt{quad-cross} & $(2,2,2,0)$ & \(\dot{x_{1}}=x_{1} x_{2} -\allowbreak x_{1}\);\allowbreak\ \(\dot{x_{2}}=x_{1}^{2} -\allowbreak 2 x_{2}\) & synthetic quadratic control \\
\texttt{quad-2d} & $(2,2,2,0)$ & \(\dot{x_{1}}=- x_{1} +\allowbreak x_{2}^{2}\);\allowbreak\ \(\dot{x_{2}}=x_{1} x_{2} -\allowbreak x_{2}\) & synthetic quadratic control \\
\hline
\addlinespace[2pt]
\multicolumn{4}{@{}l}{\emph{Synthetic degree and cascade probes}} \\*[-1pt]
\addlinespace[2pt]
\nopagebreak
\texttt{cubic-3d} & $(3,3,7,4)$ & \(\dot{x_{1}}=- x_{1} +\allowbreak x_{2}^{3}\);\allowbreak\ \(\dot{x_{2}}=- 2 x_{2} +\allowbreak x_{3}^{2}\);\allowbreak\ \(\dot{x_{3}}=x_{1}^{2} -\allowbreak 3 x_{3}\) & mixed-degree coupling probe \\
\texttt{cascade-3d} & $(3,3,6,3)$ & \(\dot{x_{1}}=- x_{1} +\allowbreak x_{2}^{3}\);\allowbreak\ \(\dot{x_{2}}=- x_{2} +\allowbreak x_{3}^{3}\);\allowbreak\ \(\dot{x_{3}}=- x_{3}\) & triangular cascade \\
\texttt{septic} & $(1,7,4,3)$ & \(\dot{x_{1}}=x_{1}^{7} -\allowbreak x_{1}\) & scalar degree probe \\
\texttt{septic-2d} & $(2,7,5,3)$ & \(\dot{x_{1}}=- x_{1} +\allowbreak x_{2}\);\allowbreak\ \(\dot{x_{2}}=- x_{1}^{7} -\allowbreak 2 x_{2}\) & degree-seven cascade \\
\addlinespace[2pt]
\multicolumn{4}{@{}l}{\emph{Semidiscrete, chain, and ring models}} \\*[-1pt]
\addlinespace[2pt]
\nopagebreak
\texttt{chafee-3} & $(3,3,6,3)$ & \(\dot{x_{1}}=x_{1}^{3} -\allowbreak 1.4 x_{1} +\allowbreak 0.2 x_{2}\);\allowbreak\ \(\dot{x_{2}}=0.2 x_{1} +\allowbreak x_{2}^{3} -\allowbreak 1.4 x_{2} +\allowbreak 0.2 x_{3}\);\allowbreak\ \(\dot{x_{3}}=0.2 x_{2} +\allowbreak x_{3}^{3} -\allowbreak 1.4 x_{3}\) & Chafee--Infante-type semidiscretization, \(u_t=\nu u_{xx}-u+u^3\) (three nodes) \\
\texttt{duffing-chain-2} & $(4,3,6,2)$ & \(\dot{q_{1}}=v_{1}\);\allowbreak\ \(\dot{v_{1}}=q_{1}^{3} -\allowbreak 2.0 q_{1} +\allowbreak 0.5 q_{2} -\allowbreak 1.0 v_{1}\);\allowbreak\ \(\dot{q_{2}}=v_{2}\);\allowbreak\ \(\dot{v_{2}}=0.5 q_{1} +\allowbreak q_{2}^{3} -\allowbreak 2.0 q_{2} -\allowbreak 1.0 v_{2}\) & Duffing-chain variant; cf.\ \citep[\S5.3]{cai2024dissipative} \\
\texttt{duffing-chain-3} & $(6,3,9,3)$ & \(\dot{q_{1}}=v_{1}\);\allowbreak\ \(\dot{v_{1}}=q_{1}^{3} -\allowbreak 2.0 q_{1} +\allowbreak 0.5 q_{2} -\allowbreak 1.0 v_{1}\);\allowbreak\ \(\dot{q_{2}}=v_{2}\);\allowbreak\ \(\dot{v_{2}}=0.5 q_{1} +\allowbreak q_{2}^{3} -\allowbreak 2.0 q_{2} +\allowbreak 0.5 q_{3} -\allowbreak 1.0 v_{2}\);\allowbreak\ \(\dot{q_{3}}=v_{3}\);\allowbreak\ \(\dot{v_{3}}=0.5 q_{2} +\allowbreak q_{3}^{3} -\allowbreak 2.0 q_{3} -\allowbreak 1.0 v_{3}\) & Duffing-chain variant; cf.\ \citep[\S5.3]{cai2024dissipative} \\
\texttt{cubic-ring-3} & $(3,3,9,6)$ & \(\dot{x_{1}}=- x_{1} +\allowbreak x_{2}^{3}\);\allowbreak\ \(\dot{x_{2}}=- 2 x_{2} +\allowbreak x_{3}^{3}\);\allowbreak\ \(\dot{x_{3}}=x_{1}^{3} -\allowbreak 3 x_{3}\) & directed cubic ring \\
\hline
\addlinespace[2pt]
\multicolumn{4}{@{}l}{\emph{Literature-derived network models}} \\*[-1pt]
\addlinespace[2pt]
\nopagebreak
\texttt{networked-vdp-2} & $(4,3,6,2)$ & \(\dot{x_{1}}=- x_{2}\);\allowbreak\ \(\dot{x_{2}}=- \frac{x_{1} x_{4}}{10} +\allowbreak x_{1} +\allowbreak x_{2} \left(\frac{4 x_{1}^{2}}{5} -\allowbreak \frac{4}{5}\right)\);\allowbreak\ \(\dot{x_{3}}=- x_{4}\);\allowbreak\ \(\dot{x_{4}}=- \frac{x_{2} x_{3}}{10} +\allowbreak x_{3} +\allowbreak x_{4} \left(\frac{11 x_{3}^{2}}{5} -\allowbreak \frac{11}{5}\right)\) & deterministic network variant \citep{liu2024compositional} \\
\texttt{liu-10d} & $(10,2,10,0)$ & \(\dot{x_{1}}=- x_{1} +\allowbreak 0.5 x_{2} -\allowbreak 0.1 x_{9}^{2}\);\allowbreak\ \(\dot{x_{2}}=- 0.5 x_{1} -\allowbreak x_{2}\);\allowbreak\ \(\dot{x_{3}}=- 0.1 x_{1}^{2} -\allowbreak x_{3} +\allowbreak 0.5 x_{4}\);\allowbreak\ \(\dot{x_{4}}=- 0.5 x_{3} -\allowbreak x_{4}\);\allowbreak\ \(\dot{x_{5}}=- x_{5} +\allowbreak 0.5 x_{6} +\allowbreak 0.1 x_{7}^{2}\);\allowbreak\ \(\dot{x_{6}}=- 0.5 x_{5} -\allowbreak x_{6}\);\allowbreak\ \(\dot{x_{7}}=- x_{7} +\allowbreak 0.5 x_{8}\);\allowbreak\ \(\dot{x_{8}}=- 0.5 x_{7} -\allowbreak x_{8}\);\allowbreak\ \(\dot{x_{9}}=0.5 x_{10} -\allowbreak x_{9}\);\allowbreak\ \(\dot{x_{10}}=- x_{10} +\allowbreak 0.1 x_{2}^{2} -\allowbreak 0.5 x_{9}\) & \citep[Ex.~8]{liu2025physics} \\
\hline
\addlinespace[2pt]
\multicolumn{4}{@{}l}{\emph{Prescribed-basin geometry test}} \\*[-1pt]
\addlinespace[2pt]
\nopagebreak
\texttt{shear-cubic-a1} & $(2,6,6,4)$ & \(\dot{x_{1}}=x_{1}^{5} -\allowbreak 2 x_{1}^{3} x_{2} +\allowbreak x_{1}^{3} +\allowbreak x_{1}^{2} +\allowbreak x_{1} x_{2}^{2} -\allowbreak x_{1} -\allowbreak x_{2}\);\allowbreak\ \(\dot{x_{2}}=x_{1}^{6} -\allowbreak x_{1}^{4} x_{2} +\allowbreak x_{1}^{4} +\allowbreak 2 x_{1}^{3} -\allowbreak x_{1}^{2} x_{2}^{2} +\allowbreak x_{1}^{2} x_{2} -\allowbreak x_{1}^{2} -\allowbreak 2 x_{1} x_{2} +\allowbreak x_{1} +\allowbreak x_{2}^{3} -\allowbreak x_{2}\) & area-preserving shear of \eqref{eq:app-shear-base} \\
\end{longtable}
\endgroup

\subsection{Single-link relay family}
\label{app:relay-benchmark}

For odd $p\ge3$ and $n\ge2$, the relay family is
\begin{equation}\label{eq:app-relay-family}
\begin{aligned}
 \dot x_i&=a_i(-x_i+x_{i+1}), && i=1,\ldots,n-1,\\
 \dot x_n&=a_n(-x_n+x_1^p), &
 a_i&=\frac{2+((i-1)\bmod3)}{2}.
\end{aligned}
\end{equation}
The executed grid is $(p,n)\in\{3,5,7\}\times\{2,4,6,8\}$.  Its origin Jacobian is upper
triangular with diagonal entries $-a_i$, and hence is Hurwitz.  Writing $q=(p-1)/2$, exact symbolic
reconstruction verifies the \textsc{DQbee} auxiliaries $x_1^2,\ldots,x_1^{q+1}$, so $m=q$ and
$N=n+q$ in every cell.  The two nonzero equilibria $\pm\mathbf 1$ motivate retaining the directions
$\pm\mathbf 1/\sqrt n$ in the common geometry set.

The grid, arm definitions, execution order, and retirement rule were frozen before this relay
cohort was run.  Each arm has three isolated repeats and an independent 600-second deadline; a
timeout or completed budget crossing retires only the same method at larger $n$ within the same
degree and repeat.  Geometry has a separate 120-second budget and uses 4096 seeded random
directions, the $2n$ signed coordinate directions, and the two equilibrium directions, with the
same radial scan and bisection resolution as above.

Table~\ref{tab:bench-quality}
gives the normalized comparison used in the main text, while Table~\ref{tab:bench-absolute} reports
the absolute score and construction-time medians for all six methods and both cohorts.

\begin{table}[!tp]
\centering
\setlength{\abovecaptionskip}{0pt}
\setlength{\belowcaptionskip}{4pt}
\caption{Absolute results corresponding to Table~\ref{tab:bench-quality}. Each complete cell gives the three-repeat median $(M/T)$ to two significant digits, with $T$ in seconds and $M\in\{L,A,V,R\}$. Bold and the $--^{R}$, $--^{A}_{k/3}$, and TO conventions follow Table~\ref{tab:bench-quality}.}\label{tab:bench-absolute}
\scriptsize
\setlength{\tabcolsep}{1pt}
\renewcommand{\arraystretch}{0.92}
\begin{tabular*}{\linewidth}{@{\extracolsep{\fill}}lccrrrrrr@{}}
\toprule
& & & \multicolumn{1}{c}{Proposed} & \multicolumn{2}{c}{Lifted baselines} & \multicolumn{3}{c}{Direct \gls{sos}} \\
\cmidrule(lr){4-4}\cmidrule(lr){5-6}\cmidrule(l){7-9}
System or $(p,n)$ & $n/N/m$ & $M$ & ODQ & Base & Gauge & SOS--2 & SOS--4(2) & SOS--4(8) \\
\midrule
\addlinespace[2pt]
\multicolumn{9}{@{}l}{\emph{Reference systems and scalar degree probes}} \\[-1pt]
\addlinespace[2pt]
\texttt{vanderpol} & $2/3/1$ & A & $(2.7/0.3)$ & $(0.9/0.11)$ & $(2/0.12)$ & $(6.5/0.55)$ & $(6.6/5.6)$ & $(6.6/15)$ \\
\texttt{planar-quintic} & $2/4/2$ & A & $\boldsymbol{(35/0.72)}$ & $(7.3/0.13)$ & $(16/0.14)$ & $(10/1.2)$ & $(14/15)$ & $(30/47)$ \\
\texttt{cubic} & $1/2/1$ & L & $(1.9/0.2)$ & $(1.3/0.12)$ & $(1.3/0.12)$ & $(2/0.24)$ & $(2/1.6)$ & $(2/4.1)$ \\
\texttt{quintic} & $1/3/2$ & L & $(1.3/0.74)$ & $(0.92/0.13)$ & $(0.94/0.14)$ & $(2/0.34)$ & $(2/2.8)$ & $(2/3)$ \\
\texttt{bistable} & $1/2/1$ & R & $(0.051/0.4)$ & $(0.00089/0.11)$ & $(0.0026/0.12)$ & $(0.1/0.27)$ & $(0.1/1.4)$ & $(0.1/1.7)$ \\
\texttt{duffing} & $2/3/1$ & A & $(3.2/0.19)$ & $(1.2/0.095)$ & $(2.1/0.1)$ & $(3.6/0.53)$ & $(4/5.9)$ & $(4/6.7)$ \\
\hline
\addlinespace[2pt]
\multicolumn{9}{@{}l}{\emph{Quadratic controls ($m=0$)}} \\[-1pt]
\addlinespace[2pt]
\texttt{quad-cross} & $2/2/0$ & A & $(8.9/0.0098)$ & $(5.9/0.0034)$ & $(8.9/0.0093)$ & $(11/0.48)$ & $(18/5.5)$ & $(16/14)$ \\
\texttt{quad-2d} & $2/2/0$ & A & $(3.9/0.0096)$ & $(3.9/0.0039)$ & $(3.9/0.0096)$ & $(5.3/0.47)$ & $(9.6/5.6)$ & $(8.1/10)$ \\
\hline
\addlinespace[2pt]
\multicolumn{9}{@{}l}{\emph{Synthetic degree and cascade probes}} \\[-1pt]
\addlinespace[2pt]
\texttt{cubic-3d} & $3/7/4$ & V & $(36/4.3)$ & $(8.5/0.15)$ & $(19/0.21)$ & $(26/0.9)$ & $(33/24)$ & $(29/79)$ \\
\texttt{cascade-3d} & $3/6/3$ & V & $(32/2.3)$ & $(9.6/0.13)$ & $(17/0.16)$ & $(22/0.9)$ & $(24/23)$ & $(30/74)$ \\
\texttt{septic} & $1/4/3$ & L & $(1.1/1.1)$ & $(0.76/0.13)$ & $(0.78/0.14)$ & $(2/0.42)$ & $(2/3.3)$ & $(2/3.7)$ \\
\texttt{septic-2d} & $2/5/3$ & A & $\boldsymbol{(18/1.6)}$ & $(5.7/0.15)$ & $(10/0.17)$ & $(7.5/2.9)$ & $(10/29)$ & $(16/96)$ \\
\hline
\addlinespace[2pt]
\multicolumn{9}{@{}l}{\emph{Semidiscrete, chain, and ring models}} \\[-1pt]
\addlinespace[2pt]
\texttt{chafee-3} & $3/6/3$ & V & $(6.4/1.8)$ & $(2.8/0.19)$ & $(2.9/0.22)$ & $(6.7/1.2)$ & $(5.9/23)$ & $(5.9/81)$ \\
\texttt{duffing-chain-2} & $4/6/2$ & R & $(0.98/1.1)$ & $(0.58/0.15)$ & $(0.71/0.19)$ & $(1/3.6)$ & $(1.2/79)$ & $(1.2/280)$ \\
\texttt{duffing-chain-3} & $6/9/3$ & R & $(0.95/6.4)$ & $(0.55/0.15)$ & $(0.7/0.28)$ & $(1/13)$ & TO & TO \\
\texttt{cubic-ring-3} & $3/9/6$ & V & $(33/11)$ & $(7.8/0.22)$ & $(19/0.35)$ & $(26/0.94)$ & $(47/22)$ & $--^{A}_{1/3}$ \\
\hline
\addlinespace[2pt]
\multicolumn{9}{@{}l}{\emph{Literature-derived network models}} \\[-1pt]
\addlinespace[2pt]
\texttt{networked-vdp-2} & $4/6/2$ & R & $(0.48/1.1)$ & $(0.3/0.13)$ & $(0.45/0.16)$ & $(0.84/1.9)$ & $(0.95/78)$ & $(0.95/120)$ \\
\texttt{liu-10d} & $10/10/0$ & R & $\boldsymbol{(19/0.29)}$ & $(14/0.074)$ & $(19/0.32)$ & $(10/45)$ & TO & TO \\
\hline
\addlinespace[2pt]
\multicolumn{9}{@{}l}{\emph{Prescribed-basin geometry test}} \\[-1pt]
\addlinespace[2pt]
\texttt{shear-cubic-a1} & $2/6/4$ & A & $(0.63/3.4)$ & $(0.072/0.4)$ & $(0.16/0.43)$ & $(1/3)$ & $(1.9/31)$ & $--^{A}_{2/3}$ \\
\hline
\addlinespace[2pt]
\multicolumn{9}{@{}l}{\emph{Single-link relay family \eqref{eq:app-relay-family}; instances are indexed by $(p,n)$}} \\[-1pt]
\addlinespace[2pt]
$(3,2)$ & $2/3/1$ & R & $(1.2/1.7)$ & $(0.87/1.7)$ & $(1.1/1.6)$ & $(1/1.8)$ & $--^{A}_{2/3}$ & $(1.2/17)$ \\
$(3,4)$ & $4/5/1$ & R & $\boldsymbol{(0.89/2.3)}$ & $(0.45/1.6)$ & $(0.88/1.8)$ & $(0.68/4.5)$ & $--^{A}_{0/3}$ & $--^{A}_{0/3}$ \\
$(3,6)$ & $6/7/1$ & R & $\boldsymbol{(0.7/3.5)}$ & $(0.35/1.5)$ & $(0.7/1.6)$ & $(0.56/13)$ & $--^{A}_{0/3}$ & $--^{A}_{0/3}$ \\
$(3,8)$ & $8/9/1$ & R & $\boldsymbol{(0.53/5.7)}$ & $(0.22/1.7)$ & $(0.53/1.6)$ & $(0.45/31)$ & TO & $--^{A}_{0/3}$ \\
$(5,2)$ & $2/4/2$ & R & $(1/1.9)$ & $(0.72/1.7)$ & $(0.94/1.8)$ & $(0.96/2.5)$ & $(1.2/14)$ & $--^{A}_{0/3}$ \\
$(5,4)$ & $4/6/2$ & R & $(0.8/2.8)$ & $(0.36/1.8)$ & $(0.77/1.9)$ & $(0.65/17)$ & $(1.3/310)$ & TO \\
$(5,6)$ & $6/8/2$ & R & $\boldsymbol{(0.62/2.9)}$ & $(0.28/1.8)$ & $(0.61/1.7)$ & $(0.55/150)$ & TO & $--^{R}$ \\
$(5,8)$ & $8/10/2$ & R & $(0.48/7.7)$ & $(0.18/1.6)$ & $(0.47/1.8)$ & TO & $--^{R}$ & $--^{R}$ \\
$(7,2)$ & $2/5/3$ & R & $(0.94/2.2)$ & $(0.6/1.6)$ & $(0.85/1.4)$ & $(0.93/3.6)$ & $(1.2/25)$ & $(1.2/57)$ \\
$(7,4)$ & $4/7/3$ & R & $\boldsymbol{(0.74/3.8)}$ & $(0.3/1.9)$ & $(0.71/1.6)$ & $(0.63/77)$ & TO & TO \\
$(7,6)$ & $6/9/3$ & R & $(0.57/5.6)$ & $(0.23/1.6)$ & $(0.56/1.8)$ & TO & $--^{R}$ & $--^{R}$ \\
$(7,8)$ & $8/11/3$ & R & $(0.45/13)$ & $(0.15/1.7)$ & $(0.44/2)$ & $--^{R}$ & $--^{R}$ & $--^{R}$ \\
\bottomrule
\end{tabular*}
\end{table}

\subsection{Supplementary geometry}
\label{app:numerical-supplement}

For the planar quintic, write $x=x_1$, $y=x_2$, let $z(x)=(x,x^2,x^3)^\top$, and partition the metric with
respect to $(z,y)$ as $\bigl[\begin{smallmatrix}C&b\\b^\top&a\end{smallmatrix}\bigr]$.
Since $a>0$, completing the square gives
\[
 V(x,y)=a\bigl(y+a^{-1}b^\top z(x)\bigr)^2+
 z(x)^\top(C-a^{-1}bb^\top)z(x).
\]
Writing $f(x)=\rho_{\rm safe}^2-z(x)^\top(C-a^{-1}bb^\top)z(x)$, the vertical slice has
width $2\sqrt{f(x)/a}$ wherever $f(x)>0$. We isolate all real roots of this degree-six
polynomial and integrate over every positive interval. This does not assume a star-shaped
set; all sampled quintic sets have one component. Symbolic checks verify the square
completion, including examples with disconnected sets and repeated roots. At the saved Gauge
and ODQ points, the area agrees with the original 8001-slice estimates to relative errors
$1.73\times10^{-6}$ and $1.05\times10^{-7}$. Root isolation interprets the input floats as
rational numbers; the subsequent quadrature and its error estimates remain floating-point.
The new grid comprises 145 fixed-gain SDP solves.

\subsection{Numerical diagnostics, failures, and timing}
\label{app:numerical-diagnostics}

\label{app:supp-numerical-diagnostics}
We read the source-locked raw artifacts for all $57$ heterogeneous and $36$ relay ODQ returns, including the nine heterogeneous $m=0$ controls. Table~\ref{tab:supp-historical-diagnostics} summarizes the diagnostics stored at those returned certificates. These are floating-point residual screens, not interval or rational proofs. The epigraph and dual residuals measure optimization accuracy; they do not replace the independent residual-aware certificate condition. Raw residuals have different scales across systems and dimensions. The additional normalized epigraph diagnostic uses $\max\{1,\|J\|_2\}$ only as a descriptive scale.

\begin{table}[tb]
\centering\scriptsize\setlength{\tabcolsep}{3pt}
\caption{Maximum stored numerical diagnostics at the historical ODQ returns. Columns include every returned certificate in the corresponding cohort; no new optimization was run to produce this table.}
\label{tab:supp-historical-diagnostics}
\begin{tabular}{lrr}
\toprule
Diagnostic & Heterogeneous ($57$) & Relay ($36$) \\
\midrule
Stored normalized Lyapunov residual bound & $8.02\!\times\!10^{-15}$ & $4.81\!\times\!10^{-16}$ \\
Certificate-map reconstruction residual ($\|\cdot\|_F$) & $8.15\!\times\!10^{-14}$ & $2.92\!\times\!10^{-14}$ \\
Epigraph discrepancy $|\sqrt{\tau}-\|J\|_2|$ & $2.45\!\times\!10^{-7}$ & $7.75\!\times\!10^{-9}$ \\
Epigraph discrepancy$/\max\{1,\|J\|_2\}$ & $9.87\!\times\!10^{-8}$ & $7.54\!\times\!10^{-9}$ \\
Gauge KKT residual $\max_k|\langle Z,J_k\rangle|$ & $4.45\!\times\!10^{-12}$ & $2.25\!\times\!10^{-11}$ \\
Generalized-Schur dual normalization residual & $1.99\!\times\!10^{-8}$ & $9.64\!\times\!10^{-9}$ \\
\bottomrule
\end{tabular}
\end{table}

The heterogeneous outer exits comprise $30$ finite-bundle tolerance exits, $18$ oracle-budget exits, and nine vacuous exits with no gain variables; all $36$ relay runs reach the finite-bundle tolerance. These are not exact Clarke-stationarity certificates. The stored outer residual is evaluated at a point different from the best returned certificate gain in $33/57$ heterogeneous and $9/36$ relay runs. Every returned ODQ certificate passes its separate numerical certificate screen, and no reported inner solve changes the requested \textsc{Mosek}/generalized-Schur backend. The historical files do not retain all oracle evaluations, so they do not support a count of every filtered inner candidate.

\begin{table}[tb]
\centering\scriptsize\setlength{\tabcolsep}{3pt}
\caption{Outcomes of all predeclared method--case--repeat cells: $57$ per heterogeneous method and $36$ per relay method. Indet. denotes a solver-indeterminate return; failed denotes another unsuccessful certificate return. Retired relay cells were not executed after an earlier resource-limit event and are counted separately from observed failures.}
\label{tab:supp-all-outcomes}
\begin{tabular}{llrrrrr}
\toprule
Cohort & Method & Success & Timeout & Indet. & Failed & Retired \\
\midrule
Heterogeneous & Base & 57 & 0 & 0 & 0 & 0 \\
Heterogeneous & Gauge & 57 & 0 & 0 & 0 & 0 \\
Heterogeneous & ODQ & 57 & 0 & 0 & 0 & 0 \\
Heterogeneous & SOS--2 & 57 & 0 & 0 & 0 & 0 \\
Heterogeneous & SOS--4(2) & 51 & 6 & 0 & 0 & 0 \\
Heterogeneous & SOS--4(8) & 48 & 6 & 0 & 3 & 0 \\
Relay & Base & 36 & 0 & 0 & 0 & 0 \\
Relay & Gauge & 36 & 0 & 0 & 0 & 0 \\
Relay & ODQ & 36 & 0 & 0 & 0 & 0 \\
Relay & SOS--2 & 27 & 6 & 0 & 0 & 3 \\
Relay & SOS--4(2) & 11 & 9 & 7 & 0 & 9 \\
Relay & SOS--4(8) & 6 & 9 & 8 & 0 & 13 \\
\bottomrule
\end{tabular}
\end{table}

The six heterogeneous timeouts of each quartic arm occur on \texttt{duffing-chain-3} and \texttt{liu-10d}, in all three repeats. The three additional SOS--4(8) failures comprise two returns on \texttt{cubic-ring-3} and one on \texttt{shear-cubic-a1}; one explicitly fails the prescribed radius-shrink screen and two return no valid certificate. A solver-indeterminate outcome in the relay cohort is not interpreted as mathematical infeasibility.

For each heterogeneous ODQ run, the saved construction time is the sum of lift construction, stable-box screening, the complete inner-oracle calls, gradient sensitivity, and remaining outer logic. The numerical bookkeeping agrees with this partition in all $57$ records. Lyapunov solves, matrix assembly, modeling-interface work, SDP solution, and certificate auditing are included within the inner-oracle component. In particular, \texttt{t\_inner\_sdp\_total} is not solver-only time. Saved Lyapunov and inner audit counters are nested diagnostics and are not added again. Table~\ref{tab:supp-historical-timing} reports arithmetic means over the three repeats so the partition remains additive before rounding.

\begin{table}[tb]
\centering\scriptsize\setlength{\tabcolsep}{3pt}
\caption{Mean heterogeneous ODQ construction-time partition over three repeats, in seconds. Lift: quadratization; box: stability screen; inner: complete inner oracle; grad.: gradient sensitivity; outer: remaining logic. Geometry, startup, warm-up, and serialization are excluded. Displayed components may differ from the displayed total by rounding.}
\label{tab:supp-historical-timing}
\begin{tabular}{lrrrrrr}
\toprule
System & Lift & Box & Inner & Grad. & Outer & Total \\
\midrule
\texttt{bistable} & $0.113$ & $0.0116$ & $0.229$ & $0.0127$ & $0.0439$ & $0.411$ \\
\texttt{cascade-3d} & $0.133$ & $7.29\!\times\!10^{-4}$ & $2.22$ & $0.0927$ & $0.0846$ & $2.53$ \\
\texttt{chafee-3} & $0.215$ & $7.47\!\times\!10^{-4}$ & $1.5$ & $0.0695$ & $0.0633$ & $1.85$ \\
\texttt{cubic} & $0.115$ & $3.14\!\times\!10^{-4}$ & $0.0475$ & $6.69\!\times\!10^{-3}$ & $0.0268$ & $0.196$ \\
\texttt{cubic-3d} & $0.165$ & $1.18\!\times\!10^{-3}$ & $4.16$ & $0.151$ & $0.0793$ & $4.56$ \\
\texttt{cubic-ring-3} & $0.245$ & $5.19\!\times\!10^{-3}$ & $10.2$ & $0.237$ & $0.0666$ & $10.8$ \\
\texttt{duffing} & $0.0955$ & $3.65\!\times\!10^{-4}$ & $0.0643$ & $7.3\!\times\!10^{-3}$ & $0.0244$ & $0.192$ \\
\texttt{duffing-chain-2} & $0.145$ & $5.07\!\times\!10^{-4}$ & $0.93$ & $0.039$ & $0.053$ & $1.17$ \\
\texttt{duffing-chain-3} & $0.178$ & $1.06\!\times\!10^{-3}$ & $6.52$ & $0.101$ & $0.0733$ & $6.87$ \\
\texttt{liu-10d} & $0.0806$ & $0$ & $0.211$ & $0$ & $1.5\!\times\!10^{-4}$ & $0.292$ \\
\texttt{networked-vdp-2} & $0.158$ & $6.85\!\times\!10^{-4}$ & $0.985$ & $0.0354$ & $0.0504$ & $1.23$ \\
\texttt{planar-quintic} & $0.126$ & $4.78\!\times\!10^{-4}$ & $0.523$ & $0.0285$ & $0.0483$ & $0.727$ \\
\texttt{quad-2d} & $4.93\!\times\!10^{-3}$ & $0$ & $6\!\times\!10^{-3}$ & $0$ & $1.07\!\times\!10^{-4}$ & $0.011$ \\
\texttt{quad-cross} & $3.24\!\times\!10^{-3}$ & $0$ & $6.81\!\times\!10^{-3}$ & $0$ & $1.23\!\times\!10^{-4}$ & $0.0102$ \\
\texttt{quintic} & $0.128$ & $4.13\!\times\!10^{-4}$ & $0.549$ & $0.0266$ & $0.051$ & $0.755$ \\
\texttt{septic} & $0.153$ & $6.07\!\times\!10^{-4}$ & $0.966$ & $0.0668$ & $0.0705$ & $1.26$ \\
\texttt{septic-2d} & $0.161$ & $6.16\!\times\!10^{-4}$ & $1.47$ & $0.0855$ & $0.0818$ & $1.8$ \\
\texttt{shear-cubic-a1} & $0.396$ & $0.0803$ & $2.71$ & $0.138$ & $0.0812$ & $3.4$ \\
\texttt{vanderpol} & $0.112$ & $3.87\!\times\!10^{-4}$ & $0.142$ & $0.0129$ & $0.0415$ & $0.308$ \\
\bottomrule
\end{tabular}
\end{table}

The relay artifacts instead retain launch-to-serialization wall time and separately timed geometry extraction, without an inner-stage partition. We do not infer a missing solver-only time by subtracting available counters, and we retain the two cohorts' different timing conventions.

To inspect matrix-level quantities absent from the historical artifacts, we perform one new fixed-gain inner SDP at each saved ODQ gain across all $19+12$ cases. The lifting map, stabilizer factorization, and $Q=I$ remain fixed, and no outer optimization is repeated. Table~\ref{tab:supp-posterior-diagnostics} concerns these newly computed matrices; it does not reconstruct the historical solver matrices or replace their reported certificates and timings.

\begin{table}[tb]
\centering\scriptsize\setlength{\tabcolsep}{3pt}
\caption{Fixed-gain audit of newly computed certificates, one per case. The first row gives the minimum positive eigenvalue; all other rows give maximum defects. These floating-point checks are distinct from the historical summaries above.}
\label{tab:supp-posterior-diagnostics}
\begin{tabular}{lrr}
\toprule
Diagnostic & Heterogeneous ($19$) & Relay ($12$) \\
\midrule
$\lambda_{\min}(P)$ & $0.0306$ & $0.0479$ \\
$\|A^TP+PA+I\|_2$ & $5.54\!\times\!10^{-15}$ & $3.26\!\times\!10^{-16}$ \\
Coefficientwise gauge identity defect & $4.44\!\times\!10^{-16}$ & $2.22\!\times\!10^{-16}$ \\
Reassembled $J$ discrepancy$/\max\{1,\|J\|_2\}$ & $2.15\!\times\!10^{-16}$ & $2.42\!\times\!10^{-16}$ \\
Model-LMI negative eigenvalue part (scaled) & $2.17\!\times\!10^{-9}$ & $2.07\!\times\!10^{-9}$ \\
Returned-$H$ LMI negative eigenvalue part (scaled) & $2.17\!\times\!10^{-9}$ & $2.07\!\times\!10^{-9}$ \\
Epigraph discrepancy$/\max\{1,\|J\|_2\}$ & $9.87\!\times\!10^{-8}$ & $7.54\!\times\!10^{-9}$ \\
Independently assembled gauge KKT residual & $4.45\!\times\!10^{-12}$ & $2.25\!\times\!10^{-11}$ \\
\bottomrule
\end{tabular}
\end{table}

For $K=H_{\rm returned}-H_{\rm native}$, the gauge check evaluates every coefficient $K_{jii}$ and $K_{jik}+K_{jki}$ for $i<k$, rather than testing sampled vectors. The certificate matrix is independently reassembled from the returned $P,H$. For $L=\left[\begin{smallmatrix}P\otimes I&G\\G^T&\tau I\end{smallmatrix}\right]$, the scaled semidefiniteness defect is $\max\{0,-\lambda_{\min}(L)\}/\max\{1,\|L\|_2\}$; we check both the solver's model $G$ and the $G$ reconstructed from the returned $H$. Small negative eigenvalues are recorded as numerical defects, not silently rounded to feasibility. These call times may overlap across workers and are diagnostic only.

\stopcontents[appendices]

\end{document}